\documentclass[reqno, 11pt, a4paper]{amsart}
\usepackage{latexsym,ifthen,xspace}
\usepackage{amsmath,amssymb,amsthm}
\usepackage{enumerate, calc}
\usepackage[colorlinks=true, pdfstartview=FitV, linkcolor=blue,
  citecolor=blue, urlcolor=blue, pagebackref=false]{hyperref}
\usepackage[text={33pc,605pt},centering]{geometry}    
\usepackage{fullpage}
\usepackage{graphicx}
\newtheorem{theorem}{Theorem}[section]
\newtheorem{lemma}[theorem]{Lemma}
\newtheorem{prop}[theorem]{Proposition}
\newtheorem{assumption}[theorem]{Assumption}
\newtheorem{corro}[theorem]{Corollary}

\theoremstyle{definition}
\newtheorem{definition}[theorem]{Definition}

\theoremstyle{remark}
\newtheorem{remark}[theorem]{Remark}

\numberwithin{equation}{section}

\DeclareMathAlphabet{\mathsl}{OT1}{cmss}{m}{sl}
\SetMathAlphabet{\mathsl}{bold}{OT1}{cmss}{bx}{sl}

\newcommand{\step}[1]{\noindent \textit{Step} #1.}
\newcommand{\substep}[1]{\noindent \textit{Substep} #1.}

\newcommand{\om}{\ensuremath{\omega}}

\newcommand{\ve}{\ensuremath{\varepsilon}}

\newcommand{\Si}{\ensuremath{\Sigma}}

\newcommand{\cF}{\ensuremath{\mathcal F}}

\newcommand{\cL}{\ensuremath{\mathcal L}}

\newcommand{\cX}{\ensuremath{\mathcal X}}
\newcommand{\cY}{\ensuremath{\mathcal Y}}

\newcommand{\bbE}{\ensuremath{\mathbb E}}

\newcommand{\bbN}{\ensuremath{\mathbb N}} 
 
\newcommand{\bbP}{\ensuremath{\mathbb P}} 
 
\newcommand{\bbR}{\ensuremath{\mathbb R}}

\newcommand{\bbZ}{\ensuremath{\mathbb Z}} 
\DeclareMathOperator{\mean}{\mathbb{E}}
\DeclareMathOperator{\Mean}{\mathrm{E}}
\DeclareMathOperator{\prob}{\mathbb{P}}
\DeclareMathOperator{\Prob}{\mathrm{P}}

\DeclareMathOperator{\var}{\mathrm{var}}

\newcommand{\ldef}{\ensuremath{\mathrel{\mathop:}=}}

\def\indicator{{\mathchoice {1\mskip-4mu\mathrm l}%
{1\mskip-4mu\mathrm l}{1\mskip-4.5mu\mathrm l}%
{1\mskip-5mu\mathrm l}}}

\def\ue{\underline e}
\def\oe{\overline e}
\def\e{\varepsilon}

\begin{document}

\title[ Berry-Essen Theorem for RCM]{Berry-Esseen Theorem and Quantitative homogenization for the Random Conductance Model with degenerate Conductances}


\author{Sebastian Andres}
\address{University of Cambridge}
\curraddr{Centre for Mathematical Sciences \\
Wilberforce Road, Cambridge
CB3 0WB}
\email{s.andres@statslab.cam.ac.uk}
\thanks{}

 \author{Stefan Neukamm}
 \address{Technische Universit\"at Dresden}
 \curraddr{Faculty of Mathematics,
Zellescher Weg 12-14,
01069 Dresden}
 \email{stefan.neukamm@tu-dresden.de}
 \thanks{}



\date{\today}

\dedicatory{}

\begin{abstract}{
 We study the random conductance model on the lattice $\bbZ^d$, i.e.\ we consider a linear, finite-difference, divergence-form operator with random coefficients and the associated random walk under random conductances. We allow the conductances to be unbounded and degenerate elliptic, but they need to  satisfy a strong moment condition and a quantified ergodicity assumption in form of a  spectral gap estimate.  As a main result we obtain in dimension $d\geq 3$ quantitative central limit theorems for the random walk in form of a Berry-Esseen estimate with speed $t^{-\frac 1 5+\varepsilon}$ for $d\geq 4$ and $t^{-\frac{1}{10}+\varepsilon}$ for $d=3$. Additionally, in the uniformly elliptic case in low dimensions $d=2,3$ we improve the rate in a quantitative Berry-Esseen theorem recently obtained by Mourrat. As a central analytic ingredient, for $d\geq 3$ we establish near-optimal decay estimates on the semigroup associated with the environment process. These estimates also play a central role in quantitative stochastic homogenization and extend some recent results by Gloria, Otto and the second author to the degenerate elliptic case. }
\end{abstract}

\maketitle

\tableofcontents

\section{Introduction}
\newcommand{\bomega}{{\boldsymbol\omega}}
Stochastic homogenization of  elliptic equations in divergence form with random coefficients started from the pioneering works of Kozlov \cite{Ko79} and Papanicolaou-Varadhan \cite{PV81}. They established a \emph{qualitative} homogenization result, which (adjusted to a discrete setting) can be rephrased as follows. The unique bounded solution $u_\varepsilon$ to the elliptic finite difference equation
\begin{equation}\label{intro:eq:1}
  \nabla^*{\bomega}\nabla u_\varepsilon=\varepsilon^2f(\varepsilon \, \cdot)\qquad\text{on }\bbZ^d
\end{equation}
with $\bomega$ describing stationary and  ergodic, uniformly elliptic, random coefficients, and $f$ an appropriate right-hand side, e.g.\ $f\in C_c(\bbR^d)$ with zero mean, converges after a rescaling to the solution $u_0$ of the deterministic, elliptic equation
\begin{equation*}
  -\nabla\cdot \bomega_{\hom}\nabla u_0=f\qquad\text{on }\bbR^d,
\end{equation*}
where $\bomega_{\hom}$ denotes a deterministic coefficient matrix, the so-called \textit{homogenized coefficients}. Quantitative stochastic homogenization is concerned with finding the rate of convergence of $u_\varepsilon$ towards $u_0$. The first result in this direction was obtained by Yurinskii \cite{Yu86} and relied on probabilistic arguments. Recently, Gloria and Otto in \cite{GO11,GO12} and together with the second author in \cite{GNO15} obtained the optimal scaling of the error for the discrete, uniformly elliptic case by combining input from 
elliptic and parabolic regularity theory with input from
 statistical mechanics, in particular a spectral gap inequality used to quantify ergodicity. Thereafter, an increasing interest emerged in quantitative stochastic homogenization, e.g. see \cite{mourrat2011, GO14, armstrong2014, GO15, BMN17, armstrong2016, armstrong2016b, armstrong2016c,AKM17, FischerOtto2015, GNOreg, BellaGiuntiOtto2015, BellaFFOtto2016}. 
\medskip

Closely related to the topic of homogenization in PDE theory is the problem of deriving \emph{invariance principles} or \emph{functional central limit theorems} for the so-called \emph{random conductance model} in probability theory, which refers to the random walk $X$ in random environment generated by the operator in \eqref{intro:eq:1}. 
Roughly speaking, an invariance principle states that the scaling limit of $X$ converges to a Brownian motion with a non-random covariance matrix $\Sigma^2$ only depending on the law of the conductances, see Theorem~\ref{thm:ip} below. In particular, the covariance matrix of the limiting process and the homogenized coefficients are related by the identity $\Sigma^2=2\bomega_{\hom}$.
Such invariance principles are subject of very active research since more than a decade, see the surveys \cite{ Bi11, Ku14} and references therein.  
\medskip

The goal of this paper is to extend the quantitative theories of stochastic homogenization and invariance principles to the case of non-uniformly elliptic conductances. In particular, we are interested in moment bounds and decay estimates for the so-called corrector problem of homogenization  and a \emph{Berry-Esseen theorem}. Regarding the latter, a  first inspiring result in this direction has been obtained by Mourrat \cite{Mo12} for uniformly elliptic
 i.i.d.\ conductances. In the uniformly elliptic case, we improve this result (in terms of the convergence rate) in dimension $d=2,3$, and we extend this result to degenerate and correlated environments in dimension $d\geq 3$ with a weaker rate of convergence in dimension $d=3$. 
Our analysis invokes input from the quantitative theory of stochastic homogenization as developed; in particular, we extend some key estimates obtained in \cite{GNO15} under the assumption of uniform ellipticity to degenerate elliptic operators under moment conditions.

\subsection{The model} Let $d\geq 2$. We study the nearest-neighbour \textit{random conductance model} on the $d$-dimen\-sional Euclidean lattice $(\bbZ^d, E_d)$, where $E_d:=\{ e=\{x,x\pm e_i\}:\,x\in\bbZ^d,i=1,\ldots,d\,\}$ denotes the set of non-oriented nearest neighbour edges and $\{e_1, \ldots ,e_d\}$ the canonical basis in $\bbR^d$. We endow  the graph with  positive random weights, which we describe by a family $\omega=\{\om(e), \, e\in E_d\}\in\Omega:=(0,\infty)^{E_d}$. We refer to $\om(e)$ as the \emph{conductance} of an edge $e\in E_d$.  To simplify notation, for any $x,y \in \bbZ^d$, we set
\begin{align*}
  \om(x,y) \;=\; \om(y,x) \;\ldef\; \om(\{x,y\}),
  \quad \forall\, \{x,y\} \in E_d,
  \mspace{28mu}
  \om(x,y) \;\ldef\; 0, \quad \forall\, \{x,y\} \not\in E_d,
\end{align*}
and define the matrix field $\bomega:\bbZ^d\to\bbR^{d\times d}$ by
\begin{equation*}
  \bomega(x)=\operatorname{diag}(\omega(x,x+e_1),\ldots,\omega(x,x+e_d)).
\end{equation*}
Henceforth, we consider random conductances that are distributed according to a probability measure $\prob$ on $\Omega$, equipped with the  $\sigma$-algebra $\cF\ldef \mathcal{B}((0,\infty))^{\otimes E_d}$. We write $\mean$ for the expectation operator with respect to $\prob$. The measure space $(\Omega, \cF)$ is naturally equipped with a \textit{group of space shifts} $\big\{\tau_x : x \in \bbZ^d\big\}$, which act on $\Omega$ as 
\begin{equation*}
  \tau_x\omega(\cdot)\;\ldef \; \omega(\cdot+x),
\end{equation*}
where the \textit{shift by $x$} of an edge $e=\{\ue,\oe\}\in E_d$ is defined as $e+x:=\{\ue+x,\oe+x\}$. 
\medskip

The \emph{random conductance model} is defined as follows. For any fixed realisation $\om$ it is a \emph{reversible} continuous time Markov chain, $X = \{X_t\!: t \geq 0 \}$, on $\bbZ^d$ with generator $\cL^{\om}$ acting on bounded functions $f\!: \bbZ^d \to \bbR$ as
\begin{align} \label{eq:defL}
\big(\cL^{\om} f)(x)
  \;=\; 
  \sum_{y \in \bbZ^d} \om(x,y) \, \big(f(y) - f(x)\big).
\end{align}
With the help of the discrete gradient $\nabla$ and its adjoint $\nabla^*$ (see Section~\ref{sec:notation} below), we can represent the generator in the compact form $\cL^{\om}=-\nabla^*\bomega\nabla$, which highlights the fact that $\cL^{\om}$ is a (finite difference) second order operator in divergence form.
We denote by $\Prob_x^\om$ the law of the process starting at the vertex $x \in \bbZ^d$ and by $\Mean_x^\om$ the corresponding expectation .  This random walk waits at $x$ an exponential time with mean $1/ \mu^\om(x)$ with  $\mu^\om(x) \ldef \sum_{y \in \bbZ^d} \om(x,y)$  and chooses the vertex $y$ for its next position with probability $\om(x,y) / \mu^{\om}(x)$.  Since the law of the waiting times does depend on the location, $X$ is also called the \emph{variable speed random walk} (VSRW).
A general assumption required in the context of invariance principles is stationarity and ergodicity of the environment.
%
%
 \begin{definition} [Stationarity and ergodicity] \label{ass:ergodic}
   We say the measure $\prob$ is stationary with respect to translations of $\bbZ^d$, if $\prob \circ\, \tau_x^{-1} \!= \prob\,$ for all $x \in \bbZ^d$. We say $\prob$ is ergodic, if $\prob[A] \in \{0,1\}\,$ for any $A \in \cF$ such that $\tau_x(A) = A\,$ for all $x \in \bbZ^d$.
 \end{definition}
First results in this context are \emph{annealed} (or \emph{averaged}) functional central limit theorems that yield the convergence of the rescaled random walk under the annealed measure $\prob_0$ defined by $\prob_0 [\cdot] := \int_\Omega P_0^\om [\cdot] \, d\prob(\om)$. This has been established in \cite{dMFGW89} (cf.\ also \cite{KV86}) for general ergodic environments under the assumption of strict positivity and a first moment condition
\begin{equation}\label{ass:moment}
  \prob\!\big[ 0 < \om(e) < \infty\big] \; = \; 1\qquad\text{and}\qquad\mean[\omega(e)] \; < \; \infty, \qquad \text{for all $e\in E^d$}.
\end{equation}
It is of particular interest to understand the finer question whether an invariance principle also holds for $\prob$-a.e.\ $\om$, that is in a \textit{quenched} form.  In \cite{ABDH13}   the quenched invariance principle has been shown for general  i.i.d.\ conductances.  However, in the case of a general ergodic environment, due to a trapping phenomenon, it is clear that some \textit{moment conditions} (stronger than the one in \eqref{ass:moment})  are needed. Indeed, Barlow, Burdzy and Timar  \cite{BBT16}  give an example on $\bbZ^2$ which satisfies a weak moment condition, but not a quenched invariance principle. To formulate the moment condition used in our paper, we set for any $ p,q \in [1,\infty]$, 
  \begin{align}\label{eq:moment_condition}
 M(p,q) \; \ldef \; \sum_{i=1}^d  \Big( \mean\!\big[ \om(0,e_i)^p\big] + \mean\big[\om(0,e_i)^{-q} \big] \Big) \in (0,\infty].
  \end{align}
  Note that in the special case, when $M(p,q)<\infty$ for $p=q=\infty$, we obtain the uniformly elliptic case, i.e. $\prob[1/c \leq \om(e) \leq c ] = 1$ for some $c > 0$ and the generator $\cL^\om=-\nabla^*\bomega\nabla$ defines a uniformly elliptic discrete operator.
  Recently, the following quenched invariance principle has been proven for random walks under ergodic conductances.
  \begin{theorem}[\cite{ADS15}]\label{thm:ip}
    Suppose that  $\prob$ is stationary and ergodic, and that the moment condition $M(p,q)<\infty$ holds for
    exponents $p, q \in  (1, \infty]$  satisfying $p^{-1} + q^{-1} < 2/d$.  
For $n\in \bbN$, define $X_t^{(n)} \ldef \frac{1}{n} X_{n^2 t}$, $t\geq 0$.  Then, for $\prob$-a.e.\ $\om$,   $X^{(n)}$ converges (under $\Prob_{\!0}^\om$) in law towards a Brownian motion on $\bbR^d$  with a deterministic non-degenerate covariance matrix $\Si^2$.
  \end{theorem}
  For further invariance principles in the setting of random degenerate conductances we refer to \cite{ABDH13, Bi11, Ku14} and references therein; for recent results on (qualitative) stochastic homogenization of elliptic operators in divergence form with degenerate coefficients, see \cite{LNO15, NSS17, BellaFO2016, FlegelHeidaSlowik2017}.

\subsection{Main results}
In this paper our main concern is to establish a quantitative central limit theorem for the random walk $X$. For definiteness, let $\xi\in \mathbb{R}^d$ be fixed and set $$\sigma_\xi^2:=\xi \cdot \Sigma^2 \xi,$$
where $\Sigma^2$ still denotes the covariance matrix in Theorem~\ref{thm:ip}.  Then, the invariance principle of Theorem~\ref{thm:ip} yields for $\bbP$-a.e.\ $\om$,
\begin{align} \label{eq:conv_distFct}
\lim_{t\to \infty} P_0^\om\big[\xi \cdot X_t \leq \sigma_\xi x \sqrt{t} \ \big] = \Phi(x),
\end{align}
 where $\Phi(x):=(2\pi)^{-1/2} \int_{-\infty}^x e^{-u^2/2} \, du$ denotes the distribution function of the standard normal distribution. 
Our goal is to quantify the speed of convergence in \eqref{eq:conv_distFct} for $d\geq 3$ by means of a \textit{Berry-Esseen theorem}. For a general ergodic environment the speed of convergence can be arbitrarily slow, since ergodic environments may have very weak mixing properties. Therefore it is necessary to quantify the assumption of ergodicity. For this purpose, following the approach of \cite{GO11,GO12, GNO15}, we assume that $\prob$ satisfies a \emph{spectral gap estimate} with respect to a Glauber dynamics on the field of conductances.
 \begin{assumption} [Spectral gap] \label{ass:sg}
   Suppose $\prob$ is stationary, and  assume that there exists $\rho>0$ such that
   \begin{align} \tag{SG}\label{SG}
     \mean\big[(u-\mean[u])^2\big] \leq \frac 1 \rho \sum_{e\in E_d} \mean \Big[ \big( \partial_{e} u \big)^2 \Big],
   \end{align}
   for any  $u \in L^2(\Omega)$. Here, the \emph{vertical derivative} $\partial_e u$ is defined as
   \begin{equation*}
     \partial_{e} u(\omega):=\limsup\limits_{h\to 0}\frac{u(\omega+h\delta_e)-u(\omega)}{h},
   \end{equation*}
   where $\delta_e:E_d\to\{0,1\}$ stands for the Dirac function satisfying $\delta_e(e)=1$ and $\delta_e(e')=0$ if $e'\not= e$. \end{assumption}

 \begin{remark}
  (i) There is a certain freedom in the choice of the derivative that appears on the right-hand side in \eqref{SG}. In \cite{GNO15} the following vertical derivative is considered,
     \begin{equation}\label{alternative:vd}
       \partial_e u=u-\mean[u|\mathcal F_e],
     \end{equation}
     where $\mean[\cdot|\mathcal F_e]$ denotes the conditional expectation w.r.t.\ the $\sigma$-algebra $\mathcal F_e=\sigma(\omega(e'):\,e'\neq e)$. In this form the (SG) turns into an \textit{Efron-Stein inequality}, which holds for any environment generated by i.i.d.\ random variables having second moments, see e.g.\ \cite[Lemma~7]{GNOlong}. All results in our paper extend to this version of \eqref{SG}. Since \eqref{alternative:vd} does not satisfy a Leibniz rule, using  \eqref{alternative:vd} instead of the classical partial derivative appearing in \eqref{SG} leads to not very enlightening technicalities in various calculations.
   \smallskip
   
   (ii) Any stationary environment satisfying  Assumption~\ref{ass:sg} is ergodic, see \cite[Corollary~6]{GNOlong}. In a sense  Assumption~\ref{ass:sg} can be interpreted as a quantified  version of ergodicity as it implies an optimal variance decay for the semigroup associated with the ``process of the environment as seen from the particle'' induced by the simple random walk on $\bbZ^d$, cf.\ \cite[Proposition~1 and Remark~5]{GNOlong}.

\smallskip   
   
   (iii) Under Assumption~\ref{ass:sg} we have the following $p$-version of the spectral gap estimate. For $p\geq 1$ and any  $u \in L^{2p}(\Omega)$ with $\mean[u]=0$,
     \begin{align} \label{eq:sg_p}
       \mean\big[u^{2p}\big] \leq c(p,\rho)
       &\mean \bigg[ \Big( \sum_{e\in E_d} \big( \partial_e u \big)^2
       \Big)^{\! p} \bigg],
     \end{align}
     which basically follows by applying \eqref{SG} to the function $|u|^{p}$, see \cite[Lemma 11]{GNO15}.
 \end{remark}

In addition to Assumption~\ref{ass:sg} we need to assume stronger moment conditions than in Theorem~\ref{thm:ip}. We do not keep track of the precise lower bounds for $p$ and $q$ in the moment condition $M(p,q)<\infty$ that we require in our analysis, since our approach is not optimal in that direction. Recall that in view of the counterexample in \cite{BBT16} a moment condition is necessary already for the (non-quantitative) invariance principle to hold.
\medskip

Our first result is an annealed Berry-Esseen theorem in dimension $d\geq 3$ with speed $t^{-\frac 1 5 +\varepsilon}$ for $d\geq 4$ and $t^{-\frac{1}{10}+\e}$ for $d=3$, as well as a quenched Berry-Esseen theorem with the same speed but in an integrated form.

\begin{theorem}[Berry-Esseen theorem] \label{thm:main}
  Let $d \geq 3$ and suppose that Assumption~\ref{ass:sg} holds. For any $\varepsilon>0$ there exist exponents $p,q\in (1,\infty)$ (only depending on $d$, $\rho$, and $\varepsilon$) such that under the moment condition $M(p,q)<\infty$ the following hold.
 \begin{enumerate}
 \item[(i)]  There exists a constant $c=c(d,\rho, \varepsilon, M(p,q))$ such that for all $t\geq 0$, 
   \begin{align*}
     \sup_{x \in \bbR} \left|  \prob_0\big[\xi \cdot X_t \leq \sigma_\xi x \sqrt{t} \ \big] - \Phi(x)\right|  \; \leq \;  \begin{cases}
   \; c \, t^{- \frac 1 {10} + \varepsilon} & \text{if $d=3$,} \\
    \; c \, t^{-\frac 1 5 + \varepsilon} & \text{if $d\geq 4$.}
  \end{cases}
  \end{align*}
\item [(ii)] There exists a random variable $\cX=\cX(d,\rho,\varepsilon, M(p,q))\in L^1(\prob)$ such that  for $\prob$-a.e.\ $\om$, 
    \begin{align*}
  \int_0^\infty  \Big(  \sup_{x \in \bbR} \left|  P_0^\om\big[\xi \cdot X_t \leq \sigma_\xi x \sqrt{t} \ \big] - \Phi(x)\right| \Big)^{\! 5} \, (t+1)^{-\frac 1 2-\varepsilon}\, dt  \; \leq \; \cX(\om) \, < \; \infty \qquad \text{if $d=3$,}
  \intertext{and}
   \int_0^\infty  \Big(  \sup_{x \in \bbR} \left|  P_0^\om\big[\xi \cdot X_t \leq \sigma_\xi x \sqrt{t} \ \big] - \Phi(x)\right| \Big)^{\! 5} \, (t+1)^{-\varepsilon}\, dt  \; \leq \; \cX(\om) \, < \; \infty \qquad \text{if $d\geq 4$.}
  \end{align*}
 \end{enumerate}  
\end{theorem}

  In the case of uniformly elliptic i.i.d.\ conductances an annealed
  Berry-Esseen theorem as in (i) has been proven in
  \cite{Mo12} for arbitrary dimension $d\geq 1$ with rate $t^{-1/10}$ (plus logarithmic corrections) for $d=2$, and rate $t^{-1/5}$ for $d\geq 3$ (with some logarithmic corrections for $d=3$). Theorem~\ref{thm:main} extends this result to
  \textit{unbounded} and \textit{correlated} random conductances. In Section~\ref{sec:examples} below we discuss some relevant examples
  linked to Ginzburg-Landau interface models that naturally yield correlated conductances. To our knowledge
  (ii) is the first quenched Berry-Esseen-type result
  for the random conductance model. 
\medskip

Let us anticipate that the general strategy of our proof is the same as the one in \cite{Mo12}. However, there is a genuine difference between the uniformly elliptic case treated there and the degenerate elliptic case considered here. In the uniformly elliptic case in \cite{Mo12} the comparability (on the level resolvents) of the simple random walk and the random walk in the random environment is exploited (see \cite[Proof of Theorem~5.1]{Mo12}), which comes in hand in studying the variance decay of the associated semigroup. In our case, no such principle is available. 
  Instead, following ideas in \cite{GNO15}, we first establish a semigroup estimate that invokes the \textit{gradient} of the heat kernel associated with the degenerate elliptic operator $\nabla^*\bomega\nabla$.
  Another difference to \cite{Mo12} is that we introduce a representation in divergence form for the carr\'e du champ applied to harmonic coordinates. It invokes the so called extended corrector $(\phi_i,\sigma_i)$, which has been recently introduced in the random case in \cite{GNOreg}. Our refined argument allows to improve the rates obtained in \cite{Mo12} in the uniformly elliptic case in low dimensions $d=2,3$:

\begin{theorem}[Improved Berry-Esseen theorem in the uniformly elliptic case] \label{thm:mainunif}
  Let $d \geq 2$. Suppose that Assumption~\ref{ass:sg} and uniform ellipticity hold, i.e.~$M(p,q)<\infty$ for $p=q=\infty$. Then there exists a constant $c=c(d,\rho,M(\infty,\infty))$ such that for all $t\geq 0$, 
   \begin{align*}
     \sup_{x \in \bbR} \left|  \prob_0\big[\xi \cdot X_t \leq \sigma_\xi x \sqrt{t} \ \big] - \Phi(x)\right|  \; \leq \;  \begin{cases}
       \; c \, \left(\frac{\log(t+1)}{t+1}\right)^{\!\frac 1 5} & \text{if $d=2$,} \\
       \; c \, (t+1)^{-\frac 1 5} & \text{if $d\geq 3$.}
     \end{cases}
  \end{align*}
\end{theorem}
The above result can be obtained by modifying the proof of Theorem~\ref{thm:main} and is based on recent estimates that are known to hold in the uniformly elliptic case. We describe the modifications in Appendix~\ref{app:mainunif} for the reader's convenience.

\medskip

The proof of Theorem~\ref{thm:main} is given in Section~\ref{S:be} below. In the following we explain the general strategy. The most classical approach to show an invariance principle is to decompose the process $X$ into a martingale part and a remainder (cf.\ e.g.\ \cite{KV86}). It turns out that for the invariance principle the remainder is negligible, and thus the scaling limit of $X$ is the same as the one for the martingale part. The latter can be analysed by martingale theory as for instance the well-established  Lindeberg-Feller functional central limit theorem or  Helland's martingale convergence theorem (see \cite{He82}).  In the proof of Theorem~\ref{thm:main} we follow the same strategy. Yet, since we seek for an estimate on the rate of convergence, at various places we need to replace qualitative arguments by estimates. 
A key result in this procedure is the following decay estimate for the semigroup  $(P_t)_{t\geq 0}$ defined by
\begin{equation*}
  P_t:L^\infty(\Omega)\to L^\infty(\Omega),\qquad (P_t u)(\omega):=\sum_{y\in\bbZ^d}p^\omega(t,0,y)\, u(\tau_y\omega),
\end{equation*}
where $p^{\om}(t,x,y) \ldef  \Prob_{x}^{\om}[X_t = y]$
denotes the transition densities or \emph{heat kernel} associated with $\cL^\om$. Throughout the paper we will often write  $p(t,y)\ldef p^\om(t,0,y)$ in short and use the fact that $ p^{\tau_z \om}(t, x, y) \;=\; p^{\om}(t, x+z, y+z)$ due to the definition of the space shifts.
The semigroup $(P_t)_{t\geq 0}$ can be interpreted as the transition semigroup of the process $(\tau_{X_t}\om)_{t\geq 0}$, which is known as the \textit{process of the environment as seen from the particle}. It is a contraction and generated by the (degenerate) elliptic operator $-D^*\bomega(0) D$, where $D$ and $D^*$ denote the \textit{horizontal derivative} and its adjoint (see Section~\ref{sec:notation} for the precise definition). Our key estimate is the following.

\begin{theorem}[Semigroup decay] \label{T:decay}
 Let $d\geq 3$ and suppose that Assumption~\ref{ass:sg} holds. Let $\varepsilon\in(0,1)$ and $n\geq\frac{d}{2\varepsilon}$. Then there exist $p, q \in (1,\infty)$ (only depending on $d,\e,n$) such that under the moment condition $M(p,q)<\infty$ the following holds. For any $F\in L^{8n}(\Omega,\mathbb R^d)$ and all $t\geq 0$ we have
  \begin{align}\label{T:decay:eq}
    \mean \big[ (P_t D^*F)^{2n} \big]^{\frac{1}{2n}}\; \leq c \;
    (1+t)^{-(\frac d 4 + \frac 1 2)+\varepsilon}\sum_{e\in E_d}\mean\big[|\partial_e F|^{8n}\big]^{\frac{1}{8n}},
  \end{align}
  where $c=c(d,\rho,n,\e,M(p,q))$.
\end{theorem}
\begin{remark}[Comparison to the uniformly elliptic case]
  For uniformly elliptic conductances, in \cite{GNO15}, Gloria, Otto and the second author established the decay estimate
  \begin{align}\label{eq:GNO15}
    \mean \big[ (P_t D^*F)^{2n} \big]^{\frac{1}{2n}} \; \leq \; c \,
    (1+t)^{-(\frac d 4 + \frac 1 2)}\sum_{e\in E_d}\mean\big[|\partial_e F|^{2n}\big]^{\frac{1}{2n}},
  \end{align}
  and deduced various estimates in stochastic homogenization based on this decay estimate. The estimate \eqref{eq:GNO15} is optimal in terms of the scaling in $t$ and in terms of the exponent of the norm on the right-hand side. 
In the degenerate elliptic case, we do not expect to get the same scaling. However, under sufficiently strong moment conditions, our estimate shows that we can get arbitrarily close to the scaling of the uniformly elliptic case. The argument in \cite{GNO15} crucially relies on a deterministic parabolic regularity estimate for $\nabla p^\om$ with optimal scaling in $t$. In the degenerate elliptic case this estimate is not valid. We replace it by a non-deterministic estimate on $\nabla p^\om$ with near-optimal scaling, see Proposition~\ref{P:HK} below, which we combine with an interpolation argument that exploits the contraction property of the semigroup. The latter is the reason for the exponent $8n$ in our estimate. 
\end{remark}
\begin{remark}
  Recently, a similar estimate has been obtained in \cite{GM16} for
  conductances uniformly bounded from above satisfying a moment
  condition on $\om(e)^{-1}$. The proofs of both Theorem~\ref{T:decay}
  and the results in \cite{GM16} follow the strategy in
  \cite{GNO15}. However, in \cite{GM16} the required on-diagonal
  estimate on the heat kernel is derived from the anchored Nash
  inequality established in \cite{MO16}, which makes a uniform upper
  ellipticity necessary, while in our setting the analogue heat kernel
  bound in Lemma~\ref{L:ondiagHK} can be deduced from the upper off-diagonal heat kernel estimates in
  \cite{ADS16a, ADS17}.  For the opposite case, i.e.\ conductances bounded
  from below having quadratic moments, we refer to \cite{dBM15}, where
  the simpler situation $P_tf$ (instead of $P_tD^*F$) is studied.
\end{remark}
We will use Theorem~\ref{T:decay} to quantify the convergence of the martingale part in the decomposition of the process $X$ mentioned above, which relies on \textit{harmonic coordinates} 
$\Psi: \Omega \times \bbZ^d \rightarrow \bbR^d$ defined as
\begin{equation}\label{eq:harm_coord}
  \Psi=(\psi_1,\ldots,\psi_d),\qquad \psi_i(\om,x)\; \ldef \; x_i+\phi_i(\om,x)-\phi_i(\om,0), \qquad i=1,\ldots,d,
\end{equation}
where $\phi_i$ denotes the \textit{corrector} from stochastic homogenization defined in Proposition~\ref{C:corrector} below. Roughly speaking, $\phi_i$ is a sublinearly growing solution to the equation $\nabla^*\omega(\nabla\phi_i+e_i)$. The corrector $\phi_i$ is a fundamental object in the qualitative and quantitative theory of stochastic homogenization, see  e.g.\ the seminal work by Papanicolaou-Varadhan \cite{PV81} or  \cite{GNO15} for quantitative results. In particular, the covariance matrix $\Sigma^2$  of the limiting process in Theorem~\ref{thm:ip} may be represented in terms of $\phi_i$ or $\psi_i$, respectively,  as
%
%
\begin{equation}\label{eq:hom_coeff}
  \Sigma^2_{ij}\; = \; \mean\Big[ \sum_{y\in \bbZ^d}
 \omega(0,y) \, \psi_i(\om,y) \, \psi_j(\om, y) \Big],\qquad i,j=1,\ldots,d,
\end{equation}
see \cite[Proposition~2.5]{ADS15}. It is well-known that for $\prob$-a.e.\ $\om$,
\begin{enumerate}[(i)]
\item the process $M=\Psi(\om,X)$ is a martingale (see Corollary~\ref{cor:qV} below) and thus features an invariance principle;
\item the remainder $X-M$ vanishes in the scaling limit due to the sublinear growth of the corrector.
\end{enumerate}
For our purpose we need to quantify both the speed of convergence in (i) and the smallness of the remainder in (ii). For this reason we establish the existence of high moments of the corrector $\phi_i$ (and an additional flux corrector $\sigma_i$, which we explain below). In the following we say that a random variable $u$ is {\it stationary}, if $u(\omega,x+y)=u(\tau_y\omega,x)$ for all $x,y\in\bbZ^d$ and $\prob$-a.e. $\omega\in\Omega$.
\begin{prop}[Extended correctors and moment bounds]\label{C:corrector}
 Let $d\geq 3$ and suppose that Assumption~\ref{ass:sg} is satisfied. Then there exist $p,q\in(1,\infty)$ (only depending on $d$ and, if applicable, on the upcoming parameters $n$, $\theta$, $p_\theta$) such that under the moment condition $M(p,q)<\infty$ the following hold for $i=1,\ldots,d$.
  \begin{enumerate}[(a)]
  \item (Existence of a non-stationary extended corrector).
    There exist a (unique) random scalar field $\phi_i:\Omega\times\bbZ^d\to\bbR$ and a (unique) random matrix field $\sigma_{i}:\Omega\times\bbZ^d\to\bbR^{d\times d}$ with the following properties.
    \begin{itemize}
    \item[(a.1)] For $\prob$-a.e.\ $\omega$ we have
      \begin{subequations}
        \begin{align}\label{eq:corrector-equation1}
          \nabla^*\omega \, \big(\nabla\phi_i+e_i\big)&=0\qquad  \text{on }\bbZ^d,\\
          \label{eq:corrector-equation2}
          \nabla^*\sigma_i&=q_i \qquad \!\!\text{on }\bbZ^d,\\
          \label{eq:corrector-equation3}
          \nabla^*\nabla \sigma_{i}&=Sq_i \quad \text{on }\bbZ^d,
        \end{align}
      \end{subequations}
      where $q_i:\Omega\times\bbZ^d\to\bbR^d$ and $Sq_i:\Omega\times\bbZ^d\to\bbR^{d\times d}$ are defined by
      \begin{align*}
        q_i&:=\;\bomega\, \big(\nabla\phi_i+e_i\big)-\bomega_{\hom}e_i,\\
        Sq_i&:=\;(Sq_i)_{k\ell}=\nabla_k q_{i\ell}-\nabla_\ell q_{ik},
      \end{align*}
      and $\bomega_{\hom}\in\bbR^{d\times d}$ denotes the \emph{homogenized coefficient} matrix characterised by
      \begin{equation} \label{eq:def_omhom}
        \bomega_{\hom}e_i\;=\; \mean\left[\omega\, \big(\nabla\phi_i+e_i\big)\right],\qquad i=1,\ldots,d.
      \end{equation}
      (Above, the divergence $\nabla^*\sigma_i$ is defined as the vector with entries $(\nabla^*\sigma_i)_k=\sum_{\ell=1}^d\nabla^*_\ell\sigma_{ik\ell}$).
    \item[(a.2)] $\prob$-a.s. the fields satisfy $\phi_i(0)=0$, $\sigma_i(0)=0$, and $\sigma$ is skew-symmetric, i.e.\ $\sigma_{i\alpha\beta}=-\sigma_{i\beta\alpha}$.
    \item[(a.3)] The  gradient fields $\nabla\phi_i$ and $\nabla\sigma_{i}$ are stationary, have finite $2$nd moments, and vanishing expectation, i.e.
      \begin{equation*}
        \mean\big[\nabla\phi_i\big]\;=\;0, \qquad \mean\big[\nabla\sigma_{i}\big]\;=\;0, \qquad \mean\Big[\big|\nabla\phi_i\big|^{2}+ \big|\nabla\sigma_{i}\big|^{2}\Big] \; \leq \; c,
      \end{equation*}
      for some $c=c(d,\rho,M(p,q))$.
    \end{itemize}
  \item (Moment bound and stationary representation for $\phi_i$). Let $n\in\bbN$. Then there exists a random variable $\phi^0_i$ with $\mean[\phi^0_i]=0$ and 
    \begin{equation}\label{eq:momentboundphi}
      \mean\Big[ \big|\phi^0_i\big|^{2n}+\big|D\phi^0_i\big|^{2n}\Big]^{\frac{1}{2n}} \; \leq \; c,
    \end{equation}
    for some $c=(d,\rho,n,M(p,q))$ and we have the stationary representation
    \begin{equation*}
      \phi_i(\omega,x)=\phi_i^0(\tau_x\omega)-\phi_i^0(\omega).
    \end{equation*}
  \item (Sublinear growth of $\sigma$). Let $\theta$ satisfy
    \begin{equation*}
      \begin{array}{ll}
        \theta>\frac{1}{2} & \mbox{if } d=3,\\
        \theta>0& \mbox{if } d=4,\\
        \theta=0& \mbox{if } d\geq 5.
      \end{array}
    \end{equation*}
    Then there exists $p_\theta>2$ such that
    \begin{equation}\label{eq:sublinsigma}
      \mean\Big[\big|\sigma_i(x)\big|^{p_\theta}\Big]^{\frac{1}{p_\theta}}\; \leq \; c \, (|x|+1)^\theta,
    \end{equation}
    for some $c=c(d,\rho,\theta,p_\theta,M(p,q))$.
\end{enumerate}
\end{prop}
While the sublinearity of the corrector can be established for general ergodic environments satisfying the relatively weak assumptions in Theorem~\ref{thm:ip} (cf.\  \cite{ADS15}), the existence of high moments of $\phi_i$ and $\nabla\phi_i$ as in Proposition~\ref{C:corrector} only holds true under sufficiently strong mixing assumptions. In dimension $d=2$, even in the case of uniformly elliptic, i.i.d.\ conductances, the stationary version of the corrector does not exist. On the other hand, for $d\geq 3$, in the uniformly elliptic case and under sufficiently strong mixing assumptions, the stationary representations of $\phi_i$ \textit{and} $\sigma_i$ exist and satisfy (high) moment bounds. This has been first achieved on the level of $\phi_i$ in \cite{GO11} (see also \cite{GNO15} where bounds are obtained via a decay estimate similar to Theorem~\ref{T:decay}, and see \cite{BMN17} for moment bounds on $\sigma_i$ and $\nabla\sigma_i$). To our knowledge Proposition~\ref{C:corrector} is the first result on moment bounds in the degenerate elliptic setting. The dependence on the growth exponent on the dimension in statement (c) is non-optimal. In fact, we expect the statements to hold for $d\geq 3$ (with $\theta=0$ and arbitrary $p_\theta<\infty$) as in the uniformly elliptic case; see Remark~\ref{R:whydgeq4}.
\medskip

Using the moments of $\phi_i$ and ergodicity we obtain a sufficiently good control on  the remainder, see Proposition~\ref{prop:conv_corr} below. In order to quantify the convergence of the martingale part in (i),  we use a Berry-Esseen estimate for martingales in \cite{Ha88} (see Theorem~\ref{thm:be_hb} below). It requires to quantify the speed of convergence of $E_0^\om \Big[\Big| \frac{\langle \xi \cdot M \rangle_t} t -\sigma_\xi^2 \Big|^2 \Big]$ towards zero, where $\langle \xi \cdot M \rangle$ denotes the quadratic variation process of the martingale $\xi\cdot M$. 
For that purpose we establish the following result which also exploits the moment bounds on $(\phi_i,\sigma_i)$ obtained in Proposition~\ref{C:corrector}.
\begin{prop}\label{P:decay}
  Let $d\geq 3$, $\e>0$, and 
  suppose Assumptions~\ref{ass:sg} is satisfied. Then there exist $p, q \in (1,\infty)$ (only depending on $d$ and $\e$) such that under the moment condition $M(p,q)<\infty$ the following holds. For any direction $\xi\in\bbR^d$ with $|\xi|=1$, we denote by $\phi_\xi$ the corrector associated with $\xi$, i.e. $\phi_\xi=\sum_{i=1}^d\xi_i\phi_i$ with $\phi_i$ as in Proposition~\ref{C:corrector}, and by
  \begin{equation}\label{eq:harmonic-coordinate}
    \psi_\xi(\omega, x):=\xi\cdot x+\phi_\xi(\omega,x)-\phi_\xi(\omega,0)
  \end{equation}
  the associated harmonic coordinate. Consider
  \begin{equation*}
    g_\xi:\Omega\to\bbR,\qquad g_\xi(\omega)\;:= \; \Gamma^\omega(\psi_\xi(\omega,\cdot))(0),
  \end{equation*}
  where $\Gamma^\om$ denotes the op\'erateur carr\'e du champ associated with $\mathcal{L}^\om$, i.e. $$\Gamma^\om f(x) \; \ldef \; \left[\cL^\om f^2 - 2 f \cL^\om f \right](x).$$
  Then there exists a constant
  $c=c(d,\rho, \theta, M(p,q))$ such that
  \begin{align*}
    \mean\left[\Big(P_t \big(g_\xi-\mean[g_\xi]\big)\Big)^{\!2}\right]^{\frac12} \; \leq \; 
    \begin{cases} 
    c\, (t+1)^{-\frac{1}{4}+\e} & \text{if $d=3$,} \\
     c\, (t+1)^{-\frac{1}{2}+\e} & \text{if $d\geq 4$}.
    \end{cases}
  \end{align*}
\end{prop}
In the proof of the proposition we make use of the field $\sigma_i$ defined in Proposition~\ref{C:corrector}, which allows us to represent the quadratic variation of $M$ in divergence form, see Lemma~\ref{L:QV} below. Similarly as $\phi_i$, the field $\sigma_i$ is a classical object in periodic homogenization. Yet, in the stochastic case it has been utilised  just recently, see  e.g.\ \cite{GNOreg,BMN17}.

\begin{remark}
  The exponent $1/5$ in Theorem~\ref{thm:main} is non-optimal. However,  the decay rates in Theorem~\ref{T:decay} and Proposition~\ref{P:decay} (for $d\geq 4$) are optimal apart from the $\varepsilon$ which appears due to the degeneracy of the environment. The exponent $1/5$ then arises by the application of the general Berry-Esseen theorem for martingales (see Theorem~\ref{thm:be_hb} below with the choice $n=2$). In $d\geq 3$ it can be deduced from the moments bounds on the corrector in Proposition~\ref{C:corrector} that the  jumps of the martingale part are in $L^n(P_0^\om)$ for any $n\in \bbN$ (cf.\ Proposition~\ref{prop:estJ} below). In such a situation the results in \cite{Mo12a} show that the decay $E_0^\om \Big[\Big| \frac{\langle \xi\cdot M \rangle_t} t -\sigma_\xi^2 \Big|^2 \Big]^{1/5}$ is optimal (for the choice $n=2$ in Theorem~\ref{thm:be_hb}). Thus, a possibility to improve the exponent $1/5$ (within the above approach) is to apply the  Berry-Esseen theorem for martingales in Theorem~\ref{thm:be_hb} for larger values of $n$ which would require control on higher moments of $\big| \frac{\langle \xi\cdot M \rangle_t} t -\sigma_\xi^2 \big|$ (cf.\ the discussion in \cite[page 6]{Mo12}). An alternative PDE-approach towards an optimal result would be to estimate the difference  between the heterogeneous and homogenized parabolic Green's function. Very recently, Armstrong et al.~\cite[Theorem~9.11]{AKM17} obtained such an estimate for the equation on $\mathbb R^d$ in the case of uniformly elliptic coefficients and under a finite range of dependence assumption. The estimate suggests that the exponent in Theorem~\ref{thm:main} (in the uniformly elliptic, i.i.d.~case) can be improved to $1/2$ (up to a logarithmic correction for $d=2$). The rate $1/2$ is the best one can hope for, since it is the convergence rate for the simple random walk.
The argument in \cite{AKM17} relies on a large scale regularity theory for elliptic operators, which is not established in the degenerate case yet. For general non-uniformly elliptic and (possibly strongly) correlated coefficients, we expect the optimal rate of convergence to depend on the parameters $p$ and $q$ as well as on the mixing behaviour of the environment. We remark that the qualitative statement in form of a quenched invariance principle holds for general i.i.d.\ conductances (cf.\ \cite{ABDH13}) and is conjectured to hold for ergodic conductances under the moment condition $M(1,1)<\infty$, which is known to be necessary in the case of general ergodic environments (cf.\ \cite{BBT16}). 
\end{remark}

\smallskip

\subsection{Notation} \label{sec:notation} 
We finally introduce some further notation used in the paper. We write $c$ to denote a positive, finite constant which may change on each appearance. Random constants depending on $\om \in \Omega$ will be denoted by calligraphic letters such as $\cX$, $\cY$ etc.  Further, $\lesssim$ means $\leq$ up to a constant depending only on some quantities  specified in the particular context. 

The cardinality of a set $A \subset \bbZ^d$ will be denoted by $\# A$.
For $x=(x_1,\ldots, x_d)\in \bbR^d$ let $|x|=\sum_{i=1}^d |x_i|$. We denote by $B(R):=\{x\in \bbZ^d:|x|<R\}$ the ball with radius $R>0$ in $\bbZ^d$. We denote by $\ell^r(\bbZ^d)$ and $\ell^r(E_d)$ $(r\geq 1)$ the usual $\ell^r$ spaces for functions on $\bbZ^d$ and $E_d$, respectively. 
For any edge $e\in E_d$ we denote by $\ue,\oe\in\bbZ^d$ the unique vertices such that $e=\{\ue,\oe\}$ and $\oe-\ue\in\{e_1,\ldots,e_d\}$.
 For $f:\bbZ^d\to\bbR$ and $e\in E_d$ we define the \textit{discrete derivative} 
\begin{align*}
  \nabla f: E_d\rightarrow \bbR, \qquad 
  \nabla f(e)\; \ldef \; f(\oe)-f(\ue),
\end{align*}
and note that for $f,g:\bbZ^d\to\bbR$, the discrete product rule takes the form 
\begin{align} \label{eq:discr_prod}
 \nabla(fg)(e)=f(\oe)\, \nabla g(e)+g(\ue)\, \nabla f(e).
\end{align}
We define the \textit{discrete divergence} of a function $F:E_d\to\bbR$ by
\begin{align*}
\nabla^*F(x):=\sum_{e\in E_d\atop\,\oe=x}F(e)-\sum_{e\in E_d\atop\ue=x}F(e)=\sum_{i=1}^dF(\{x-e_i,x\})-F(\{x,x+e_i\}). 
\end{align*}
Since for all $f\in \ell^2(\bbZ^d)$ and $F \in \ell^2(E_d)$ we have
\begin{align} \label{eq:adjoint}
\langle \nabla f, F\rangle_{\ell^2(E_d)} = \langle f, \nabla^*F \rangle_{\ell^2(\bbZ^d)},
\end{align}
$\nabla^*$ can be seen as the adjoint of $\nabla$.  
Note that the generator $\cL^\om$ defined in \eqref{eq:defL} above is a finite-difference operator in divergence form as it can be rewritten as
  \begin{align*}
    \big(\cL^{\om} f)(x)
    \;=\;
   - \nabla^*(\omega\nabla f)(x).
  \end{align*}
 
We tacitly identify scalar functions defined on $E_d$ with vector-valued functions on $\bbZ^d$. In particular, given $F:E_d\to\bbR$ we write $F(x)\ldef \big( F(\{x,x+e_1\}), \ldots,F(\{x,x+e_d\})$; and given $f:\bbZ^d\to\bbR$ we write $\nabla f(x):=(\nabla_1 f(x),\ldots,\nabla_d f(x))$ with $\nabla_i f(x) \ldef f(x+e_i)-f(x)$. Likewise, for $f: \bbZ^d \times \bbZ^d\to\bbR$ we define the \emph{second mixed discrete derivatives} as 
\begin{align*}
  \nabla\nabla f: E_d\times E_d \rightarrow \bbR, \qquad \nabla\nabla f(e,e')\; \ldef \; f(\oe,\oe')-f(\ue,\oe')-f(\oe,\ue')+f(\ue,\ue'),  
\end{align*}
and denote by $\nabla\nabla f(x,y)=(\nabla_i\nabla_j f(x,y))_{i,j=1,\ldots d}$ the matrix valued function with entries $\nabla_i\nabla_j f(x,y):=\nabla\nabla f(\{x,x+e_i\},\{y,y+e_j\})$. Further, we will also use the abbreviation
\begin{align} \label{eq:def_normom}
\big|F(x)\big|^2_\bomega \; \ldef \; \sum_{i=1}^d \om(x,x+e_i) \, \big| F( \{x,x+e_i \})\big|^2, \qquad x\in \bbZ^d.
\end{align}

With any random variable $\zeta: \Omega \rightarrow \bbR$  we associate its $\prob$-\emph{stationary extension}
$\bar \zeta: \Omega \times \bbZ^d \rightarrow \bbR$ via $ \bar \zeta(\om,x) \ldef \zeta(\tau_x \om)$. Conversely, we say that a random field $\tilde \zeta: \Omega \times \bbZ^d \rightarrow \bbR$ is $\prob$-\emph{stationary} if there exists a random variable $\zeta$ with $\tilde \zeta=\bar \zeta$ $\prob$-a.s. For a random variable $\zeta: \Omega \rightarrow \bbR$ we define the \emph{horizontal derivative} $D\zeta$ via
\begin{align*}
D\zeta\big(\om) \; \ldef \; (D_1\zeta(\om), \ldots , D_d\zeta(\om)\big), \qquad D_i\zeta(\om)\; \ldef \; \zeta(\tau_{e_i} \om) -\zeta(\om).
\end{align*}
Its adjoint in $L^2(\Omega)$ denoted by $D^*:L^2(\Omega)^d \rightarrow L^2(\Omega)$ is defined by
\begin{align*}
D^*\zeta(\om)\; \ldef \; \sum_{i=1}^d D_i^* \zeta_i(\om), \qquad D_i^* \zeta_i(\om) \; \ldef \; \zeta_i(\tau_{-e_i} \om) -\zeta_i(\om).
\end{align*}
Note that we have $\overline{(D^*\bomega(0)D u)}(\omega,x)=\nabla^*\omega(x)\nabla \overline u(\omega,x)$.

\medskip
\subsection{Structure of the paper}
In Section~\ref{sec:hk} we establish the estimates on the gradient of the heat kernel and the Green's function needed in the proofs. Then Section~\ref{S:decay} is devoted to the proof of the semigroup decay stated in Theorem~\ref{T:decay} and Section~\ref{sec:corr} to the construction and the moment bounds of the extended corrector. The variance decay for the semigroup applied to the carr\'e du champ operator in Proposition~\ref{P:decay} is shown in Section~\ref{S:P:decay}. Then, the Berry-Esseen estimates in Theorem~\ref{thm:main} are proven in Section~\ref{S:be}. Finally, some relevant examples satisfying our assumptions are discussed in Section~\ref{sec:examples}.

\section{Gradient heat kernel and annealed Green's function estimates} \label{sec:hk}

In this section we establish regularity estimates for averages of the \textit{gradient of the heat-kernel} and the \textit{mixed second gradient of the elliptic Green's function}, which we require for the proofs in Section~\ref{S:decay}. We consider both \textit{annealed} estimates, where the average is taken w.r.t.\ the probability measure $\prob$, and spatially averaged estimates in weighted $\ell^2$-spaces with weight $m^{2\alpha}$ where
\begin{align} \label{def:m_alpha}
  m(t,x)\ldef \bigg( \frac{(|x|+1)^2}{t+1} +1 \bigg)^{\! 1/2},\qquad x\in \bbZ^d,\, t\geq 0.
\end{align}
In the case of uniformly elliptic and bounded conductances the upcoming estimates are well-known. The estimates that we obtain in the degenerate case are weaker in two ways.
\begin{itemize}
\item We only obtain near optimal estimates, in the sense that the decay rate deviates from the optimal decay rate in the uniformly elliptic case by a small parameter $\e>0$.
\item The estimates are random, in the sense that they hold up to a random constant, whose integrability is monitored by an exponent $n\geq 1$.
\item The parameters $\e$ and $n$ can be chosen arbitrarily close $0$ and $\infty$, respectively, provided we impose sufficiently strong moment conditions on the conductances.
\end{itemize}

Throughout this section we assume $d\geq 2$. We start with the following (spatially averaged) decay estimate for the \textit{gradient of the heat-kernel}, which is a key ingredient for the proof of Theorem~\ref{T:decay}. 
\begin{prop}\label{P:HK}
  Suppose that Assumption~\ref{ass:sg} holds. For any $\varepsilon\in(0,1)$, $n\geq 1$ and $\alpha \geq 0$ there exist $p, q \in (1,\infty)$ (only depending on $d,\e,n, \alpha$) such that under the moment condition $M(p,q)<\infty$ the following holds. There exists a family of random variables $(\mathcal Z_t)_{t\geq 0}$ with $\sup_{t\geq 0}\mean[|\mathcal Z_t|^{n}]\leq c$  for some $c=(d,\rho,\e,n,\alpha, M(p,q))$ such that $\prob$-a.s.
  \begin{equation*}
    \Bigg(\sum_{y\in \bbZ^d} m(t,y)^{2\alpha}\, \big|\nabla p(t,y)\big|^2 \Bigg)^{\!\frac{1}{2}} \; \leq \; \mathcal Z_t \, (t+1)^{-(\frac{d}{4}+\frac{1}{2})+\varepsilon}.
  \end{equation*}
\end{prop}
In the uniformly elliptic case, in particular in the special case of $\nabla^*\bomega\nabla=\nabla^*\nabla$, the estimate holds with $\varepsilon=0$ (which corresponds to the optimal decay in time), and with $\sup_{t\geq 0}\mathcal Z_t$ bounded by a deterministic constant,  see e.g.\ \cite[Theorem~3]{GNO15}. In the present situation the degeneracy of the conductances leads to a loss in the decay. As we will explain in Section~\ref{S:ondiag}, our proof of the estimate relies on an on-diagonal upper heat kernel bound. The result is then obtained by parabolic regularity arguments following \cite{GNO15}.
\medskip

From Proposition~\ref{P:HK} we deduce a couple of annealed estimates. 
\begin{corro}[Suboptimal annealed heat kernel estimate]\label{C:AHK}
  Suppose that Assumption~\ref{ass:sg} is satisfied. For any $\varepsilon\in(0,1)$, $n\geq 1$, and $\alpha\geq 0$, there exist $p, q \in (1,\infty)$ (only depending on $d,\e,n,\alpha$) such that under the moment condition $M(p,q)<\infty$ the following holds. There exists $c=c(d,\rho,\e,\alpha,M(p,q))$ such that for all $x\in\bbZ^d$ and all $t\geq 0$,
  \begin{eqnarray*}
    \mean\Big[\big|\nabla p(t,x,0)\big|^n\Big]^{\frac{1}{n}} &\leq&  c \, (t+1)^{-(\frac{d}{2}+\frac{1}{2})+\varepsilon+\frac{1}{2}\frac{n-1}{n}} \, m(t,x)^{-\alpha},\\
    \mean\Big[\big|\nabla\nabla p(t,x,0)\big|^n\Big]^{\frac{1}{n}} &\leq&  c \, (t+1)^{-(\frac{d}{2}+1)+\varepsilon+\frac{n-1}{n}} \, m(t,x)^{-\alpha}.
  \end{eqnarray*}
\end{corro}
The proofs of Corollary~\ref{C:AHK} and the Corollary~\ref{L:AG} below are presented in Section~\ref{S:annealedHK}. Next we establish an annealed estimate on the gradient and the mixed second derivative of the elliptic Green's function, which for $\prob$-a.e.\ $\omega$ and all $x,y\in\bbZ^d$ can be defined by the integral
\begin{equation}\label{HKtoG}
  \nabla G^\omega(x,0):=\int_0^\infty\nabla p^\omega(t,x,0)\,dt,\qquad   \nabla\nabla G^\omega(x,y):=\int_0^\infty\nabla\nabla p^\omega(t,x,y)\,dt.
\end{equation}
\begin{corro}[Suboptimal annealed Green's function estimate]\label{L:AG}
  Suppose that Assumption~\ref{ass:sg} holds and let $\varepsilon\in(0,1)$ and $n\geq 1$. There exist $p, q \in (1,\infty)$ (only depending on $d,\e,n$) such that under the moment condition $M(p,q)<\infty$ we have for all $x\in\bbZ^d$,
  \begin{align*}
    \mean\Big[\big|\nabla G(x,0)\big|^n\Big]^\frac{1}{n} &\leq \;   c\, (|x|+1)^{-(d-1)+\e+\frac{n-1}{n}},\\
    \mean\Big[\big|\nabla\nabla G(x,0)\big|^n\Big]^{\frac{1}{n}} &\leq  \;  c\, (|x|+1)^{-d+\varepsilon+2\frac{n-1}{n}},
  \end{align*}
  with $c=(d,\rho,\e,m,M(p,q))$.
\end{corro}

The decay exponent in Corollary~\ref{L:AG} is not optimal. For $n=1$ we miss the optimal decay by $\e$, and for large $n$ by the additional exponent $\frac{n-1}{n}$ and $2\frac{n-1}{n}$, respectively. As shown recently in \cite{MO15} in the uniformly elliptic case and under the assumption that $\mathbb P$ satisfies a logarithmic Sobolev inequality, the estimate holds for any $n\in\bbN$ with the optimal decay exponent $d-1$ and $d$, respectively. The argument in \cite{MO15} lifts the estimate for $n=1$ to higher moments by using the logarithmic Sobolev inequality and a \textit{deterministic} regularity estimate, which is not available in our degenerate setting. Our (simple, yet suboptimal) argument is as follows. For $n=1$ the estimate follows from Proposition~\ref{P:HK} exploiting stationarity. For $n\gg 1$ we obtain it by interpolating with a suboptimal estimate on high moments of $\nabla p$, that follows from Proposition~\ref{P:HK} as well.

\begin{remark}\label{R:whydgeq4}
  The non-optimality of the above estimate is the limiting factor that hinders us to improve the decay exponent in our main result, Theorem~\ref{thm:main}, in dimension $d=3$.  This becomes visible in the sublinear growth estimate for $\sigma$, see \eqref{eq:sublinsigma}. In the proof of this estimate we exploit that thanks to Corollary~\ref{L:AG} we have
  \begin{equation*}
   \sum_{x\in\bbZ^d}\bigg(\mean\Big[\big|\nabla\nabla G(x,0)\big|^n\Big]^{\frac{1}{n}} \bigg)^{\! s} \; < \; \infty
  \end{equation*}
for  exponents $0<n-2\ll1$ and $s>\frac{d}{d-1}$. With the optimal estimate for $\nabla\nabla G$ at hand, the above estimate would hold for any $s>1$, and we would obtain \eqref{eq:sublinsigma} with $\theta=0$ for any $d\geq 3$. Eventually, this would improve the decay rate in Theorem~\ref{thm:main} for $d=3$ to $\frac{1}{5}-\e$.
\end{remark}

\subsection{Gradient heat kernel estimate: Proof of Proposition~\ref{P:HK}}

In this section we prove Proposition~\ref{P:HK}. The starting point of the argument is an on-diagonal upper heat kernel bound.

\begin{lemma}[On-diagonal heat kernel estimate]\label{L:ondiagHK}
  Suppose that Assumption~\ref{ass:sg} holds. For any $n\in\bbN$ there exist $p, q \in (1,\infty)$ (only depending on $d,n$) such that under the moment condition $M(p,q)<\infty$ the following holds.  There exists a random variable $\mathcal Y\geq 1$ with $\mean\big[|\mathcal Y|^n \big]\leq c$ for some $c=c(d,\rho,n,M(p,q))$ such that $\prob$-a.s.\ for all $t\geq 0$,
  \begin{equation}\label{odHK}
    \sum_{y\in\bbZ^d}p(t,y)^2 \; \leq \; \mathcal Y \, (t+1)^{-\frac d 2}.
  \end{equation}
  
\end{lemma}
For the proof see Section~\ref{S:ondiag} below. Next, we introduce the stationary weights
  \begin{align} \label{eq:def_mu}
  \mu^\omega(x):=\sum_{y\in\bbZ^d}\omega(x,y),\qquad \nu^\omega(x):=\sum_{y\in\bbZ^d}\frac{1}{\omega(x,y)},
  \end{align}
  and write $\mu\ldef\mu^\om(0)$ and $\nu\ldef \nu^\om(0)$ for abbreviation.
 Henceforth the random varable $\mathcal{Y}$ in \eqref{odHK} is fixed as  defined in the proof of Lemma~\ref{L:ondiagHK} below.  In addition to $\mathcal Y$, further random variables will appear in the estimates below. In this subsection, to keep the presentation lean, we say $\mathcal X$ denotes a random variable of class $\mathcal X(c_1,c_2,\ldots,c_n)$ (in short $\mathcal X=\mathcal X(c_1,c_2,\ldots,c_n)$), if it can be written in the form
  \begin{equation*}
    \mathcal X(\omega)\; \ldef \; c_0\bigg( \sum_{x\in\bbZ^d} \big(|x|+1\big)^{-(d+1)} \, \big(\mu^\omega(x)+1\big)^{p_1}\bigg)^{\! p_2},
  \end{equation*}
  with exponents $p_1,p_2\geq 1$ and a constant $c_0\geq 1$ that can
  be chosen only depending on the parameters
  $c_1,\ldots,c_n$. Evidently, the class  is stable under taking products, sums and powers of such random variables. Moreover, finite moments of
  such a random variable are bounded, provided $\mean[\mu^p]<\infty$ for $p$ sufficiently large.  
\smallskip

Following \cite{GNO15}, we lift the on-diagonal estimate of Lemma~\ref{L:ondiagHK} to a weighted $\ell^2$-estimate for $p$ and $\nabla p$.

  \begin{lemma}\label{L:gradHK:aux1}
    Suppose that \eqref{odHK} holds. Let $\alpha\geq \frac{d}{2}+1$ and $\e>0$. Then there exists a random variable $\mathcal X=\mathcal X(d,\alpha,\e)$ such that for all $t\geq 0$ and $T\geq 1$,
    \begin{eqnarray}\label{eq:weight1}
      \sum_{y\in\bbZ^d} \big(|y|+1\big)^{2\alpha} \, p(t,y)^2&\leq&\mathcal Y\mathcal X(t+1)^{-\frac d2+\e+\alpha},\\
      \label{eq:weight2}
      \frac{1}{T}\int_T^{2T}\sum_{y\in\bbZ^d} \big(|y|+1\big)^{2\alpha} \, \big|\nabla p(t,y)\big|^2_\bomega\,dt&\leq&\mathcal Y\mathcal X(T+1)^{-\frac d2-1+\e+\alpha}.
    \end{eqnarray}
  \end{lemma}
  \begin{remark}
    In the case of conductances that are bounded from above, we may choose $\e=0$ and thus recover the optimal scaling in $t$. 
  \end{remark}
  Based on Lemma~\ref{L:gradHK:aux1} we establish the following variant of Proposition~\ref{P:HK}.
  \begin{lemma} \label{L:gradHK}
    Suppose that \eqref{odHK} holds. Let $\alpha\geq 0$ and $0<\e<1$. Then there exists a random variable $\mathcal X=\mathcal X(d,\alpha,\e)$ such that $\prob$-a.s.\ for all $t\geq 0$,
    \begin{align*}
      \sum_{y\in\bbZ^d}  m(t,y)^{2\alpha}\, \big|\nabla p(t,y)\big|_\bomega^2\;\leq \; \mathcal X_t\, (t+1)^{-(\frac{d}{2}+1) +2\e},
    \end{align*}
    where
    \begin{equation*}
      \mathcal X_t(\omega)\, \ldef \; \sum_{z\in\bbZ^d}\frac{m(t,z)^{-(d+1)}}{(t+1)^{\frac d2}} \, \big|(\mathcal Y\mathcal X)(\tau_z\omega)\big|^\frac32 \, \big|(\mathcal Y\mathcal X)(\omega)\big|^\frac12.
    \end{equation*}
    In particular, for every $n\in \bbN$,  $\sup_{t\geq 0}\mean[|\mathcal X_t|^{n}]\leq c$  for some $c=(d,\rho,\e,n,M(p,q))$.
  \end{lemma}
The proofs of Lemma~\ref{L:gradHK:aux1} and Lemma~\ref{L:gradHK} are postponed to Sections~\ref{sec:gradHK:aux1} and \ref{sec:grad:HK}, respectively.  Proposition~\ref{P:HK} easily follows from Lemma~\ref{L:gradHK} and Lemma~\ref{L:ondiagHK} as can be seen by the following short argument.
\begin{proof}[Proof of Proposition~\ref{P:HK}]
 It suffices to consider $n\geq\frac{d}{\varepsilon}$. Let $\theta> \frac{2d}{n}$ and set
  \begin{equation*}
    f_t(e)\; \ldef \; (t+1)^{(\frac{d}{2}+1)-2\varepsilon} \, \Big(m(t,\ue)^{2{\alpha+\theta}} \, \big|\nabla p(t,e)\big|^2 \, \omega(e)\Big).
  \end{equation*}
  By Lemma~\ref{L:gradHK} we have
  \begin{equation}\label{P:HK:eq:S001}
    \|f_t\|_{\ell^1}\; \leq \; \mathcal X_t\qquad\text{and}\qquad \sup_{t\geq 0}\mean\Big[\mathcal X_t^\frac{nq}{2q-n}\Big]\; <\; \infty,
  \end{equation}
 provided $M(p,q)<\infty$ for $p$ and $q$ sufficiently large. Next we consider
  \begin{align} \label{eq:defZt}
    {\mathcal Z}_t(\omega) \; \ldef \; (t+1)^{(\frac{d}{4}+\frac12)-\varepsilon} \, \Big(\sum_{y\in\bbZ^d}m(t,y)^{2\alpha} \big|\nabla p(t,y)\big|^2\Big)^\frac12.
  \end{align}
  %
  Note that ${\mathcal Z}^2_t=(t+1)^{-\varepsilon}\sum_{e\in E_d}g(e)f_t(e)$ where $g(e):=m(t,\ue)^{-\theta} \, \omega(e)^{-1}$. Hence,  H\"older's inequality with exponent $(\frac{n}{2},\frac{n}{n-2})$ and the discrete estimate $\|f_t\|_{\ell^{\frac{n}{n-2}}}\leq\|f_t\|_{\ell^1}$ yield ${\mathcal Z}^2_t \leq(t+1)^{-\varepsilon}\|g\|_{\ell^{\frac{n}{2}}}\|f_t\|_{\ell^1}$. We take the $n/2$-th moment and apply H\"older's inequality (w.r.t.\ $\prob$) with  exponents $(\frac{2q}{n},\frac{2q}{2q-n})$ and the shift-invariance of $\prob$ to obtain
  \begin{eqnarray*}
    \mean\big[{\mathcal Z}_t^n\big]
    &\leq& (t+1)^{-\varepsilon \frac n 2}\mean\Big[\big(\|g\|_{\ell^{\frac{n}{2}}}\|f_t\|_{\ell^1}\big)^\frac{n}{2}\Big] \; = \; 
    (t+1)^{-\varepsilon \frac n 2}\sum_{e\in E_d}m(t,\ue)^{-\theta \frac n 2} \, \mean\Big[\omega(e)^{-\frac{n}{2}}\|f_t\|_{\ell^1}^\frac{n}{2}\Big]\\
    &\leq& \bigg((t+1)^{-\varepsilon \frac n 2}\sum_{e\in E_d}m(t,\ue)^{-\theta \frac n 2}\bigg) \, \mean\Big[ \big|\bomega(0)\big|^{-q}\Big]^\frac{n}{2q} \, \mean\Big[\|f_t\|_{\ell^1}^\frac{nq}{2q-n}\Big]^{\frac{2q-n}{2q}}.
  \end{eqnarray*}
  Note that $\sup_{t\geq 0}(t+1)^{-\varepsilon \frac n 2}\sum_{e\in E_d} m(t,\ue)^{-\theta \frac n 2}<\infty$ since $\theta n/2>d$ and $\varepsilon n/2>\frac{d}{2}$. Furthermore, $\mean\big[|\bomega(0)|^{-q}\big]<\infty$ by the moment condition. Combined with \eqref{P:HK:eq:S001} we finally deduce that $\sup_{t\geq 0}\mean\big[{\mathcal Z}_t^n\big]<\infty$, which completes the proof.
\end{proof}
 
\subsubsection{On-diagonal heat kernel estimate: Proof of Lemma~\ref{L:ondiagHK}} \label{S:ondiag}

The statement is a rather direct consequence of an on-diagonal estimate (see Lemma~\ref{L:odHKaux1} below), which can be obtained from \cite{ADS17}, and an application of the spectral gap estimate of Assumption~\ref{ass:sg} used to control moments of the estimate's random constant, see Lemma~\ref{L:odHKaux2}. 
 Assuming $M(p,q)<\infty$ for any $p,q\in(1,\infty)$, we denote by $\mathcal{R}=\mathcal R(\omega,p,q)\geq 1$ the smallest integer such that for all $R\geq \mathcal{R}$,
  \begin{equation}\label{lem:on-diag:eq0001}
    \frac{1}{\# B(R)}\sum_{x\in B(R)}\mu^\omega(x)^p\leq 2\mean[\mu^p]<\infty \qquad\text{and}\qquad    \frac{1}{\# B(R)}\sum_{x\in B(R)}\nu^\omega(x)^q\leq 2\mean[\nu^q]<\infty,
  \end{equation}
 with $\mu^\om$ and $\nu^\om$ defined in \eqref{eq:def_mu}.  Then, $\prob$-a.s., $\mathcal R<\infty$ by the ergodic theorem.

\begin{lemma}\label{L:odHKaux1}
  Let $p,q\in(1,\infty)$ satisfy $\frac{1}{p}+\frac{1}{q}<\frac{2}{d}$. Suppose that $\prob$ is stationary and ergodic, and that the moment condition $M(p,q)<\infty$ is satisfied.
 Then, there exists $c=c(d,p,q,M(p,q))$ such that for $t\geq \mathcal R^2$,
  \begin{equation*}
    p(t,0) \; \leq \; c \, (t+1)^{-\frac{d}{2}}.
  \end{equation*}
\end{lemma}
\begin{proof}
This on-diagonal bound follows immediately from the upper heat kernel bounds in \cite[Theorem~2.5]{ADS17}, which is based on arguments in \cite{ADS16a}. Indeed,  by our assumptions $\mathcal{R}=\mathcal{R}(\om,p,q)$ defined via \eqref{lem:on-diag:eq0001} is $\prob$-a.s.\ finite for $p,q\in(1,\infty)$  with $\frac{1}{p}+\frac{1}{q}<\frac{2}{d}$. Therefore the assumptions of \cite[Theorem~2.5]{ADS17} are satisfied for $\prob$-a.e.\ $\om$.
Alternatively, the estimate can be deduced from the parabolic Harnack inequality established in \cite{ADS16}, see Proposition~4.7 and Remark~1.5 therein.
\end{proof}

\begin{lemma}\label{L:odHKaux2}
  Suppose that Assumption~\ref{ass:sg} holds and for any $p,q\in(1,\infty)$ let $\mathcal{R}$ be defined via \eqref{lem:on-diag:eq0001}.Then, for any $n\in\bbN$ there exist $p',q'\in(1,\infty)$ (only depending on $p,q,n$) such that under the moment condition $M(p',q')<\infty$ we have $\mean\big[|\mathcal R|^n\big]\leq c$ with $c=c(d,\rho,p,q,n,M(p',q'))$.
\end{lemma}
\begin{proof}
 We only present the argument for $\mu^\omega$, since the argument for $\nu^\omega$ is the same. To that end consider the random variable
  \begin{equation*}
    f_R=f_R^\omega:=    \frac{1}{\# B(R)}\sum_{x\in B(R)}\Big(\frac{\mu^\omega(x)^p}{\mean[\mu^p]}-1\Big).
  \end{equation*}
  which is well-defined since $\mean[\mu^p]>0$ by assumption. We claim that for any $k\in\bbN$,
  \begin{equation}\label{lem:on-diag:eq0002}
    \mean \big[|f_R|^{2k}\big] \; \lesssim \; R^{-dk} \, \mean\big[\mu^{2(p-1)k}\big].
  \end{equation}
 Indeed, since $\mean[f_R]=0$, the spectral gap inequality in form of \eqref{eq:sg_p} yields 
  \begin{equation*}
    \mean\big[|f_R|^{2k}\big]\; \lesssim \; \mean\bigg[\Big(\sum_{e\in E_d}|\partial_ef_R|^2\Big)^k \bigg].
  \end{equation*}
  Since $\partial_e\mu(x)=\indicator_{\{\ue,\oe\}}(x)$, we deduce that
  \begin{equation*}
    \partial_e f_R \; \leq \;
    \begin{cases}
      \frac{p}{\# B(R)} \big(\mu^{p-1}(\ue)+\mu^{p-1}(\oe)\big)&\text{if }\ue\in B(R)\text{ or }\oe\in B_R,\\
      0&\text{else.}
    \end{cases}
  \end{equation*}
  Now, the combination of the previous two estimates gives \eqref{lem:on-diag:eq0002}.
  \smallskip

  Next, by a slight abuse of notation let $\mathcal R(\omega)\geq 1$ be the smallest integer such that
  \begin{equation*}
    \sup_{R\geq\mathcal R(\omega)}|f_R|\leq 1.
  \end{equation*}
  Since $\mu$ is stationary and $\mean[f_R]=0$, Birkhoff's ergodic theorem shows that $\prob$-a.s.\ $\sup_{r\geq R}f_r\to 0$  as $R\uparrow\infty$.  In particular, $\mathcal R<\infty$ $\prob$-a.s.\ and $\mathcal R$ satisfies the first property in \eqref{lem:on-diag:eq0001}. We finally estimate the moments of $\mathcal R$ by using \eqref{lem:on-diag:eq0002}. To that end, note that for all $R\in\bbN$ with $R\geq 2$,
  \begin{equation*}
    \prob\big[\mathcal R=R\big] \; \leq \; \prob\big[|f_{R-1}|>1\big] \; \stackrel{\eqref{lem:on-diag:eq0002}}{\lesssim} \; R^{-dk} \,\mean\big[\mu^{2(p-1)k}\big].
  \end{equation*}
  Hence,
  \begin{eqnarray*}
    \mean\big[\mathcal R^n \big]\; = \; \sum_{R\in\bbN}R^n \, \prob\big[\mathcal R=R\big] \, \lesssim \, \sum_{R\in\bbN}R^{n-dk}\mean\big[\mu^{2(p-1)k}\big].
  \end{eqnarray*}
  We choose $k>(n+1)/d$ and $p'=2(p-1)k$ to get the claim.
\end{proof}

Lemma~\ref{L:ondiagHK} is now a simple consequence of the previous two results.
\begin{proof}[Proof of Lemma~\ref{L:ondiagHK}]
  Since $p(2t,0) =  \sum_{y\in \bbZ^d} p(t,y)^2 $ by the symmetry of the kernel, and the fact that $\sum_{y\in\bbZ^d}p(t,y)=1$, we deduce from Lemma~\ref{L:odHKaux1} that
  \begin{equation*}
    \sum_{y\in\bbZ^d}p(t,y)^2 \; \leq \; \mathcal Y \, (t+1)^{-\frac{d}{2}},\qquad \mathcal Y\; := \; (c+1) \, (\mathcal R+1)^d.
  \end{equation*}
  By Lemma~\ref{L:odHKaux2} we can achieve $\mean[\mathcal Y^n]<\infty$ if sufficiently high moments of $\omega$ and $\omega^{-1}$ exist.
\end{proof}

\subsubsection{Proof of Lemma~\ref{L:gradHK:aux1}} \label{sec:gradHK:aux1} 
 \step 1 First we prove \eqref{eq:weight1}. 
For abbreviation we set $m_0(x)\ldef 1+|x|$, $x\in \bbZ^d$. Recall that $ \partial_t p = -\nabla^*(\bomega\nabla p)$. Hence, by \eqref{eq:adjoint} and the discrete product rule in \eqref{eq:discr_prod} we get
  \begin{align*} 
    &\frac{1}{2}\frac{d}{dt}\sum_{y\in\bbZ^d} m_0(y)^{2\alpha} \, p(t,y)^2 =  \sum_{y\in\bbZ^d} m_0(y)^{2\alpha} \, p(t,y) \, \partial_t  p(t,y) \nonumber \\
    &\qquad =\ - \sum_{e\in E_d}\omega(e)\, \nabla\big( m_0^{2\alpha} p(t,\cdot) \big)(e) \,\nabla p(t,e) \nonumber \\
    &\qquad \leq\ - \sum_{e\in E_d}  \omega(e) \, m_0(\ue)^{2\alpha} \, \big|\nabla p(t,e) \big|^2 
    + \omega(e) \, p(t,\oe) \, \big|\nabla m_0^{2\alpha}(e)\big| \, \big|\nabla p(t,e)\big| \nonumber.
  \end{align*}
  Since $\big|\nabla (m_0(e)^{2\alpha})\big|\leq \sqrt{c_\alpha} \, m_0(\oe)^{\alpha-1} \, m_0(\ue)^{\alpha}$, Young's inequality yields
  \begin{align*}
    & \sum_{e\in E_d} \omega(e)\, p(t,\oe)\, \big|\nabla m_0^{2\alpha}(e)\big| \, \big|\nabla p(t,e)\big| \\
    \leq &
    \; \sqrt{c_\alpha}\sum_{e\in E_d}\sqrt{\omega(e)}\, m_0(\oe)^{\alpha-1}\, p(t,\oe) \cdot \sqrt{\omega(e)}\, m_0(\ue)^\alpha\, \big|\nabla p(t,e)\big| \\
    \leq &
    \; \frac{c_\alpha}{2}\sum_{y\in \bbZ^d}  \mu(y) \, m_0(y)^{2\alpha-2} \, p(t,y)^2
    +\frac{1}{2}\sum_{y\in\bbZ^d} m_0(y)^{2\alpha} \, \big|\nabla p(t,y) \big|^2_{\bomega}
  \end{align*}
  with $|\cdot |_{\bomega}$ as defined in \eqref{eq:def_normom}. We conclude that
  \begin{equation}\label{eq:st0001}
    \frac{d}{dt}\sum_{y\in\bbZ^d} m_0(y)^{2\alpha} \, p(t,y)^2 \; \leq \; \underbrace{c_\alpha \sum_{y\in \bbZ^d} \mu(y) \, m_0(y)^{2\alpha-2} \, p(t,y)^2}_{=:I(t)}
    -\sum_{y\in\bbZ^d} m_0(y)^{2\alpha} \, \big|\nabla p(t,y) \big|^2_{\bomega}.
  \end{equation}
  To estimate $I(t)$ we set $\theta:=\alpha+\e$ and write $2\alpha-2=2\alpha(1-\frac{1}{\theta})-2(1-\frac{\alpha}{\theta})$. Then by H\"older's inequality with exponent $\frac{d+1}{2(1-\frac{\alpha}{\theta})}>1$ and the discrete $\ell^{q}-\ell^1$-estimate (for $q\geq 1$)
  we get
  \begin{equation*}
    I(t)\; \leq \; \mathcal X \, \sum_{y\in\bbZ^d}m_0(y)^{2\alpha(1-\frac{1}{\theta})} \, p(t,y)^2,\qquad \mathcal X\; \ldef \; c_\alpha \, \bigg(\sum_{y\in\bbZ^d}\mu(y)^{\frac{\theta}{\theta-\alpha}\frac{d+1}{2}}m_0(y)^{-(d+1)}\bigg)^{\frac{\theta-\alpha}{\theta}\frac{2}{d+1}}.
  \end{equation*}
 In combination with 
  \begin{eqnarray*}
    \sum_{y\in\bbZ^d}m_0(y)^{2\alpha(1-\frac{1}{\theta})} \, p(t,y)^2  \; \leq \; \bigg(\sum_{y\in\bbZ^d}m_0(y)^{2\alpha} \, p(t,y)^2\bigg)^{\! \frac{\theta-1}{\theta}} \, \bigg(\sum_{y\in\bbZ^d}p(t,y)^2\bigg)^{\!\frac{1}{\theta}},
  \end{eqnarray*}
  Lemma~\ref{L:ondiagHK} and \eqref{eq:st0001} (where we drop the non-positive second term on the right-hand side), this implies
  \begin{equation*}
    \frac{d}{dt}f(t)\; \leq \; \mathcal X \, \mathcal Y^{\frac 1\theta} \, f(t)^{\frac{\theta-1}{\theta}} \, (t+1)^{-\frac{d}{2}\frac 1\theta},\qquad f(t):=\sum_{y\in\bbZ^d}   m_0(y)^{2\alpha} \, p(t,y)^2.
  \end{equation*}
  Hence,
  \begin{eqnarray*}
    \frac{d}{dt}\big(f(t)^{\frac{1}{\theta}}\big) \; \leq \;\frac{1}{\theta}\, f(t)^{\frac{1}{\theta}-1} \, \frac{d}{dt} f(t)\;\leq \;
    \frac{1}{\theta} \, \mathcal X \, \mathcal Y^\frac{1}{\theta}\, (t+1)^{-\frac{d}{2}\frac 1\theta}.
  \end{eqnarray*}
  Since $\frac{d}{2}\frac{1}{\theta}<1$ as $\alpha>\frac{d}{2}$, an integration in $t$ and the fact that $f(0)=1$ yields \eqref{eq:weight1}.
  \medskip

  \step 2  Next we show \eqref{eq:weight2}.  
  The starting point of the argument is \eqref{eq:st0001}, which we recall in an integrated and rearranged form
  \begin{equation*}
    \frac{1}{T}\int_{T}^{2T}\sum_{y\in\bbZ^d}m_0(y)^{2\alpha} \, \big|\nabla p(t,y)\big|^2_\bomega \; \leq \; \frac{1}{T}\sum_{y\in\bbZ^d}m_0(y)^{2\alpha}p(T,y)^2+\frac{1}{T}\int_T^{2T}I(t)\,dt.
  \end{equation*}
  By Step~1 the first term on the right-hand side is bounded by $\mathcal Y\mathcal X(T+1)^{-\frac{d}{2}-1+\e+\alpha}$. 
  Furthermore, by H\"older's inequality and Step~1 we have for all $t\in(T,2T)$,
  \begin{eqnarray*}
    I(t) \; \leq \; \mathcal X\sum_{y\in\bbZ^d}m_0(y)^{2(\alpha-1)+\e} \, p(t,y)^2
    \; \leq \; \mathcal Y \, \mathcal X \, (T+1)^{-\frac{d}{2}-1+\e+\alpha}.
  \end{eqnarray*}
  The combination of the previous estimates yields \eqref{eq:weight2}.
\qed

\subsubsection{Proof of Lemma~\ref{L:gradHK}} \label{sec:grad:HK}
Throughout the proof, $\mathcal X$ denotes a generic random variable (that might change from line to line) of class $\mathcal X(d,\alpha,\e)$. Moreover, we write $\lesssim$ if $\leq$ holds up to a constant only depending on $d,\alpha,\e$. 

\medskip
\step 1 First we show that there exists $\mathcal X=\mathcal X(d,\alpha,\e)$ such that for all $t\geq 0$ and $T\geq 1$, 
\begin{eqnarray}\label{eq:gradHK:pf001}
  \sum_{y\in\bbZ^d} m(t,y)^{2\alpha} \, p(t,y)^2&\leq& \mathcal Y \,\mathcal X \, (t+1)^{-\frac{d}{2}+\e},\\
  \label{eq:gradHK:pf002}
  \frac{1}{T}\int_T^{2T}\sum_{y\in\bbZ^d} m(t,y)^{2\alpha} \, \big|\nabla p(t,y)\big|^2_\bomega\,dt&\leq& \mathcal Y \, \mathcal X \, (T+1)^{-\frac{d}{2}-1+\e}.
\end{eqnarray}
We start with \eqref{eq:gradHK:pf001}. First assume that $\alpha\geq \alpha_0:=\frac{d}{2}+1$. Since $m(t,y)^{2\alpha} \lesssim\left(\frac{(|y|+1)^{2\alpha}}{(t+1)^{\alpha}}+ 1\right)$,
\begin{equation*}
  \sum_{y\in\bbZ^d}  m(t,y)^{2\alpha} \, p(t,y)^2  \; \lesssim \; \bigg(\big(t+1\big)^{-\alpha}\sum_{y\in\bbZ^d} \big(|y|+1\big)^{2\alpha} \, p(t,y)^2  +\sum_{y\in\bbZ^d} p(t,y)^2\bigg),
\end{equation*}
which combined with Lemma~\ref{L:gradHK:aux1} and Lemma~\ref{L:ondiagHK} yields the desired estimate. In the case $0\leq\alpha\leq\alpha_0$ we proceed by interpolating the estimate for $\alpha=0$ and $\alpha=\alpha_0$. Indeed, H\"older's inequality yields
\begin{align*}
  \sum_{y\in\bbZ^d} m(t,y)^{2\alpha} \, p(t,y)^2  \; \leq \; \bigg(\sum_{y\in\bbZ^d} m(t,y)^{2\alpha_0} \, p(t,y)^2 \Bigg)^{\!\frac{\alpha}{\alpha_0}}\Bigg(\sum_{y\in\bbZ^d}p(t,y)^2\bigg)^{\!\frac{\alpha_0-\alpha}{\alpha_0}}.
\end{align*}
The first term on the right-hand side can be estimated as above, and the second term on the right-hand side is estimated by Lemma~\ref{L:ondiagHK}.
\smallskip

 Next, we prove \eqref{eq:gradHK:pf002}. First note that
 \begin{eqnarray}\label{eq:ondiag_grad:1a}
  \frac{1}{T}\int_T^{2T} \sum_{y\in\bbZ^d} \big|\nabla p(t,y)\big|^2_{\bomega} \, dt
  \; \leq \;\mathcal Y \, \mathcal{X}\, (T+1)^{-\frac d 2-1}.
\end{eqnarray}
Indeed, by integrating the identity  $\frac{1}{2} \frac d {dt} \sum_{y\in \bbZ^d} p(t,y)^2 = - \sum_{y\in\bbZ^d} |\nabla p(t,y)|^2_\bomega$  w.r.t.\ $t$, we get $\frac{1}{T}\int_T^{2T} \sum_{y\in\bbZ^d} \big|\nabla p(t,y)\big|^2_{\bomega} \, dt \leq
\frac{1}{2T} \sum_{y\in \bbZ^d} p(T,y)^2$, which in combination with Lemma~\ref{L:ondiagHK} yields \eqref{eq:ondiag_grad:1a}. Now, we argue as above to obtain \eqref{eq:gradHK:pf002} for $\alpha\geq\alpha_0$ by using Lemma~\ref{L:gradHK:aux1}. Finally, the estimate for $0\leq \alpha\leq\alpha_0$ follows (as in the proof of \eqref{eq:gradHK:pf001}) by interpolation.
\medskip

\step 2 In this step we show that for any $\alpha\geq 0$ and $0<\e<1$ there exists $\mathcal X=\mathcal X(d,\alpha,\e)$ such that 
\begin{equation*}
  m(t,z)^\alpha p^\omega(t,0,z) \; \leq \; \sqrt{(\mathcal Y\mathcal X)(\omega) \,(\mathcal Y\mathcal X)(\tau_z\omega)} \, (t+1)^{-\frac{d}{2}+\e},
\end{equation*}
for all $z\in\bbZ^d$ and $t\geq 0$.
 By the triangle inequality for the weight in form of $m(2t,z)^\alpha\lesssim m(t,z-y)^\alpha \, m(t,y)^\alpha$, the semigroup property, and the shift property $p^{\tau_z\om}(t,0,y-z)=p^\om(t,y,z)$,
\begin{eqnarray*}
   m(2t,z)^\alpha p^\om(2t,0,z)
   & \lesssim   &  \sum_{y\in \bbZ^d} m(t,y)^\alpha\, p^\om(t,0,y)\, m(t,z-y)^\alpha \, p^\om(t,y,z)\\
   &\lesssim & \Bigg(\sum_{y\in \bbZ^d} m(t,y)^{2\alpha} \, p^\om(t,0,y)^2\Bigg)^{\! \frac 1 2} \, \Bigg(\sum_{y\in\bbZ^d} m(t,z-y)^{2\alpha} \, p^\om(t,y,z)^2\Bigg)^{\! \frac 1 2}\\
   &=& \Bigg(\sum_{y\in \bbZ^d} m(t,y)^{2\alpha} \, p^\om(t,0,y)^2\Bigg)^{\! \frac 1 2} \, \Bigg(\sum_{y\in\bbZ^d} m(t,z-y)^{2\alpha} \, p^{\tau_z\om}(t,y-z,0)^2\Bigg)^{\! \frac 1 2}.
\end{eqnarray*}
Using symmetry in form of $p^{\tau_z\omega}(t,y-z,0)=p^{\tau_z\omega}(t,0,y-z)$ and applying \eqref{eq:gradHK:pf001} yields
\begin{eqnarray*}
   m(2t,z)^\alpha p^\om(2t,0,z)
   & \leq & \sqrt{(\mathcal Y\mathcal X)(\omega)(\mathcal Y\mathcal X)(\tau_z\omega)}(t+1)^{-\frac{d}{2}+\e}.
\end{eqnarray*}
%

\step 3 Now we show the statement. First note that it suffices to prove the claimed estimate for $t\geq 1$, since for $0\leq t\leq 1$ the estimate follows from \eqref{eq:gradHK:pf001} and the fact that $|\nabla p(t,e)|\leq p(t,\ue)+p(t,\oe)$.
For $t\geq 1$, by the semigroup property and Jensen's inequality we have for all $e\in E_d$,
\begin{equation*}
  \big|\nabla p^\omega(t,0,e)\big|^2 \;\leq \; \frac{3}{t}\int_{\frac{t}3}^{\frac 2 3 t}\sum_{z\in \bbZ^d} p^\om(t-s,0,z) \, \big|\nabla p^\om(s,z,e)\big|^2\,ds,
\end{equation*}
and by the triangle inequality for the weight in form of $m(s,\ue)^{2\alpha}\lesssim m(t-s,z)^{2\alpha}\, m(s,\ue-z)^{2\alpha}$,
\begin{align*}
  I\; \ldef \; &\sum_{e\in E_d}m(t,\ue)^{2\alpha} \, \big|\nabla p^\omega(t,0,e)\big|^2 \, \omega(e)\\
  \lesssim \; & \frac{3}{t}\int_{\frac{t}3}^{\frac 2 3 t}\sum_{z\in \bbZ^d} \sum_{e\in E_d}m(t,\ue)^{2\alpha} \, p^\om(t-s,0,z) \, \big|\nabla p^\omega(t,z,e)\big|^2 \, \omega(e) \,ds\\
     \;\lesssim  \; &\frac{3}{t}\int_{\frac{t}3}^{\frac23 t }\sum_{z\in \bbZ^d} m(t-s,z)^{2\alpha} \, p^\omega(t-s,0,z)\, \Bigg(\sum_{e\in E_d} m(s,\ue-z)^{2\alpha} \, \big|\nabla p^\om(s,z,e)\big|^2 \,  \omega(e) \Bigg)\,ds.
\end{align*}
Now, note that  Step~2 implies for all $s\in(\frac{t}{3},\frac{2t}{3})$, 
  \begin{align*}
  & m(t-s,z)^{2\alpha} \, p^\omega(t-s,0,z)
   \;= \; 
   m(t-s,z)^{2\alpha+d+1} \, p^\omega(t-s,0,z) \, m(t-s,z)^{-(d+1)}\\
    & \mspace{36mu} \leq \;\sqrt{(\mathcal Y\mathcal X)(\om) \, (\mathcal Y\mathcal X)(\tau_z\om)} \, \big(t-s+1\big)^{-\frac{d}{2}+\e} \, m(t-s,z)^{-(d+1)}\\
    &\mspace{36mu} \leq \; \sqrt{(\mathcal Y\mathcal X)(\om) \, (\mathcal Y\mathcal X)(\tau_z\om)} \, \big(t+1\big)^{\e} \, \tilde m(t,z),\qquad\tilde m(t,z):=\frac{m(t,z)^{-(d+1)}}{(t+1)^\frac d2},
  \end{align*}
  where we used in the last step that $\frac{t}{3}\leq t-s\leq \frac{2t}{3}$.
Further,  the shift property $\nabla p^{\tau_z\om}(s,0,e-z)=\nabla p^\om(s,z,e)$ gives that
  \begin{equation*}
    \sum_{e\in E_d} m(s,\ue-z)^{2\alpha} \, \big|\nabla p^\om(s,z,e)\big|^2 \,  \omega(e) \; = \; \sum_{y\in\bbZ^d} m(s,y)^{2\alpha} \, \big|\nabla p^{\tau_z\om}(s,0,y)\big|^2_{\tau_z\bomega}.
  \end{equation*}

We conclude that
\begin{align*}
  I  \; \lesssim  \; 
  & \big(t+1\big)^\e \sum_{z\in \bbZ^d}\tilde m(t,z)
          \sqrt{(\mathcal X \mathcal Y)(\tau_z\omega)\, (\mathcal X \mathcal Y)(\omega)} \, \Bigg(\frac{3}{t}\int_{\frac{t}3}^{\frac23 t }\sum_{y\in\bbZ^d} m(s,y)^{2\alpha} \, \big|\nabla p^{\tau_z\om}(s,0,y)\big|^2_{\tau_z\bomega}\,ds \Bigg)\\
   \;\stackrel{\eqref{eq:gradHK:pf002}}{\leq} \; & (t+1)^{-\frac{d}{2}-1+2\e} \, \mathcal X_t,\qquad\mathcal X_t:=\sum_{z\in \bbZ^d}\tilde m(t,z) \,
  \big((\mathcal Y\mathcal X)(\tau_z\omega)\big)^{\frac32} \, \big((\mathcal Y\mathcal X)(\omega)\big)^\frac12.
\end{align*}
Finally, $\sup_{t\geq 1}\mean\big[|\mathcal{X}_t|^n\big]<\infty$ for every $n\in \bbN$ since 
$\|\tilde m(t,x)\|_{\ell^1}\lesssim 1$ uniformly in $t$  and arbitrarily high moments of $\mathcal Y$ and $\mathcal X$ exist if $M(p,q)<\infty$ holds for $p$ and $q$ sufficiently large. 
\qed

\subsection{Annealed  estimates: Proofs of Corollary~\ref{C:AHK} and Lemma~\ref{L:AG}} \label{S:annealedHK}

\begin{proof}[Proof of Corollary~\ref{C:AHK}]
 \step 1 Estimates for $\nabla p$ and $\nabla\nabla p$ for $n=1$.
\smallskip

Since the argument for $\nabla p$ and $\nabla\nabla p$ are similar, we only discuss the estimate for $\nabla\nabla p$, which follows the  discussion below \cite[Proposition~1]{MO15}. For the reader's convenience we sketch the short argument. Since $p(t,x,x')=\sum_{y\in\bbZ^d}p(\frac{t}{2},x,y) \, p(\frac{t}{2},y,x')$ by the semigroup property, for any $e,e'\in E_d$ we have
  \begin{equation*}
    \nabla\nabla p(t,e,e')=\sum_{y\in\bbZ^d}\nabla p(\tfrac{t}{2},e,y) \, \nabla p(\tfrac{t}{2},y,e').
  \end{equation*}
  We multiply this identity with $m(t,\ue-\ue')^\alpha$ and obtain by the triangle inequality for the weight, i.e.\ $m(t,\ue-\ue')^\alpha\leq 2^\alpha m(t,\ue-y)^\alpha m(t,\ue'-y)^\alpha$, and the Cauchy Schwarz inequality in $\sum_{y\in\bbZ^d}$,
  \begin{eqnarray*}
    &&m(t,\ue-\ue')^\alpha \, \big|\nabla\nabla p(t,e,e')\big|\\
    &\leq& 2^\alpha \, \bigg(\sum_{y\in\bbZ^d} m(t,\ue-y)^{2\alpha}  \, \big|\nabla p(\tfrac{t}{2},e,y)\big|^2\bigg)^{\! \frac{1}{2}}\bigg(\sum_{y\in\bbZ^d} m(t,\ue'-y)^{2\alpha}  \, \big|\nabla p(\tfrac{t}{2},e',y)\big|^2\bigg)^{\! \frac{1}{2}}.
  \end{eqnarray*}
  We take the expectation, apply Cauchy Schwarz w.r.t. $\prob$ and exploit stationarity and symmetry in form of $\mean\big[|\nabla p(t,e,y)|^2\big]=\mean\big[|\nabla p(t,e-y,0)|^2\big]$ to obtain
  \begin{align*}
    & m(t,\ue-\ue')^\alpha\, \mean\Big[\big|\nabla\nabla p(t,e,e')\big|\Big] \leq 2^\alpha\sum_{y\in\bbZ^d}\mean\Big[\big|\nabla p(\tfrac{t}{2},e-y,0)\big|^2 \, m(t,\ue-y)^{2\alpha}\Big]\\
    & \mspace{36mu} \leq \; 2^\alpha \sum_{e''\in E_d} \mean\Big[ \big|\nabla p(\tfrac{t}{2},e'',0)\big|^2 \, m(t,\ue'')^{2\alpha}\Big]\\
    & \mspace{36mu} \lesssim \; \, (t+1)^{-(\frac d 2 +1) +\varepsilon},  
  \end{align*}
  where the last inequality holds by Proposition~\ref{P:HK}. Since $e,e' \in E_d$ are arbitrary, the claim follows.
\medskip

\step 2 Estimate for $\nabla p$ and $n\gg 1$.
\smallskip

Proposition~\ref{P:HK} yields
\begin{eqnarray*}
  \mean\Big[ \big|\nabla p(t,y)\big|^{n}\Big]^{\frac{1}{n}} &=&  \mean\Big[\big(m(t,y)^\alpha \, |\nabla p(t,y)|\big)^{n}\Big]^{\frac{1}{n}} \, m(t,y)^{-\alpha}\\
  &\leq&   \mean\bigg[\Big(\sum_{x\in\bbZ^d}m(t,x)^{2\alpha} \big|\nabla p(t,0,x)\big|^{2}\Big)^\frac{n}{2}\bigg]^{\frac{1}{n}} \, m(t,y)^{-\alpha}\\
  &\leq& \mean\big[\mathcal Z_t^{n}\big]^{\frac{1}{n}} \, (t+1)^{-(\frac{d}{4}+\frac{1}{2})+\e} \, m(t,y)^{-\alpha}\\
                                                            &\lesssim& (t+1)^{-(\frac{d}{4}+\frac{1}{2})+\e} \, m(t,y)^{-\alpha},
\end{eqnarray*}
where we used that $\mathcal{Z}_t$ defined in \eqref{eq:defZt} satisfies $\sup_{t\geq 0}\mean[\mathcal Z_t^{n}]^{\frac{1}{n}}<\infty$.
\medskip

\step 3 Estimates for $\nabla p$ and $\nabla\nabla p$ for $n>1$.

For $n>1$ we obtain the claimed estimates by interpolation of the estimates in Step~1 and Step~2 via H\"older's inequality in form of
\begin{equation*}
   \|u\|_{L^n(\Omega)}\leq\|u\|_{L^1(\Omega)}^{\lambda}\|u\|_{L^{\frac{n(1-\lambda)}{1-\lambda n}}(\Omega)}^{1-\lambda}, \qquad  0<\lambda<\frac{1}{n}.
\end{equation*}
Indeed, applied to $u=\nabla p$, we obtain
\begin{eqnarray*}
  \mean\Big[\big|\nabla p(t,y)\big|^n\Big]^\frac{1}{n}&\lesssim& (t+1)^{-(\frac{d}{2}+\frac{1}{2})\lambda-(\frac{d}{4}+\frac{1}{2})(1-\lambda)+\e} \, m(t,y)^{-\alpha}\\
  &=&(t+1)^{-(\frac{d}{2}+\frac{1}{2})+\frac{d}{4}(1-\lambda)+\e} \, m(t,y)^{-\alpha},
\end{eqnarray*}
and the claimed estimate follows by choosing $\lambda$ close to $\frac{1}{n}$. For $\nabla\nabla p$ first notice that by the triangle inequality we have 
\begin{equation*}
  \mean\Big[\big|\nabla\nabla p(t,x)\big|^n\Big]^{\frac{1}{n}}\;\lesssim \; \max_{|x-x'|\leq 1}\mean\Big[\big|\nabla p(t,x')\big|^n\Big]^{\frac{1}{n}}.
\end{equation*}
Now, the estimate follows by interpolating the estimate for $\nabla\nabla p$ for $n=1$ with the estimate for $\nabla p$ for $n'\gg n$.
\end{proof}
 
\begin{proof}[Proof of Corollary~\ref{L:AG}]
  First note that for $\theta>0$ and $0\leq \alpha<2\theta$, we have for all $x\in\bbZ^d$ the elementary estimate
  \begin{equation}\label{L:AG:eq1}
    \int_0^\infty(t+1)^{-\theta-1}\, m(t,x)^{-\alpha}\,dt \; \leq \; c(d,\alpha,\theta)\, \big(|x|+1\big)^{-\alpha},
  \end{equation}
  which follows by using that $m(t,x)^{-\alpha}\lesssim \big(|x|+1\big)^{-\alpha} \big(t+1\big)^{\alpha/2}$. Now, the claimed estimates follow from the identities $\nabla G(x,0)=\int_0^\infty\nabla p(t,x,0)\,dt$ and $\nabla\nabla G(x,0)=\int_0^\infty \nabla \nabla p(t,x,0)\,dt$, the estimates in Corollary~\ref{C:AHK} and \eqref{L:AG:eq1} by choosing $\alpha$ close to $2\theta$.
\end{proof}

\section{Semigroup decay -- Proof of Theorem~\ref{T:decay}}\label{S:decay}
In this section we prove Theorem~\ref{T:decay}, which yields a rate for the decay of the semigroup $P_t:L^\infty(\Omega)\to L^\infty(\Omega)$ defined by $P_tu(\omega):=\sum_{y\in\bbZ^d}p^\omega(t,0,y)\, u(\tau_y\omega)$. First we recall some basic facts on the semigroup $(P_t)_{t\geq 0}$. We refer to \cite{GNO15} for details.
\begin{itemize}
\item Since the heat kernel is normalised such that $\sum_{y\in\bbZ^d}p^\omega(t,0,y)=1$, the semigroup $P_t$ is a contraction on $L^\infty(\Omega)$.
\item The semigroup is characterised by a discrete heat equation on $\bbZ^d$. The connection is based on the stationary extension that associates a random variable, say $u(\omega)$, with the random field $\bar u(\omega,x):=u(\tau_x\omega)$ called the stationary extension of $u$. Now, consider $u\in L^\infty(\Omega)$ and $v(t):=P_t u$. Then for $\prob$-a.e. $\omega$, the function $\bar v(\omega,\cdot,\cdot):[0,\infty)\times\bbZ^d\to\bbR$, $\bar v(\omega,t,x):=v(t,\tau_x\omega)$ is the unique solution in $C([0,\infty),\ell^\infty(\bbZ^d))\cap C^1((0,\infty),\ell^\infty(\bbZ^d))$ to the Cauchy problem
\begin{equation}
  \label{eq:hea-eq}
  \begin{aligned}
    (\partial_t+\nabla^*\bomega\nabla)\bar v \; = \; &0&\qquad&\text{on }(0,\infty)\times\bbZ^d,\\
    \bar v(0,\cdot)\; = \; &\bar u&&\text{on }\bbZ^d,
  \end{aligned}
\end{equation}
which directly follows from the definition of the semigroup.
\item An alternative characterisation by a Cauchy problem in $L^\infty(\Omega)$ is as follows. The stationary extension $\overline{(\cdot)}$, the discrete derivatives $\nabla_i,\nabla_i^*$ and the horizontal derivatives $D_i,D_i^*$ (see Section~\ref{sec:notation}) are related by the identities
\begin{equation*}
  \nabla_i\bar u(\omega,x)=\overline{(D_iu)}(\omega,x),\qquad   \nabla_i^*\bar u(\omega,x)=\overline{(D_i^*u)}(\omega,x).
\end{equation*}
Therefore, \eqref{eq:hea-eq} is equivalent to the Cauchy problem in $L^\infty(\Omega)$ given by
\begin{equation}\label{eq:hea-prob}
  \begin{aligned}
    (\partial_t+D^*\bomega(0)D)v \; = \; &0&\qquad&\text{for }t>0,\\
    v(0)\; = \; &u.&&
  \end{aligned}
\end{equation}
\item The family $(P_t)_{t\geq 0}$ is the Markovian transition semigroup associated with the $\Omega$-valued process $\{ \tau_{X_t} \om \}_{t\geq 0}$, which is known as the process of the ``environment as seen from the particle''. Furthermore, if $\prob$ is stationary and ergodic, and \eqref{ass:moment} holds, then the measure $\prob$ is stationary, reversible and ergodic for the environment process $\{ \tau_{X_t} \om \}_{t\geq 0}$ and its semigroup $(P_t)_{t\geq 0}$, respectively (see e.g.\ \cite[Lemma~2.4]{ADS15}).

\end{itemize}

\subsection{Proof of Theorem~\ref{T:decay}}
We follow the argument of \cite[Theorem~1]{GNO15}, where the optimal estimate is obtained in the case of uniformly elliptic coefficients. In our setting the lack of uniform ellipticity leads to a loss of decay, since at various places we use H\"older's inequality and the moment conditions in order to move the conductances outside (or inside) of some integrals. A central element is the weighted $\ell^2$-regularity estimate for the gradient of the heat kernel obtained in Proposition~\ref{P:HK}, which we apply in form of the following two estimates, the proof of which are postponed to Section~\ref{S:auxLemmas}.
\begin{lemma}\label{L:decay-aux}
  Let $\varepsilon\in(0,1)$, $n\geq\frac{d}{2\varepsilon}$ and $\theta\in(1,2)$. There exists $p,q\in(1,\infty)$ (only depending on $d,n,\e,\theta$) and a constant $c=c(d,n,\e,\theta,M(p,q))$ such that if $M(p,q)<\infty$ the following holds.
  For all $F\in L^{2n\theta}(\Omega,\bbR^d)$ and $t\geq 0$,
  \begin{eqnarray*}
    \mean \Bigg[ \bigg(\sum_{y\in \bbZ^d} \Big(\nabla p(t,y)\cdot \bar F(y)\Big)^2\bigg)^{\!n} \, \Bigg]^{\frac 1{2n}} \; \leq \; c \,  (t+1)^{-(\frac{d}{4}+\frac{1}{2})+\varepsilon} \, \mean\Big[|F|^{2n\theta}\Big]^{\frac{1}{2n\theta}},
  \end{eqnarray*}
  where $\bar F(\omega,x):=F(\tau_x\omega)$.
\end{lemma}
\begin{remark}
  In the uniformly elliptic case Lemma~\ref{L:decay-aux} holds with $\varepsilon=0$ and $\theta=1$, which can be shown along the lines of \cite{GNO15}. For our purpose it is important that $\theta$ and $\ve$ can be chosen arbitrarily close to $1$ and $0$, respectively, provided we suppose a sufficiently strong moment condition.
\end{remark}
\begin{lemma}\label{L:decay-aux-2}
  Let $\varepsilon\in(0,1)$ and $n\geq\frac{d}{2\varepsilon}$. There exist $p,q\in(1,\infty)$ (only depending on $d,n,\e$) and a constant $c=c(d,n,\e,M(p,q))$ such that if $M(p,q)<\infty$ the following holds.
  For any random field $H:\Omega\times E_d \to\bbR^d$ and $t\geq0$,
  \begin{align*}
    \mean \Bigg[ \bigg(\sum_{e\in E_d}\Big(\sum_{z\in \bbZ^d} \nabla p(t,z) \cdot H(\tau_z\omega,e-z)\Big)^{\!2}\bigg)^{\! n} \, \Bigg]^{\! \frac 1{2n}}
  \;  \leq \; c \, (t+1)^{-(\frac{d}{4}+\frac{1}{2})+\varepsilon}\sum_{e\in E_d}\mean\Big[\big|H(\omega,e)\big|^{4n}\Big]^{\! \frac{1}{4n}}.
  \end{align*}
\end{lemma}

For the proof of Theorem~\ref{T:decay} we further need  a non-linear Caccioppoli inequality for the operator $D^*\bomega(0) D$. The following result is an extension of \cite[Lemma~5]{GNO15} to the degenerate setting.

\begin{lemma}[Caccioppoli] \label{lem:cacc}
  Let $v(t)=P_tu$. Then for every $n\in\bbN$ and $1<\theta<2$ we have
\begin{align*}
  \mean \Big[\big|Dv(t)\big|^{2n\theta}  \Big]^{\frac{1}{2n\theta}} \; \leq \; \left(- c \, \frac d {dt} \mean\big[v(t)^{2n}\big]\right)^{\! \frac{1}{2n}\frac{2-\theta}{\theta}}
  \, \mean\big[|u|^{8n}\big]^{\frac{1}{8n}\frac{2\theta-2}{\theta}} \mean\Big[\big|\bomega(0)^{-1}\big|^\frac{2(2-\theta)}{\theta-1}\Big]^{\frac{\theta-1}{4n\theta}}
\end{align*}
with $c=c(d,n,\theta)$. 
\end{lemma}
A last ingredient is the decay estimate on ordinary differential inequalities.
\begin{lemma} \label{lem:ODE}
Assume that 
\begin{align*}
0 & \leq \; a(t) \; \leq \; c_0 \, \left( \big(t+1\big)^{-\gamma} + \int_0^t \big(t-s+1\big)^{-\gamma} \, b^\delta(s) \, ds \right), \\
0 & \leq \;  b^{2n}(t) \; \leq \; - \frac d {dt} \big[ a^{2n}(t) \big]  
\end{align*}
with $n\in [1,\infty)$, $\gamma\in [1,\infty)$ and $\delta \in \big(\frac{\gamma}{\gamma+\frac 1 {2n}},1\big)$. Then, there exists  $c=c(n, \gamma, \delta, c_0)<\infty$ such that
\begin{align*}
a(t) \; \leq \; c \, (t+1)^{-\gamma}.
\end{align*}
\end{lemma}
This is a generalisation of \cite[Lemma~15]{GNO15}, obtained in \cite[Lemma~3.1]{GM16}.
\begin{proof}[Proof of Theorem~\ref{T:decay}]
 Let $\theta\in(1,2)$ to be chosen later. In the following $\lesssim$ stands for $\leq$ up to a constant only depending on $d,\rho,n,\theta,\e$ and $M(p,q)$. 
  For abbreviation we set $\gamma:=(\frac{d}{4}+\frac{1}{2})-\varepsilon$,  $u:=D^*F$ and $v(\omega,t):=P_tu(\omega)=\sum_{y\in\bbZ^d}p^\omega(t,0,y) \, \bar u(\omega,y)$. 
  By the homogeneity of the estimate it suffices to consider the case 
  \begin{equation}\label{T:decay:eq5}
 \sum_{e\in E_d}   \mean\Big[\big|\partial_eF\big|^{8n}\Big]\leq 1.
  \end{equation}
  We claim that for any $\theta\in(1,2)$,
  \begin{equation}\label{P:decay-2:eq1-class}
    \mean\big[v(t)^{2n}\big]^{\frac 1{2n}} \; \lesssim \; (t+1)^{-\gamma}+\int_0^t (t-s+1)^{-\gamma} \, \mean\Big[\big|Dv(s)\big|^{2n\theta}\Big]^{\frac{1}{2n\theta}} \, ds.
  \end{equation}
  We first note that $\mean[u]=0$ implies $\mean[v(t)]=0$ for all $t\geq 0$. Hence,  the $n$-version of the spectral gap estimate in \eqref{eq:sg_p} gives
  \begin{equation}\label{P:decay-2:eq1b-class}
    \mean\big[v(t)^{2n}\big]^{\frac 1{2n}} \; \lesssim \; \mean\bigg[\Big(\sum_{e\in E_d} \big|\partial_ev(t)\big|^{2}\Big)^{\! n}\bigg]^{\frac 1{2n}}.
  \end{equation}
  In order to identify $\partial_e v$, recall that $\bar v(\omega,\cdot,\cdot)$ solves
  \begin{equation}\label{P:decay-2:eq001}
    \begin{aligned}
      (\partial_t+\nabla^*\omega\nabla)\bar v \; = \; &0&&\text{on }(0,\infty)\times\bbZ^d,\\
      \bar v(t=0,\cdot)\; = \; &\bar u(\cdot)&&\text{on }\bbZ^d.
    \end{aligned}
  \end{equation}
Now we apply $\partial_e$ to this equation to get a characterisation of $\partial_e\bar v$. More precisely, since $\partial_e(\nabla^*\omega\nabla\bar v)=\nabla^*\omega\nabla\partial_e\bar v(s,\cdot)+\nabla^*\delta_e(\cdot)\nabla\bar v(s,e)$ (here $\delta_e:E_d\to\{0,1\}$ denotes the Dirac function with $\delta_e(e)=1$ and $\delta_e(e')=0$ for any $e'\neq e$), we have
  \begin{equation*}
    \begin{aligned}
      (\partial_t+\nabla^*\omega\nabla)\partial_e\bar v \; = \; &-\nabla^*\delta_e(\cdot)\nabla\bar v(t,e)&&\text{on }(0,\infty)\times\bbZ^d,\\
      \partial_e\bar v(t=0,\cdot)\; = \;&\partial_e\bar u(\cdot)&&\text{on }\bbZ^d,
    \end{aligned}
  \end{equation*}
  and thus by Duhamel's formula
  \begin{align*}
    \partial_e v(t)\; = \;\partial_e \bar v (t,0) \;=\; \sum_{z\in \bbZ^d} p^\om(t,0,z) \, \partial_e \bar u(z) + \int_0^t \nabla p^\om(t-s,0,e)\nabla\bar v(s,e) \, ds.
  \end{align*}
  We combine this identity with \eqref{P:decay-2:eq1b-class} and apply the triangle inequality to obtain
  \begin{align}
    \mean\big[v(t)^{2n}\big]^{\frac 1{2n}} \; \lesssim \; &  \mean \Bigg[ \Bigg(\sum_{e\in E_d}\Big(\sum_{z\in \bbZ^d} p(t,z) \, \partial_e \bar u(z)\Big)^2\Bigg)^n \, \Bigg]^{\frac 1{2n}} \nonumber  \\ 
    &+
    \int_0^t\mean \Bigg[ \Bigg(\sum_{e\in E_d}\Big(\nabla p(t-s,e) \, \nabla\bar v(s,e)\Big)^2\Bigg)^{\!n} \, \Bigg]^{\frac 1{2n}}\, ds \nonumber\\
    =:\; & I+II.\label{eq:afterSG}
  \end{align}
  For term $I$ note that $\bar u(\omega,x)=\nabla^*\bar F(\omega,x)$ where $\bar F$ denotes the stationary extension of $F$. 
By definition of $\partial_e$, $\nabla^*$ and the stationary extension, we have for any random variable $f$ the general calculus rules $\partial_e(\nabla_if)=\nabla_i(\partial_ef)$ and $\partial_e\bar f(\omega,x)=\overline{(\partial_{\omega(e-x)}f)}(\omega,x)$. Hence,
  \begin{equation*}
    \partial_e\bar u(z)=\partial_e\nabla^*\bar F(z)=\nabla^*\overline{\partial_{\omega(e-z)}F}(z)=\nabla^*H(\tau_z\omega,e-z),
  \end{equation*}
  where $H(\omega,e):=\partial_eF(\omega)$. Hence, using an integration by parts,  Lemma~\ref{L:decay-aux-2} and \eqref{T:decay:eq5} we get
  \begin{align}\label{T:decay:eq4}
    I \; =\ \; \mean \Bigg[ \Bigg(\sum_{e\in E_d}\Big(\sum_{z\in \bbZ^d} \nabla p(t,z) \cdot H(\tau_z\omega,e-z)\Big)^2\Bigg)^n \, \Bigg]^{\frac 1{2n}} \; \lesssim \; (t+1)^{-\gamma}.
  \end{align}
  For term $II$ we use Lemma~\ref{L:decay-aux} and get for any $\theta\in(1,2)$,
  \begin{equation*}
    II \; \lesssim \; \int_0^t (t-s+1)^{-\gamma} \, \mean\Big[\big|Dv(s)\big|^{2n\theta}\Big]^{\frac{1}{2n\theta}} \, ds,
  \end{equation*}
  which  completes the argument for \eqref{P:decay-2:eq1-class}.
  \medskip

  Next we apply Lemma~\ref{lem:cacc} to obtain
  \begin{align} \label{eq:appl_cacc}
   \mean \Big[\big|Dv(t)\big|^{2n\theta}  \Big]^{\frac{1}{2-\theta}} \; \lesssim \; -  \frac d {dt} \mean\big[v(t)^{2n}\big],
  \end{align}
%
%
  %
  where we have used that 
  \begin{equation*}
    \mean\big[|u|^{8n}\big]\; \leq \; \mean\Big[\big|D^*(F-\mean[F])\big|^{8n}\Big] \; \lesssim \; \mean\bigg[\Big(\sum_{e\in E_d}\big|\partial_e F\big|^2\Big)^{4n}\bigg] \; \leq \; 1,
  \end{equation*}
  thanks to \eqref{eq:sg_p}, a discrete $\ell^2$-$\ell^1$-estimate  and \eqref{T:decay:eq5}. Finally, in view of \eqref{P:decay-2:eq1-class} and \eqref{eq:appl_cacc}, choosing $\theta>1$ sufficiently close to $1$ (only depending on $d$, $\varepsilon$ and $n$), we may apply Lemma~\ref{lem:ODE} with $  a(t):=\mean\big[v(t)^{2n}\big]^{\frac 1{2n}}$, $b(t):= \mean \Big[\big|Dv(t)\big|^{2n\theta} 
 \Big]^{\frac{1}{2n(2-\theta)}}$ and $\delta:=\frac {2-\theta}{\theta}$  and get
  \begin{equation*}
    \mean\big[v(t)^{2n}\big]^{\frac 1{2n}}\lesssim (t+1)^{-\gamma},
  \end{equation*}
  which is the claimed estimate.
\end{proof}

\subsection{Proofs of Lemmas~\ref{L:decay-aux},~\ref{L:decay-aux-2} and  \ref{lem:cacc}} \label{S:auxLemmas}

\begin{proof}[Proof of Lemma~\ref{L:decay-aux}]
  In the following $\lesssim$ stands for $\leq$ up to a constant only depending on $d,n, \theta$ and $\e$. Fix $\alpha>\frac{d}{2n}$. Consider $I:=\left(\sum_{y\in \bbZ^d}m(t,y)^{-2n\alpha}\big|\bar F(y)\big|^{2n}\right)^\frac{1}{n}$. H\"older's inequality, the discrete $\ell^{\frac{n}{n-1}}$-$\ell^1$-estimate and  Proposition~\ref{P:HK} yield
\begin{eqnarray*}
  \sum_{y \in \bbZ^d}\Big(\nabla p(t,y)\cdot \bar F(y)\Big)^2
  &\leq&I \, \bigg(\sum_{y\in\bbZ^d} \Big(m(t,y)^{2\alpha} \, \big|\nabla p(t,y)\big|^2\Big)^{\frac{n}{n-1}} \bigg)^{\!\frac{n-1}{n}}\\
  &\leq&I \, \bigg(\sum_{y\in\bbZ^d} m(t,y)^{2\alpha} \, \big|\nabla p(t,y)\big|^2\bigg) \; \leq \; I \, \mathcal Z_t^2 \, (t+1)^{-(\frac{d}{2}+1)+\varepsilon}.
\end{eqnarray*}
We take the expectation of the $n$-th power  to obtain
\begin{align*}
  &\mean\Bigg[\bigg(\sum_{y\in \bbZ^d}\Big(\nabla p(t,y)\cdot \bar F(y)\Big)^2\bigg)^{\!n}\Bigg]^\frac 1 {2n} \\
  \leq \; & (t+1)^{-(\frac{d}{4}+\frac 12)+\frac{\e}{2}}\,\mean\bigg[\mathcal
    Z_t^{2n}\sum_{y\in \bbZ^d}m(t,y)^{-2n\alpha}  \big|\bar F(y)\big|^{2n}\bigg]^\frac{1}{2n}\\
  \leq \; &  (t+1)^{-(\frac{d}{4}+\frac 12)+\frac{\e}{2}}\,\bigg(\sum_{y \in \bbZ^d}m(t,y)^{-2n\alpha}   \mean\Big[\mathcal
          Z_t^{2n} \, \big|\bar F(y)\big|^{2n}\Big]\bigg)^\frac{1}{2n}.
\end{align*}
By H\"older's inequality with exponents $(\theta,\frac{\theta}{\theta-1})$ and by stationarity this can be further estimated from above by
\begin{align*}
  & (t+1)^{-(\frac{d}{4}+\frac 12)+\frac{\e}{2}}\,\mean\Big[\mathcal Z_t^{2n\frac{\theta}{\theta-1}}\Big]^{\frac{1}{2n}\frac{\theta-1}{\theta}}
\, \mean\Big[ \big|F\big|^{2n\theta}\Big]^\frac{1}{2n\theta}
    \, \bigg(\sum_{y\in \bbZ^d} m(t,y)^{-2n\alpha}\bigg)^{\frac{1}{2n}}\\
  \lesssim\; & (t+1)^{-(\frac{d}{4}+\frac 12)+\frac{\e}{2}+\frac{d}{4n}}\,\mean\Big[\mathcal Z_t^{2n\frac{\theta}{\theta-1}}\Big]^{\frac{1}{2n}\frac{\theta-1}{\theta}}
               \, \mean\Big[ \big|F\big|^{2n\theta}\Big]^\frac{1}{2n\theta},
\end{align*}
where we used that $\big(\sum_{y\in\bbZ^d}m(t,y)^{-2n\alpha}\big)^{\frac{1}{2n}}\lesssim (t+1)^{\frac{d}{4 n}}$ thanks to $2n\alpha>d$. Since we may assume that the $2n\frac{\theta}{\theta-1}$-moment of $\mathcal Z_t$ is bounded by a constant independent of $t$, and because $\frac{d}{4n}\leq \frac{\e}{2}$, the proof is complete.
\end{proof}
\begin{proof}[Proof of Lemma~\ref{L:decay-aux-2}]
%
First we compute
\begin{eqnarray*}
 & &\mean \Bigg[ \Bigg(\sum_{e\in E_d}\Big(\sum_{z\in \bbZ^d} \nabla p(t,z) \cdot H\big(\tau_z\omega,e-z\big)\Big)^2\Bigg)^{\! n} \, \Bigg]^{\frac 1{2n}}\\
   &=&\mean \Bigg[ \Bigg(\sum_{x\in\bbZ^d\atop i=1,\ldots,d}\Big(\sum_{z\in \bbZ^d} \nabla p(t,z)\cdot H\big(\tau_z\omega,\{x{-}z,x{-}z{+}e_i\}\big)\Big)^2\Bigg)^{\! n} \, \Bigg]^{\frac 1{2n}}\\
   &\stackrel{y=x-z}{=}&\mean \Bigg[ \Bigg(\sum_{x\in\bbZ^d\atop i=1,\ldots,d}\Big(\sum_{y\in \bbZ^d} \nabla p(t,x-y)\cdot H(\tau_{x-y}\omega,\{y,y{+}e_i\})\Big)^2\Bigg)^{\! n} \, \Bigg]^{\frac 1{2n}}\\
   &\stackrel{x'=x-y}{\leq}&\sum_{e'\in E_d}\mean \Bigg[ \Bigg(\sum_{x'\in\bbZ^d}\Big(\nabla p(t,x')\cdot H(\tau_{x'}\omega,e')\Big)^2\Bigg)^{\! n} \, \Bigg]^{\frac 1{2n}}.
\end{eqnarray*}
In order to estimate the expectation, we apply Lemma~\ref{L:decay-aux} to $F(\omega)=H(\omega,e')$ for any $e'\in E_d$ and obtain
\begin{equation*}
  \mean\Bigg[ \Bigg(\sum_{e\in E_d}\Big(\sum_{z\in \bbZ^d} \nabla p(t,z) \cdot H(\tau_z\omega,e-z)\Big)^2\Bigg)^n \, \Bigg]^{\frac 1{2n}}\leq c \, \big(t+1\big)^{-(\frac{d}{4}+\frac{1}{2})+\varepsilon}\sum_{e'\in E_d}\mean\Big[\big|H(\omega,e')\big|^{2n\theta}\Big]^{\frac{1}{2n\theta}}.
\end{equation*}
Combined with H\"older's inequality  in form of $\mean\Big[\big|H(\omega,e')\big|^{2n\theta}\Big]^{\frac{1}{2n\theta}}\leq \mean\Big[\big|H(\omega,e')\big|^{4n}\Big]^{\frac{1}{4n}}$ the claim follows.
\end{proof}

\begin{proof}[Proof of Lemma~\ref{lem:cacc}]
  We first claim that
  \begin{equation}\label{lem:cacc:eq1}
    I:=\mean \bigg[\sum_{i=1}^d \big|D_iv(t)\big|^{2n} \, \omega(0,e_i)\bigg] \;  \lesssim \; - \frac 1 {2n}  \, \frac d {dt} \mean \big[ v(t)^{2n} \big],
  \end{equation}
  where here and below we write $\lesssim$ if the relation holds up to a constant only depending on $d$, $n$ and $\theta$. Indeed, using the elementary estimate
  \begin{align*}
    \big(a-b\big)^{2n} \; \lesssim \; \big(a^{2n-1}-b^{2n-1}\big) \, \big(a-b\big), \qquad a,b\in\bbR, 
  \end{align*}
 by appealing to \eqref{eq:hea-prob}  we get
  \begin{align*}
     I \; &  \lesssim \; 
      \sum_{i=1}^d \mean \Big[ \om(0,e_i) \, D_i v(t) \, D_i\big(v(t)^{2n-1}\big)  \Big] \\
    & =\; \mean \Big[ D(v(t)^{2n-1}) \cdot \bomega(0) Dv(t) \Big] 
      \;=\; \mean \Big[ v(t)^{2n-1}  \,  D^*\bomega(0) Dv(t) \Big] \\
    &  =\; - \mean \Big[ v(t)^{2n-1} \, \partial_t v(t) \Big]
      \, = \; - \frac 1 {2n}  \, \frac d {dt} \mean \Big[ v(t)^{2n} \Big],
  \end{align*}
  and thus the claimed inequality \eqref{lem:cacc:eq1}. Next, we need to estimate $\mean\big[\big|Dv(t)\big|^{2n\theta}\big]^\frac{1}{2n\theta}$ by the left-hand side in \eqref{lem:cacc:eq1}. To that end set $\theta_0=2-\theta$, so that $\theta=\theta_0+2(\theta-1)$. H\"older's inequality with exponents $(\frac{1}{\theta_0},\frac{1}{\theta-1})$ yields
  \begin{eqnarray*}
    \mean\left[\big|Dv(t)\big|^{2n\theta}\right]
    &\lesssim&    \sum_{i=1}^d\mean\left[|D_iv(t)|^{2n\theta_0}\omega(0,e_i)^{\theta_0}|D_iv(t)|^{2n(\theta-\theta_0)}\omega(0,e_i)^{-\theta_0}\right]\\
    &\leq&    \left(\sum_{i=1}^d\mean\left[|D_iv(t)|^{2n}\omega(0,e_i)\right]\right)^{\theta_0}\left(\sum_{i=1}^d\mean\left[|D_iv(t)|^{4n}\omega(0,e_i)^{-\frac{2-\theta}{\theta-1}}\right]\right)^{\theta-1}\\
    &\lesssim&    I^{2-\theta}
    \mean\left[|Dv(t)|^{8n}\right]^{\frac{\theta-1}{2}}
    \mean\Big[\big|\bomega(0)^{-1}\big|^{\frac{2(2-\theta)}{\theta-1}}\Big]^{\frac{\theta-1}{2}}.
  \end{eqnarray*}
Using $|Dv(\omega,t)|\lesssim \sum_{|x| \leq 1}|v(\tau_{x}\omega,t)|$, the shift-invariance of $\prob$ and the contractivity of the semigroup $P_t:L^{8n}(\Omega)\to L^{8n}(\Omega)$,  we deduce that
  \begin{equation*}
    \mean\left[|Dv(t)|^{8n}\right]\lesssim \mean\left[|v(t)|^{8n}\right]\leq\mean\left[|u|^{8n}\right].
  \end{equation*}
 In combination with \eqref{lem:cacc:eq1} these estimates give the claim.
\end{proof}

\section{Moment bounds for the extended corrector: Proof of Proposition~\ref{C:corrector}} \label{sec:corr}
In this section, unless stated otherwise, $\xi$ denotes one of the coordinate vectors $e_1,\ldots,e_d$, and we drop the index in the notation for $\phi$, $\sigma$ and $q$. We split the proof of Proposition~\ref{C:corrector} into three steps. In Step~1 we prove (a) and (b), i.e.\ the existence of $\phi$ and $\sigma$, and the moment bounds for $\phi$, which is a rather direct consequence of Theorem~\ref{T:decay}.  In Step~2 we establish a sensitivity estimate for the right-hand side of \eqref{eq:corrector-equation3}. In Step~3 we establish the growth bound for $\sigma$.
\medskip

\step 1 Proof of (a) and (b).

We first claim that there exists $\phi^0\in L^{2n}(\Omega)$ with $\mean[\phi^0]=0$ satisfying \eqref{eq:momentboundphi} such that
\begin{equation*}
  D^*\bomega(0) \big(D\phi^0+\xi\big) \; = \; 0.
\end{equation*}
For the argument set $F:=-\bomega(0)\xi$, $u(t):=P_tD^*F$, and note that
\begin{equation*}
  \mean\bigg[\Big(\sum_{e\in E_d}\big|\partial_e F\big|\Big)^{\! 8(n+1)}\bigg]^\frac{1}{8(n+1)}\; =\; \mean\bigg[\Big(\sum_{i=1}^d \big|\partial_{\{0,e_i\}} F\big|\Big)^{\! 8(n+1)}\bigg]^\frac{1}{8(n+1)} \;\lesssim \; |\xi|=1.
\end{equation*}
Consequently, Theorem~\ref{T:decay} yields  $\mean\Big[\big|u(t)\big|^{2(n+1)}\Big]^\frac{1}{2(n+1)}\lesssim (t+1)^{-\gamma}$
for some $1<\gamma<\frac{d}{4}+\frac{1}{2}$. Since $\gamma>1$, we can define the sought random variable as Laplace transform $\phi^0:=\int_0^\infty u(t)\,dt$.\\

Next, we set $\phi(\omega,x):=\phi^0(\tau_x\omega)-\phi^0(\omega)$. By construction we have $\phi(\omega,0)=0$. Since $D_i$ is the discrete generator of the shift $\tau_{e_i}$, we deduce that
\begin{equation*}
  \nabla\phi(\omega,x)=D\phi(\tau_x\omega),\qquad \nabla^*\bomega(x)(\nabla\phi(\omega,x)+\xi)=\Big(D^*\bomega(0)(D\phi^0+\xi)\Big)(\tau_x\omega)=0,
\end{equation*}
and we conclude that $\phi$ satisfies all the claimed properties.\\

Next, we prove the existence of $\sigma$. To that end we first rewrite the right-hand side in \eqref{eq:corrector-equation3} in divergence form. For $k,\ell=1,\ldots d$ we introduce the random vector fields
\begin{eqnarray*}
  q^0(\omega)&:=&\bomega(0)(D\phi^0(\omega)+\xi)-\bomega_{\hom}\xi,\\
  Q_{k\ell}^0(\omega)&:=&\big(q^0(\tau_{\ell}\omega)\cdot e_k\big) \, e_\ell- \big(q^0(\tau_{e_k}\omega)\cdot e_\ell\big) \, e_k,
\end{eqnarray*}
and denote by $Q_{k\ell}(\omega,x):=Q^0_{k\ell}(\tau_x\omega)$ the stationary extension. Note that by construction we have $q(\omega,x)=q^0(\tau_x\omega)$ and $(Sq)_{k\ell}=\nabla^*Q_{k\ell}$. From the moment bound \eqref{eq:momentboundphi} and moment condition on the conductances we deduce that $q^0$ and $Q^0$ have finite second moments. Therefore, for any $T\geq 1$ the regularized equation $(\frac{1}{T}+D^*D)\sigma_{T,k\ell}^0=D^*Q_{k\ell}^0$ admits a unique solution in $L^2(\Omega)$ satisfying the a priori estimate
\begin{equation*}
  \frac{1}{T}\mean\Big[\big|\sigma_{T,k\ell}^0\big|^2\Big]+\frac{1}{2}\mean\Big[ \big|\nabla\sigma_{T,k\ell}^0|^2\Big] \;\leq \; \frac{1}{2}\mean\Big[\big|Q^0_{k\ell}\big|^2\Big] \; \lesssim \; 1.
\end{equation*}
Since the estimate on $D\sigma^0_T$ is uniform in $T\geq 1$, we deduce that (up to a subsequence) $D\sigma^0_{T,k\ell}$ weakly converges for $T\uparrow\infty$ to some random variable $W_{k\ell}=(W_{k\ell,1},\ldots,W_{k\ell,d})$ with $\mean[|W_{k\ell}|^2]\leq \frac{1}{2}\mean\Big[\big|Q^0_{k\ell}\big|^2\Big]$ satisfying
\begin{align}\label{eq:pfcorr:2}
  D^*W_{k\ell} \; = \; D^*Q^0_{k\ell}.
\end{align}
Note that $\mean[W_{k\ell}]=0$, since $\mean[D\sigma^0_T]=0$. Moreover, $W_{k\ell}$ is curl-free in the sense that $D_iW_{k\ell,j}=D_jW_{k\ell,i}$. Hence, there exists a random field $\sigma_{T,k\ell}(\omega,x)$ with $\sigma_{T,k\ell}(\omega,0)=0$ and $\nabla\sigma_{T,k\ell}(\omega,x)=W_{k\ell}(\tau_x\omega)$, and \eqref{eq:pfcorr:2} turns into $\nabla^*\nabla\sigma_{T,k\ell}=\nabla^*Q_{k\ell}$, which is \eqref{eq:corrector-equation3} due to the definition of $Q$. Moreover, $\sigma_{k\ell}$ is skew-symmetric in $k,\ell$, since so is $(Sq)_{k\ell}$. Moreover, \eqref{eq:corrector-equation2} follows from the identity $\nabla^*\nabla(\nabla^*_\ell\sigma_{k\ell}-q_k)=0$ (see e.g.\ \cite[Proof of Lemma 9, Step 2]{BMN17}).
\medskip

\step 2 Sensitivity estimate.

Let $f:\bbZ^d\to\bbR$ be compactly supported and  $Q$ be defined as in Step~1. Consider exponents $s,r$ such that
\begin{equation}\label{eq:pfcorr:3}
s>\frac{d}{d-1},\qquad r\geq 1,\qquad 1+\frac{1}{2}=\frac{1}{s}+\frac{1}{r}.
\end{equation}
Then for any $p>1$ satisfying
\begin{equation}\label{eq:pfcorr5a}
  s \, \frac{d-1}{d}-\frac{p-1}{p} \, \frac{1}{d} \; > \;1
\end{equation}
(i.e. for $0<p-1\ll1 $), we have
\begin{equation*}
  I\; := \; \mean\Bigg[\bigg(\sum_{e\in E_d}\Big(\sum_{y\in\bbZ^d} \big|\nabla f(y)\big| \, \big|\partial_{e}Q_{k\ell}(y)\big|\Big)^{\!2}\bigg)^{\!p}\Bigg]^{\frac{1}{2 p}} \; \lesssim \;  \big\|\nabla f\big\|_{\ell^{r}(\bbZ^d)}.
\end{equation*}

This can be seen as follows. By the triangle inequality (w.r.t.\ $\| \cdot \|_{L^p(\Omega, \prob)}$), by expanding the square, another application of the triangle inequality, and the Cauchy Schwarz inequality, 
  \begin{eqnarray*}
    I^2  &\leq&\sum_{e\in E_d}\mean\bigg[\Big(\sum_{y\in\bbZ^d} \big|\nabla f(y)\big| \, \big|\partial_{e}Q_{k\ell}(y)\big|\Big)^{\!2p}\bigg]^{\frac{1}{p}}\\
      &=&
      \sum_{e\in E_d}\mean\bigg[\Big(\sum_{y,y'\in\bbZ^d} \big|\nabla f(y)\big|\, \big|\nabla f(y')\big| \, \big|\partial_{e}Q_{k\ell}(y)\big| \big|\partial_{e}Q_{k\ell}(y')\big|\Big)^{\!p}\bigg]^{\frac{1}{p}}\\
      &\leq&
      \sum_{e\in E_d}\sum_{y,y'\in\bbZ^d} \big|\nabla f(y)\big| \, \big|\nabla f(y')\big| \, \mean\bigg[\Big( \big|\partial_{e}Q_{k\ell}(y)\big| \, \big|\partial_{e}Q_{k\ell}(y')\big|\Big)^{\!p}\bigg]^{\frac{1}{p}}\\
      &\leq&
      \sum_{e\in E_d}\sum_{y,y'\in\bbZ^d} \big|\nabla f(y)\big| \, \big|\nabla f(y')\big| \, \mean\Big[\big|\partial_{e}Q_{k\ell}(y)\big|^{2p}\Big]^{\frac{1}{2p}} \, \mean\Big[ \big|\partial_{e}Q_{k\ell}(y')\big|^{2p}\Big]^{\frac{1}{2p}}\\
      &=&
      \sum_{e\in E_d}\bigg(\sum_{y\in\bbZ^d} \big|\nabla f(y)\big| \, \mean\Big[ \big|\partial_{e}Q_{k\ell}(y)\big|^{2p}\Big]^{\frac{1}{2p}}\bigg)^{\!2}.
  \end{eqnarray*}
  For the following discussion it is convenient to set
  \begin{equation*}
    g(z):=\sum_{e=\{z,z+e_i\}\atop i=1,\ldots,d}\mean\Big[ \big|\partial_{e}Q_{k\ell}^0\big|^{2p}\Big]^{\frac{1}{2p}}.
  \end{equation*}
  Since $Q_{k\ell}(\omega,y)=Q^0_{k\ell}(\tau_y\omega)$ we infer that $    \partial_e Q_{k\ell}(\omega,y)=(\partial_{e-y}Q^0)(\tau_y\omega)$. Since $\prob$ is stationary, we obtain $\mean\big[|\partial_{e}Q_{k\ell}(y)|^{2p}\big]^{\frac{1}{2p}}=\mean\big[|\partial_{e-y}Q_{k\ell}^0|^{2p}\big]^{\frac{1}{2p}}\leq g(\ue-y)$, and thus
  \begin{equation*}
    I \; \leq \; \bigg(\sum_{z\in\bbZ^d}\Big(\sum_{y\in\bbZ^d} \big|\nabla f(y)\big| \,  g(z-y)\Big)^{\!2}\bigg)^{\frac{1}{2}}.
  \end{equation*}
  The inner sum is a convolution. Thanks to \eqref{eq:pfcorr:3} Young's estimate for convolutions yields
  \begin{equation*}
    I \; \leq \; \|\nabla f\|_{\ell^{r}(\bbZ^d)} \, \|g\|_{\ell^{s}(\bbZ^d)},
  \end{equation*}
  and it remains to show that the norm of $g$ is finite. It suffices to show that for some $\gamma>0$ with $\gamma>\frac{d}{s}$ we have
  \begin{equation}\label{eq:pfcorr3b}
    \mean\Big[ \big|\partial_{e}Q_{k\ell}^0\big|^{2p}\Big]^{\frac{1}{2p}} \; \lesssim \;  \big(|\ue|+1\big)^{-\gamma}, \qquad  \forall e\in E_d.
  \end{equation}
  We first note that
  \begin{equation}\label{eq:pfcorr4}
    \big|\partial_e D\phi^0(\omega)\big| \; \leq \; \big|\nabla\nabla G^\omega(0,\ue)\big| \, \big|D\phi^0(\tau_{\ue} \omega)+\xi\big|.
  \end{equation}
  Indeed, this follows from  applying $\partial_e$ to \eqref{eq:corrector-equation1}, which yields
  \begin{equation}\label{eq:pfcorr5}
    \nabla^*\omega\nabla\partial_e\phi \; = \; \nabla^*\big(\partial_e\bomega(\cdot)\big) \, \big(\nabla\phi+\xi\big).
  \end{equation}
  Since $|\partial_e\bomega(\cdot)|\leq\delta(\cdot-\ue)$, and $\partial_e D\phi^0(\omega)=\nabla\partial_e\phi(\omega,0)$, we obtain \eqref{eq:pfcorr4} by the Green's function representation for \eqref{eq:pfcorr5}. Hence, applying $\partial_e$ to $Q^0$ gives
  \begin{equation*}
    \partial_e Q^0(\omega) \; = \; \delta(\ue)\big(D\phi^0(\omega)+\xi\big)-\bomega(0) \, \partial_e D\phi^0(\omega),
  \end{equation*}
  and thus
  \begin{equation*}
    \big|\partial_e Q^0(\omega)\big| \; \leq \;  \delta(\ue) \, \big(|D\phi^0|+1\big) + |\bomega(0)| \, \big|\nabla\nabla G(\omega,0,\ue)\big| \, \big(|D\phi^0(\tau_{\ue}\omega)|+1\big).
  \end{equation*}
  Since we may assume that high moments of $D\phi^0$ and $\bomega(0)$ exist, Corollary~\ref{L:AG} yields for any $p'>p$ and any $\e>0$,
  \begin{equation*}
    \mean\Big[\big|\partial_e Q^0\big|^{2p}\Big]^{\frac{1}{2p}} \; \lesssim \; \mean\Big[\big|\nabla\nabla G(0,\ue)\big|^{2p'}\Big]^{\frac{1}{2p'}} \; \lesssim \; \big(|\ue|+1\big)^{-\gamma},\qquad \gamma:=d-2\frac{2p'-1}{2p'}-\e,
  \end{equation*}
  provided the conductances satisfy sufficiently strong moment conditions with integrability exponents that depend on $p'$ and $\e$. Note that for $p'\downarrow p$ and $\e\downarrow 0$, we have $\gamma\uparrow \gamma_p:=d-\frac{2p-1}{p}=(d-1)-\frac{p-1}{p}$. Thanks to \eqref{eq:pfcorr5a} we have $\gamma_p\frac{s}{d}>1$, and thus we obtain $\gamma>\frac{d}{s}$ by choosing $p'$ and $\e$ sufficiently close to $p$ and $0$. This completes the argument for \eqref{eq:pfcorr3b}.
\medskip

\step 3 Sublinear estimate for $\sigma$.
\smallskip

We basically follow arguments in \cite{GNOest}, where a similar statement is obtained for uniformly elliptic, continuous systems, and \cite{Nlecture}, where the argument is carried out for the corrector $\phi$ in the uniformly elliptic, discrete setting in dimension $d=2$. The argument is split into three substeps.
\smallskip

\substep{3.1} For $L\geq 1$ consider
\begin{equation*}
  v_L(\omega):=\sigma(\omega,0)-\frac{1}{\# B(L)}\sum_{y\in B(L)}\sigma(\omega,y).
\end{equation*}
Then for any exponents $r$ and $p$ satisfying \eqref{eq:pfcorr:3} and \eqref{eq:pfcorr5a}, we have
\begin{align*}
  \mean\Big[ \big|v_L\big|^{2p}\Big]^{\frac{1}{2p}} \; \lesssim \;     \begin{cases}
    L^{\frac{d}{r}+(1-d)}& \text{if $r<\frac{d}{d-1}$,}\\
    (\log L)^{\frac{1}{r}}& \text{if $r=\frac{d}{d-1}$,}\\
    1& \text{if $r>\frac{d}{d-1}$.}
  \end{cases}
\end{align*}
For the argument let $f:\bbZ^d\to\bbR$ denote the unique decaying solution to
\begin{equation*}
  \nabla^*\nabla f=h,\qquad \text{where }h:=\delta_0 -\frac{1}{\# B(L)}\indicator_{B(L)}.
\end{equation*}
By representing $\nabla f$ with help of the discrete Green's function for $\nabla^*\nabla$ we find that
\begin{equation}\label{eq:pfcorr:7a}
  \big|\nabla f(y)\big| \; \lesssim \;  \big(L\wedge|y|\big) \, \big(|y|+1\big)^{-d}.
\end{equation}
Since $\sigma(\omega,\cdot)$ grows sublinearly $\prob$-a.s., we deduce that
\begin{equation}\label{eq:pfcorr:8}
  v_L(\omega) \; = \; \sum_{y\in\bbZ^d}\sigma(\omega,y) \, h(y) \; = \;\sum_{y\in\bbZ^d}\nabla\sigma(\omega,y)\cdot\nabla f(y).
\end{equation}
In particular, we find that $\mean[v_L]=0$, since $\mean[\nabla\sigma]=0$. Thus, the $p$-version of the Spectral Gap estimate, \eqref{eq:sg_p}, yields
\begin{equation}\label{eq:pfcorr:10}
  \mean\Big[ \big|v_L\big|^{2p}\Big]^{\frac{1}{2p}}\lesssim \mean\Big[\big(\sum_{e\in E_d}\big|\partial_e v_L\big|^2\big)^p\Big]^{\frac{1}{2p}}.
\end{equation}
Note that $\partial_e v_L=\sum_{y\in\bbZ^d}\nabla\partial_e\sigma(y)\cdot \nabla f(y)=\sum_{y\in\bbZ^d}\partial_e Q(y)\cdot\nabla f(y)$, where we used that
\begin{equation*}
  \nabla^*\nabla\partial_e\sigma=\nabla^*\partial_e Q.
\end{equation*}
Hence, we can estimate the right-hand side of \eqref{eq:pfcorr:10} by appealing to Step 2 and \eqref{eq:pfcorr:7a}. This completes the argument.
\smallskip

\substep{3.2} For $x\in\bbZ^d$ and $L\geq 1$ consider
\begin{equation*}
  v'_L(\omega):=\frac{1}{\# B(L)}\sum_{y\in B(L)}\big(\sigma(\omega,x+y)-\sigma(\omega,y)\big).
\end{equation*}
Then for any $r$ and $p$ satisfying \eqref{eq:pfcorr:3} and \eqref{eq:pfcorr5a}, we have
\begin{equation*}
  \mean\Big[\big|v'_L\big|^{2p}\Big]^{\frac{1}{2p}} \; \lesssim \; |x|\, L^{\frac{d}{r}-d} \, \log L.
\end{equation*}
For the argument let $f:\bbZ^d\to\bbR$ denote the unique decaying solution to
\begin{equation*}
  \nabla^*\nabla f\; = \; h,\qquad \text{where }h:=\frac{1}{\# B(L)} \Big(\indicator_{B(L)}(\cdot-x)-\indicator_{B(L)}(\cdot)\Big).
\end{equation*}
By representing $\nabla f$ with help of the discrete Green's function for $\nabla^*\nabla$ we find that
\begin{equation}\label{eq:pfcorr:7}
  \big|\nabla f(y)\big| \; \lesssim \; |x| \, \big(L\vee|y|\big)^{-d} \, \log L,
\end{equation}
and, moreover,
\begin{align*}
  v'_L(\omega)\; = \;\sum_{y\in\bbZ^d}\sigma(\omega,y) \, h(y)\; = \;\sum_{y\in\bbZ^d}\nabla\sigma(\omega,y)\cdot\nabla f(y).
\end{align*}
Hence, arguing as in Substep~{3.2} we see that the claim follows from Step~2.
\smallskip

\substep{3.3} Let $L:=|x|+1$. Then we have
\begin{eqnarray*}
  \sigma(\omega,x)&=&\sigma(\omega,x)-\sigma(\omega,0)\\
  &=&\sigma(\omega,x)-\frac{1}{\#B(L)}\sum_{y\in B(L)}\sigma(\omega,x+y)+
           \frac{1}{\#B(L)}\sum_{y\in B(L)}\big(\sigma(\omega,x+y)-\sigma(\omega,y)\big)\\
           &&+\frac{1}{\#B(L)}\sum_{y\in B(L)}\sigma(\omega,y) -\sigma(\omega,0)\\
  &=&v_L(\tau_x\omega)+v'_L(\omega)-v_L(\omega),
\end{eqnarray*}
where the last identity holds due to the identities of the previous steps, and by stationarity of $\nabla\sigma$ in combination with identity \eqref{eq:pfcorr:8}. Hence, the triangle inequality, $|x|\leq L$, and the estimates of Substep 3.1 und Substep 3.2 yield for any exponents $r$ and $p$ satisfying \eqref{eq:pfcorr:3} and \eqref{eq:pfcorr5} the estimate
\begin{align*}
  \mean\Big[ \big|\sigma(x)\big|^{2p}\Big]^{\frac{1}{2p}}
  &\lesssim \, L^{\frac{d}{r}-d+1}\log L + \begin{cases}
    L^{\frac{d}{r}+(1-d)}& \text{if $r<\frac{d}{d-1}$},\\
    \big(\log L\big)^{\frac{1}{r}}& \text{if $r=\frac{d}{d-1}$},\\
    1& \text{if $r>\frac{d}{d-1}$}.
  \end{cases}
\end{align*}
In dimensions $3\leq d\leq 4$ any exponent $1\leq r<\frac{2d}{d+2}\leq \frac{d}{d-1}$ is admissible. Since the upper bound $r\uparrow \frac{2d}{d+2}$ implies  $\frac{d}{r}+(1-d)\downarrow 2-\frac{d}{2}$, the claimed statement follows. On the other hand, in dimension $d\geq 5$, we might choose any exponent $\frac{d}{d-1}<r<\frac{2d}{d+2}$, which completes the argument.
\qed

\section{Variance decay for the  carr\'e du champ: Proof of Proposition~\ref{P:decay}}\label{S:P:decay}
A key ingredient in the proof of Proposition~\ref{P:decay} is the following lemma, which for $d\geq 3$ yields a representation of $g_\xi$ in divergence form.

\begin{lemma}\label{L:QV}
  Consider the situation of Proposition~\ref{C:corrector}. Let $\xi\in\bbR^d$ be fixed and let $(\phi,\sigma)$ denote the associated extended corrector, i.e.\ $(\phi,\sigma):=\sum_{i=1}^d\xi_i \,(\phi_i,\sigma_i)$ with $(\phi_i,\sigma_i)$ as in Proposition~\ref{C:corrector} (a). Consider
  \begin{equation*}
    g(\omega):=\sum_{y\in\bbZ^d}\omega(0,y)\psi(\omega,y)^2\qquad\text{where } \psi(\omega,y):=\xi\cdot y+\phi(\omega,y)-\phi(\omega,0).
  \end{equation*}
  Then,
  \begin{eqnarray*}
    g(\omega)-\mean[g]\;= \; g(\omega)-2\xi\cdot\bomega_{\hom}\xi\; =\;\nabla^*H(\omega,0),
  \end{eqnarray*}
  where $H=(H_{1},\ldots, H_{d})$  is defined by
  \begin{align*}
    H_i(\omega,x)\;\ldef \; & \omega(x,x+e_i) \, \big(\xi_i+\nabla_i\phi(\omega,x)\big)^2 \\
    & +2 \, \Big(\sigma(\omega,x)^t\xi+\phi(\omega,x+e_i) \,\bomega(x)\, \big(\xi+\nabla\phi(\omega,x)\big)\Big)\cdot e_i.
  \end{align*}
  %
\end{lemma}

\begin{proof}
  Define 
  \begin{equation*}
    H'_i(\omega,x):=\omega(x,x+e_i)\big(\xi_i+\nabla_i\phi(\omega,x))^2\,\qquad
    H''_i(\omega,x):=\phi(\omega,x+e_i)\big(e_i\cdot\bomega(x)(\xi+\nabla\phi(\omega,x))\big),
  \end{equation*}
  and note that $H_i=H_i'+2(\sigma^t\xi)\cdot e_i+2H_i''$. 
  We have $\psi(\omega,e_i)=\xi_i+\nabla_i\phi(\omega,0)$ and $\psi(\omega,-e_i)=-\xi_i-\nabla_i\phi(\omega,-e_i)=-\xi_i-\nabla_i\phi(\tau_{-e_i}\omega,0)$, thanks to stationarity of $\nabla\phi$ (see Proposition~\ref{C:corrector} (a.3)).  Hence,
  \begin{align}\label{L:QV:eq000} \nonumber
    g(\omega) &= \;\sum_{i=1}^d\omega(0,e_i)\, \big(\xi_i+\nabla_i\phi(\omega,0)\big)^2+\sum_{i=1}^d\omega(-e_i,0)\,
    \big(-\xi_i-\nabla_i\phi(\omega,-e_i)\big)^2\\\nonumber
   & = \;\sum_{i=1}^dH_i'(\omega,0)+\sum_{i=1}^dH_i'(\tau_{-e_i}\omega,0)\;
    =\;2 \, \Big(\sum_{i=1}^dH_i'(\omega,0)\Big)+D^*H'(\omega,0)\\
    & = \;2 \, \Big(\sum_{i=1}^dH_i'(\omega,0)\Big)+\nabla^*H'(\omega,0),
  \end{align}
  where the last identity holds, since $H'$ is stationary in the sense that $H'(\tau_x\omega,y)=H'(\omega,y+x)$.
  With $q(\omega,x):=\bomega(x)\, \big(\xi+\nabla\phi(\omega,x)\big)-\bomega_{\hom}\xi$, we can rewrite the first term on the right-hand side  as
  \begin{align} \label{eq:Hprime}
    \sum_{i=1}^dH_i'(\omega,0)&=\; \big(\xi+\nabla\phi(\omega,0)\big)\cdot\bomega(0)\, \big(\xi+\nabla\phi(\omega,0)\big) \nonumber\\
    &= \; \xi\cdot q(\omega,0)+\nabla\phi(\omega,0)\cdot \bomega(0)\, \big(\xi+\nabla\phi(\omega,0)\big)+\xi\cdot\bomega_{\hom}\xi.
  \end{align}
  In view of Proposition~\ref{C:corrector}, the first term takes the form
  \begin{equation}\label{rep2}
    \xi\cdot q(\omega,0)=\nabla^*(\sigma^t\xi)(\omega,0).
  \end{equation}
  For the second term on the right-hand side of \eqref{eq:Hprime} we use the general discrete product rule
  \begin{equation*}
    \nabla \phi\cdot F=\nabla^*[\phi,F]-\phi \,(\nabla^*F),\qquad [\phi,F]_i:=\phi(\cdot+e_i) \, F_i,
  \end{equation*}
  which we apply with $F(x)=\bomega(x) \, \big(\xi+\nabla\phi(\omega,x)\big)$. By the corrector equation we have $\nabla^*F=0$ and therefore 
  \begin{eqnarray*}
    \nabla\phi(\omega,\cdot)\cdot \bomega(\cdot)\, \big(\xi+\nabla\phi(\omega,\cdot)\big)\;=\; \nabla^*H''(\omega,\cdot).
  \end{eqnarray*}
  Now the claimed representation follows from \eqref{L:QV:eq000}, \eqref{eq:Hprime} and \eqref{rep2}, and the fact that the mean of the right-hand side in \eqref{eq:Hprime} is $\xi\cdot\bomega_{\hom}\xi$.
\end{proof}

\begin{proof}[Proof of Proposition~\ref{P:decay}]
  First we recall that
  \begin{align*}
    (\Gamma^\om f)(x)& \ldef \left[\cL^\om f^2 - 2 f \cL^\om f \right](x)
    =\sum_{y\in \bbZ^d} \om(x,y) \left( f(y)-f(x)\right)^2,
  \end{align*}
  which holds for any $f:\bbZ^d\to\bbR$ as can be seen by a direct calculation. In order to recover $g_\xi(\omega)$, we need to consider $f(x)=\psi_\xi(\omega,x)$ and evaluate at $x=0$.   By the definition of $\psi_\xi$ we have $\psi_\xi(\omega,0)=0$, and thus Lemma~\ref{L:QV} yields
  \begin{align}\label{L:QV:eq000} \nonumber
    g_\xi(\omega)&= \; \Gamma^\omega(\psi_\xi(\omega,\cdot))(0) \; = \; \sum_{|y|=1}\omega(0,y) \, \psi_\xi(\omega,y)^2 \\
    &=\mean[g_\xi]+\nabla^*H(\omega,0),
    \nonumber
  \end{align}
  where $H$ is defined as in Lemma~\ref{L:QV}. Note that $H(y):=H(\om,y)=H(\tau_y \om, 0)$. Thus, by the definition of $P_t$, an integration by parts, H\"older's inequality, and Proposition~\ref{P:HK}, we get for all $\alpha>d$ and $0<\e<1$, 
  \begin{eqnarray*}
    I&:=&\mean\bigg[\Big(P_t \big(g_\xi-\mean[g_\xi]\big)\Big)^{\!2}\bigg]^{\frac12}=    \mean\bigg[\Big(\sum_{y\in\bbZ^d} p^\om(t,0,y) \,\nabla^*H(y)\Big)^{\!2}\bigg]^{\frac12}\\
    &=&\mean\bigg[\Big(\sum_{y\in\bbZ^d}\nabla p(t,y)\cdot H(y)\Big)^{\!2}\bigg]^{\frac12}\\
    &\leq&\mean\bigg[\Big(\sum_{y\in\bbZ^d} \big|\nabla p(t,y)\big|^2 \, m(t,y)^{\alpha}\Big) \, \Big(\sum_{y\in\bbZ^d} \big|H(y)\big|^2 \, m(t,y)^{-\alpha}\Big)\bigg]^{\frac12}\\
    &\leq&(t+1)^{-(\frac{d}{4}+\frac{1}{2})+\e} \, \mean\bigg[\mathcal Z_t^2 \, \Big(\sum_{y\in\bbZ^d} \big|H(y)\big|^2 m(t,y)^{-\alpha}\Big)\bigg]^{\frac12}\\
    &=&(t+1)^{-(\frac{d}{4}+\frac{1}{2})+\e}\, \bigg(\sum_{y\in\bbZ^d}m(t,y)^{-\alpha}\mean\Big[\mathcal Z_t^2 \, \big|H(y)\big|^2\Big]\bigg)^{\frac12}.
  \end{eqnarray*}
  From the definition of $H$, the moment bounds of Proposition~\ref{C:corrector}, and the property that high moments of $\mathcal Z_t$ are bounded uniformly in $t$ (see Proposition~\ref{P:HK}), we deduce that for any $\theta>\frac{1}{2}$ in dimension $d=3$ and any $\theta>0$ in dimension $d\geq 4$, we have
  \begin{equation*}
    \sup_{t\geq 0}\mean\Big[\mathcal Z_t^2\, \big|H(y)\big|^2\Big] \; \lesssim \; (|y|+1)^{2\theta}\leq (t+1)^{\theta}m(t,y)^{2\theta},
  \end{equation*}
  and thus
  \begin{equation*}
    I\lesssim (t+1)^{-(\frac{d}{4}+\frac{1}{2})+\e+\frac{\theta}{2}}\, \bigg(\sum_{y\in\bbZ^d}m(t,y)^{-\alpha+2\theta}\bigg)^{\frac12}.
  \end{equation*}
  Since $\sum_{y\in\bbZ^d}m(t,y)^{-\alpha+2\theta} \lesssim (t+1)^{\frac{d}{2}}$ (whenever $\alpha-2\theta>\frac{d}{2}$), we conclude that
  \begin{equation*}
    I \; \lesssim \; (t+1)^{-\frac{1}{2}+\e+\frac{\theta}{2}}.
  \end{equation*}
  Note that in dimension $d\geq 4$ (resp. $d=3$) we might choose $\theta$ arbitrarily close $0$ (resp. $\frac{1}{2}$), while we can choose $\e>0$ as small as we wish. This completes the proof.
\end{proof}

\section{Berry-Esseen Theorem} \label{S:be}
In this section we prove Theorem~\ref{thm:main}. As mentioned earlier we will show a Berry-Esseen  theorem for the martingale part and afterwards derive that the corrector converges sufficiently fast to zero along the path of the random walk. We start by setting-up the decomposition of $X$.  
Recall the definition of the corrector in Proposition~\ref{C:corrector} and of the harmonic coordinates $\Psi: \Omega \times \bbZ^d \rightarrow \bbR^d$ in \eqref{eq:harm_coord} above. Let $\chi:\Omega \times \bbZ^d \rightarrow \bbR^d$ be defined by 
\begin{align*}
\chi(\om,x)\ldef \Psi(\om,x) -x.
\end{align*}

\begin{remark} \label{rem:cocycle}
The harmonic coordinates satisfy the \emph{cocycle property}, that is for $\prob$-a.e.\ $\om$,
\begin{align*}
\Psi(\om,y)-\Psi(\om,x)=\Psi(\tau_x\om,y-x), \qquad \forall x,y\in \bbZ^d,
\end{align*}
and a similar relation holds for $\chi$.
 Note that conventions about the sign of the corrector $\chi$ differ, compare for instance \cite{ADS15, Bi11} and \cite{SS04}.
\end{remark}

\begin{corro} \label{cor:qV}
For $\prob$-a.e.\ $\om$, the process
\begin{align*}
M_t \; \ldef \; \Psi(X_t), \qquad t\geq 0,
\end{align*}
is a $P_0^\om$-martingale and
\begin{align} \label{eq:decompX}
X_t\; = \; M_t- \chi(\om,X_t), \qquad t\geq 0.
\end{align}
Moreover, for every $\xi\in \bbR^d$, $\xi\cdot M$ is a $P^\om_0$-martingale with quadratic variation process given by
\begin{align} \label{eq:vM_qv}
\langle \xi\cdot M \rangle_t  \; = \;  \int_0^t  \sum_{y \in \bbZ^d} \om(X_s,y) \big(\psi_\xi(\om,y)- \psi_\xi(\om,X_s) \big)^2 \, ds 
  \; = \; 
 \int_0^t \Gamma^{\om} \big(\psi_\xi(\tau_{X_s}\om, \cdot)\big)(0) \, ds.
%
\end{align}
\end{corro}
\begin{proof}
Clearly, $\nabla^* \om \nabla \Psi_i=0$ for every $i=1,\ldots d$, so $\Psi_i$ is $\mathcal{L}^\om$-harmonic. In particular, $M$ and hence also $\xi\cdot M$ are $P_0^\om$-martingales. The decomposition \eqref{eq:decompX} follows directly from the definition of $\chi$. 
To compute $\langle \xi\cdot M \rangle$, which is the unique predictable process 
such that $(\xi\cdot M)^2-\langle \xi\cdot M \rangle$ is a martingale, recall that
\begin{align*}
\langle \xi\cdot M \rangle_t \;=\; \int_0^t \Gamma^\om(\psi_\xi(\om, \cdot))(X_s) \, ds,
\end{align*}
where $\Gamma^\om$ still denotes the op\'erateur carr\'e 
du champ associated with  $\mathcal{L}^\om$ given by
\begin{align*}
 \Gamma^\om(f)(x) \; \ldef \; \left[\cL^\om f^2 - 2 f \cL^\om f \right](x)
\;= \;\sum_{y\in \bbZ^d} \om(x,y) \, \big( f(y)-f(x)\big)^2,
\end{align*}
and $\psi_\xi=\xi\cdot \Psi$ as defined in \eqref{eq:harmonic-coordinate}.
Finally, since $\psi_\xi(\om,0)=0$ we observe that for any $x\in \bbZ^d$,
\begin{align*}
\sum_{y \in \bbZ^d}\omega(x,y) \big(\psi_\xi(\om,y)- \psi_\xi(\om,x) \big)^2 \; = \; \sum_{y \in \bbZ^d} (\tau_x \om)(0,y) \, \psi_\xi(\tau_x\om,y)^2 
\; = \; \Gamma^{\om} \big(\psi_\xi(\tau_x\om, \cdot)\big)(0), 
\end{align*}
which completes the proof.
\end{proof}

\subsection{Quantitative CLT for the martingale part}
In this section we show the following Berry-Esseen theorem for the martingale part. 
\begin{prop} \label{prop:beM}
  Let $d \geq 3$ and suppose that Assumption \ref{ass:sg} holds. For any $\varepsilon>0$ there exist exponents $p,q\in [1,\infty)$ (only depending on $d$, $\rho$ and $\varepsilon$) such that under the moment condition $M(p,q)<\infty$ the following hold. 
 \begin{enumerate}
 \item[(i)] There exists a constant $c=c(d,\rho, \varepsilon, M(p,q))$ such that for all $t\geq 0$, 
    \begin{align*}
  \sup_{x \in \bbR} \Big|  \prob_0 \big[\xi \cdot M_t \leq \sigma_\xi x \sqrt{t} \ \big] - \Phi(x)\Big|  \; \leq \;
  \begin{cases}
   \; c \, t^{- \frac 1 {10} + \varepsilon} & \text{if $d=3$,} \\
    \; c \, t^{-\frac 1 5 + \varepsilon} & \text{if $d\geq 4$.}
  \end{cases}
  \end{align*}
\item [(ii)] There exists a random  $\cX=\cX(d,\rho, \varepsilon, M(p,q))$ in $L^1(\prob)$ such that  for $\prob$-a.e.\ $\om$, 
    \begin{align*}
  \int_0^\infty  \Big(  \sup_{x \in \bbR} \left|  P_0^\om\big[v \cdot M_t \leq \sigma_\xi x \sqrt{t} \ \big] - \Phi(x)\right| \Big)^5  (t+1)^{-\frac  1 2-\varepsilon} \, dt  \; & \leq \; \cX(\om) \; < \; \infty, \qquad \text{if $d=3$,}
  \intertext{and}
  \int_0^\infty  \Big(  \sup_{x \in \bbR} \left|  P_0^\om\big[v \cdot M_t \leq \sigma_\xi x \sqrt{t} \ \big] - \Phi(x)\right| \Big)^5  (t+1)^{-\varepsilon} \, dt  \; & \leq \; \cX(\om) \; < \; \infty, \qquad \text{if $d\geq 4$.}
  \end{align*}
 \end{enumerate}  
\end{prop}

To prove Proposition~\ref{prop:beM} we will apply the following general quantitative central limit theorem for martingales.

\begin{theorem} \label{thm:be_hb}
 Let $(N_t)_{t\geq 0}$ be a locally square-integrable martingale (w.r.t.\ some probability measure $P$) and denote by $\Delta N_t:=N_t-N_{t-}$ its jump process and by $\langle N \rangle_t$ its quadratic variation process. 
Then, for any $n>1$, there exists a constant $c>0$ depending only on $n$ such that
\begin{align*}
 \sup_{x \in \bbR} \left|  P\big[N_1 \leq  x  \ \big] - \Phi(x)\right| \;\leq \; c \, \Big( E\big[ \big| \langle N\rangle_1 -1\big|^n \big]+ E \Big[ \sum_{0\leq t \leq 1} \big| \Delta N_t\big|^{2n} \Big]\Big)^{1/(2n+1)}.
\end{align*}
\end{theorem}
\begin{proof}
 See Theorem 2 in \cite{Ha88} (cf.\ also \cite{HB70}).
\end{proof}

\begin{prop} \label{prop:estV}
  Let $d \geq 3$ and suppose that Assumption \ref{ass:sg} holds. For any $\varepsilon>0$ there exist exponents $p,q\in [1,\infty)$ (only depending on $d$, $\rho$ and $\varepsilon$) such that if $M(p,q)<\infty$ the following hold.
\begin{enumerate}
\item[(i)] There exists a constant $c=c(d,\rho, \varepsilon, M(p,q))$ such that for all $t > 0$, 
\begin{align*}
\mean \Bigg[ E_0^\om \bigg[\Big| \frac{\langle \xi\cdot M \rangle_t} t -\sigma_\xi^2 \Big|^2 \bigg] \Bigg] \; \leq \;   \begin{cases}
   \; c \, t^{- \frac 1 2 + \varepsilon} & \text{if $d=3$,} \\
    \; c \, t^{-1 + \varepsilon} & \text{if $d\geq 4$.}
  \end{cases}
\end{align*}
\item[(ii)]There exists a random  $\cX=\cX(d,\rho, \varepsilon, M(p,q))$ in $L^1(\prob)$ such that  for $\prob$-a.e.\ $\om$,
 \begin{align*}
  \int_0^\infty   E_0^\om \bigg[\Big| \frac{\langle \xi\cdot M \rangle_t} t -\sigma_\xi^2 \Big|^2 \bigg]  \,  (t+1)^{-\frac 1 2-\varepsilon} \, dt  \; \leq \; \cX(\om) \, < \; \infty, \qquad \text{if $d=3$},
  \end{align*}
  and 
   \begin{align*}
  \int_0^\infty   E_0^\om \bigg[\Big| \frac{\langle \xi\cdot M \rangle_t} t -\sigma_\xi^2 \Big|^2 \bigg] \,  (t+1)^{-\varepsilon} \, dt  \; \leq \; \cX(\om) \, < \; \infty, \qquad \text{if $d\geq 4$}.
  \end{align*}
\end{enumerate}
\end{prop}

\begin{proof}
(i) 
Recall that $g_\xi=\Gamma^{\om} (\psi_\xi(\om, \cdot))(0)$ and note that $\sigma_\xi^2= \mean[g_\xi]$ since $\Sigma^2=2\bomega_{\hom}$ (cf.\ \eqref{eq:hom_coeff} and \eqref{eq:def_omhom} above). Setting $G_\xi\ldef g_\xi - \mean[g_\xi]$, we have by Corollary~\ref{cor:qV} that
  \begin{align} \label{eq:relMV}
  \frac{\langle \xi \cdot M \rangle_t} t -\sigma_\xi^2
=  \frac 1 t \int_0^t G_\xi(\tau_{X_s}\om) \, ds.
\end{align}
Arguing as in \cite[Section 6]{Mo12} we get
\begin{align*}
 \bbE E_0^\om \left[ \Big(\int_0^t G_\xi(\tau_{X_s} \om) \, ds \Big)^2  \right]&= \;2 \int_{0\leq s \leq u \leq t }  \bbE E_0^\om \big[ G_\xi(\tau_{X_s} \om) \, G_\xi(\tau_{X_u} \om)\big] \, ds \, du \\
&= \; 2 \int_{0\leq s \leq u \leq t }  \bbE E_0^\om \big[ G_\xi(\om) \, G_\xi(\tau_{X_{u-s}} \om)\big] \, ds \, du \\
&= \; 2 \int_0^t (t-s)  \bbE  \big[ G_\xi(\om) \, P_s G_\xi( \om)\big] \, ds,
\end{align*}
where we used the stationarity of $(\tau_{X_s}\om)$ in the second step and a change of variable in the last step. Since $(\tau_{X_s}\om)$ is reversible w.r.t.\ $\mathbb{P}$, the semigroup operator $P_s $ is symmetric in $L^2(\bbP)$ and we obtain that
\begin{align*}
 \bbE E_0^\om \left[ \Big(\int_0^t G_\xi(\tau_{X_s} \om) \, ds \Big)^2  \right]\;=\; 2 \int_0^t (t-s)  \bbE  \left[ \big( P_{s/2} G_\xi( \om)\big)^2\right] \, ds 
\; \leq \; 2 t \int_0^t  \bbE  \left[ \big( P_{s/2} G_\xi( \om)\big)^2\right] \, ds.
\end{align*}
Now we apply Proposition~\ref{P:decay}, which gives for any $\varepsilon>0$,
\begin{align} \label{eq:L2bound_G}
  \bbE E_0^\om \left[ \Big(\int_0^t G_\xi(\tau_{X_s} \om) \, ds \Big)^2  \right] \; \leq  
   \begin{cases}
   \; c \, t^{\frac 3 2 + \varepsilon} & \text{if $d=3$,} \\
    \; c \, t^{1 + \varepsilon} & \text{if $d\geq 4$,}
  \end{cases}
\end{align}
for some constant $c=c(d,\rho,\varepsilon, M(p,q))$ provided $M(p,q)<\infty$ for suitable $p$ and $q$.
In view of \eqref{eq:relMV} this finishes the proof of (i).
Further, to show (ii), for a given $\varepsilon>0$ we use Fubini's theorem and apply (i) for any $\varepsilon'\in (0,\varepsilon)$ to obtain for $d\geq 4$,
\begin{align*}
& \mean \Bigg[  \int_0^\infty   E_0^\om \bigg[\Big| \frac{\langle \xi \cdot M \rangle_t} t -\sigma_\xi^2 \Big|^2 \bigg] \,  (t+1)^{-\varepsilon} \, dt  \Bigg]  \\
& \mspace{36mu} = \;
\int_0^\infty   \mean \Bigg[ E_0^\om \bigg[\Big| \frac{\langle \xi\cdot M \rangle_t} t -\sigma_\xi^2 \Big|^2 \bigg] \Bigg] \,  (t+1)^{-\varepsilon} \, dt \; \lesssim \; 
\int_0^\infty  (t+1)^{-(1+\varepsilon-\varepsilon')} \, dt \; < \; \infty,
\end{align*}
which implies (ii). In $d=3$ the result follows by a similar argument. 
\end{proof}

\begin{remark}
 Using the ergodicity of the environment process in \eqref{eq:relMV} one can apply  the theory of `fractional coboundaries' of Derriennic and Lin \cite{DL01} and deduce from the statement in Proposition~\ref{prop:estV} (i), e.g.\ if $d\geq 4$, that  for any $\varepsilon >0$ we have for $\bbP$-a.e.\ $\om$,
 \begin{align*}
 \lim_{t\to \infty} t^{\frac 1 2 - \varepsilon} \, \Big| \frac{\langle \xi\cdot M \rangle_t} t -\sigma_\xi^2 \Big| \; = \; 0, \qquad \text{$P_0^\om$-a.s.}
 \end{align*}
\end{remark}

\begin{prop} \label{prop:estJ}
Let $d\geq 3$, $n\in \bbN$ and suppose that Assumption \ref{ass:sg} holds. There exist $p=p(d,n)$ and $q=q(d,n)$ such that if $M(p,q)<\infty$ we have that for $\bbP$-a.e.\ $\om$,
\begin{align*}
 E_0^\om \Big[ \sum_{0\leq s \leq t} \big|\xi\cdot M_s - \xi\cdot M_{s-}\big|^n  \Big] \; \leq \; \cX \,t, \qquad \forall t>0,
\end{align*}
for some random $\cX=\cX(d,n)\in L^1(\prob)$.
\end{prop}
\begin{proof}
 Recall that for any function $f:\bbZ^d \times \bbZ^d \to \bbR$ that vanishes on the diagonal, the process
\[
\sum_{0\leq s\leq t} f(X_{s-},X_s)-\int_{(0,t]} \sum_{y\in \bbZ^d} \om(X_{s-},y) f(X_{s-},y) \, ds
\]
is a local $P_0^\om$-martingale for $\bbP$-a.e.\ $\om$. Then choosing 
$ f(x,y)= \big| \psi_\xi(y)-\psi_\xi(x) \big|^n$,
we obtain by the cocycle property (cf.\ Remark~\ref{rem:cocycle}) and the ergodic theorem that
\begin{align*}
& E_0^\om \Big[ \frac 1 t \sum_{0\leq s\leq t} \big| \xi\cdot M_s- \xi\cdot M_{s-} \big|^n \Big] 
\; = \; E^\om_0 \Big[ \frac 1 t \sum_{0\leq s\leq t} \big| \psi_\xi(\om,X_s)-  \psi_\xi(\om, X_{s-} )\big|^n  \Big] \\
& \mspace{36mu} = \; \frac 1 t \int_0^t ds \, E^\om_0\Big[  \sum_{y\in \bbZ^d} \om(X_{s-},y) \,
  \big|  \psi_\xi(\om,y)-  \psi_\xi(\om, X_{s-} )\big|^n \Big]   \\
&\mspace{36mu} = \;  \frac 1 t \int_0^t ds \, E^\om_0\Big[  
 \sum_{y\in \bbZ^d} (\tau_{X_{s-}}\om)(0, y-X_{s-})  \,
  \big|  \psi_\xi(\tau_{X_{s-}}\om, y- X_{s-})\big|^n   \Big]  \\
& \mspace{36mu}  \rightarrow \;  \bbE\Big[ \sum_{y\in \bbZ^d}  \om(0,y) \, \big|\psi_\xi(\om,y)\big|^n \Big] \;< \;\infty
\end{align*}
as $t$ tends to infinity $\prob$-a.s.\ and in $L^1(\Omega, \prob)$,  provided sufficiently high moments of $\psi_\xi$ (or the gradient of $\phi_\xi$, respectively) are finite, which can be ensured by Proposition~\ref{C:corrector}, \eqref{eq:momentboundphi}. Further, an application of the maximal ergodic theorem yields
\begin{align*}
\cX \; \ldef \; \sup_{t>0}  \frac 1 t \int_0^t ds \, E^\om_0\Big[  
 \sum_{y\in \bbZ^d} (\tau_{X_{s-}}\om)(0, y-X_{s-})  \,
  \big|  \psi_\xi(\tau_{X_{s-}}\om, y- X_{s-})\big|^n   \Big]  \in L^1(\prob),
\end{align*}
which finishes the proof.
\end{proof}

\begin{proof} [Proof of Proposition~\ref{prop:beM}]
 We shall apply the general result in Theorem~\ref{thm:be_hb} with the choice $n=2$ on the martingale
\begin{align*}
 N_s \; \ldef \; \frac{\xi \cdot M_{st}}{\sqrt t \, \sigma_\xi}, \qquad 0\leq s \leq 1.
\end{align*}
 Since
\begin{align*}
 \langle N \rangle_1-1 \;=\; \frac{\langle \xi\cdot M \rangle_t}{t \, \sigma_\xi^2}-1\; =\;\frac 1 {\sigma_\xi^2} \bigg( \frac{\langle \xi \cdot M \rangle_t}{t}- \sigma_\xi^2 \bigg),
\end{align*}
we get by Proposition~\ref{prop:estV} that for any $\varepsilon>0$,
\begin{align} \label{eq:estMtilde1}
\mean \Big[ E_0^\om \big[ | \langle N \rangle_1 -1|^{2} \big] \Big]
\; \lesssim \;\begin{cases}
   \;  t^{- \frac 1 2 + \varepsilon} & \text{if $d=3$,} \\
    \;  t^{-1 + \varepsilon} & \text{if $d\geq 4$.}
  \end{cases}
\end{align}
Moreover, for $\prob$-a.e.\ $\om$,
\begin{align} \label{eq:estMtilde2dim3}
 \int_0^\infty E_0^\om \Big[ \big| \langle N \rangle_1 -1\big|^{2} \Big] \,
 \big(t+1\big)^{-\frac 1 2 -\varepsilon} \, dt  \; & \leq \; \cX(\om) \; < \; \infty, \qquad \text{if $d=3$,} \\
 \intertext{and}  \label{eq:estMtilde2}
  \int_0^\infty E_0^\om \Big[ \big| \langle N \rangle_1 -1\big|^{2} \Big] \,
 \big(t+1\big)^{-\varepsilon} \, dt \; & \leq \; \cX(\om) \; < \; \infty, \qquad \text{if $d\geq 4$.} 
\end{align}
Furthermore,
\begin{align*}
 \sum_{s\leq 1} \big| \Delta N_s\big|^{4} \; = \; (t\sigma_\xi^2)^{-2} \sum_{s\leq t }  \big|\xi\cdot M_s - \xi\cdot M_{s-}\big|^{4},
\end{align*}
and we obtain from Proposition~\ref{prop:estJ} that
\begin{align} \label{eq:estMtilde3}
  E_0^\om \Big[ \sum_{0\leq s \leq 1} \big| \Delta N_s\big|^{4} \Big] \; \lesssim \; t^{-1}.
\end{align}
Now we apply Theorem~\ref{thm:be_hb}, first under the annealed measure $\prob_0$, which gives
\begin{align*}
 & \sup_{x \in \bbR} \left|  \prob_0\big[\xi \cdot M_t \leq  x \sqrt{t} \ \big] - \Phi(\tfrac x {\sigma_\xi})\right| \;= \;
\sup_{x \in \bbR} \left|  \prob_0\big[ N_1 \leq  x \ \big] - \Phi(x)\right| \\
& \mspace{36mu} \lesssim \; \Bigg( \mean \bigg[ E_0^\om \Big[ \big| \langle N \rangle_1 -1\big|^{2} \Big] \bigg]+  \mean \bigg[ E_0^\om \Big[ \sum_{0\leq s \leq 1} \big| \Delta N_s\big|^{4} \Big] \bigg]\Bigg)^{\frac 1 5},
\end{align*}
so that (i) follows from \eqref{eq:estMtilde1} and \eqref{eq:estMtilde3}. On the other hand, for $\prob$-a.e.\ $\om$ we apply Theorem~\ref{thm:be_hb} under $P_0^\om$ and obtain
\begin{align*}
 & \sup_{x \in \bbR} \Big|  P_0^\om\big[\xi \cdot M_t \leq  x \sqrt{t} \ \big] - \Phi(\tfrac x {\sigma_\xi})\Big|^5 \;= \;
\sup_{x \in \bbR} \Big|  P_0^\om\big[ N_1 \leq  x \ \big] - \Phi(x)\Big|^5 \\
& \mspace{36mu} \lesssim \;  E_0^\om \Big[ \big| \langle N \rangle_1 -1\big|^{2} \Big]+ E_0^\om \Big[ \sum_{0\leq s \leq 1} \big| \Delta N_s\big|^{4} \Big],
\end{align*}
 which implies (ii) by   \eqref{eq:estMtilde2dim3} or \eqref{eq:estMtilde2}, respectively,  and \eqref{eq:estMtilde3}.
\end{proof}

\subsection{Speed of convergence for the corrector}

\begin{prop} [Suboptimal estimate for the growth of corrector]\label{prop:conv_corr}
Let $d\geq 3$ and suppose that Assumptions~\ref{ass:sg} holds. For any $\delta\in (0,1)$ and any $n\in\bbN$ there exist $p,q\in [0,\infty)$ (only depending on $d$, $\delta$ and $n$) such that under the moment condition $M(p,q) <\infty$ the following holds. There exists  a random constant $\cX =\cX (d, p, \delta, n)$ satisfying $\mean[\cX^n] < \infty$ such that for $\prob$-a.e.\ $\om$,
\begin{align*}
E_0^\om\Big[ \big|\xi \cdot \chi(\om,X_t) \big| \Big] \; \leq \;  \cX(\om)  \, (t+1)^{\delta}, \qquad t\geq 0.
\end{align*}
\end{prop}
\begin{proof}
We denote by $d$ be the natural graph distance on $\bbZ^d$, i.e.\ $d(x, y)$ is the minimal length of a path between $x$ and $y$ and with a slight abuse of notation  we set $B(r) := \{y \in \bbZ^d \, | \, d(0, y) \leq r\}$. Fix exponents $\alpha\in(d,d+2\delta)$, $k\in\mathbb N$ such that $d+2d/k<\alpha$ and $\varepsilon\in (0,\delta)$. Hence, with $m$ denoting the weight of \eqref{def:m_alpha}, we get by appealing to Lemma~\ref{L:gradHK:aux1} the estimate
\begin{align*}
  \sum_{y \in \bbZ^d} p^\om(t,0,y) \, d(0,y)^{\frac dk} 
   & \leq \; \Big( \sum_{y \in \bbZ^d} p^\om(t,0,y)^2 \, m(t,y)^\alpha  \Big)^{1/2} \, \Big( \sum_{y \in \bbZ^d} d(0,y)^{\frac{2d}k} m(t,y)^{-\alpha}  \Big)^{1/2} \\
   &\leq \; \mathcal{X}(\om) \, (t+1)^{-\frac d 4 + \frac \varepsilon 2  +\frac \alpha 4}\;\leq \; \mathcal{X}(\om) \, (t+1)^{\delta},
\end{align*}
where $\cX$ denotes a random constant (with arbitrarily hight moments provided $M(p,q)<\infty$ for $p$ and $q$ sufficiently large). Consider $R_t:=d(0,X_t)$ and note that
\begin{equation*}
   E_0^\om \big[ R_t^{\frac d k} \big] \; = \; \sum_{y \in \bbZ^d} p^\om(t,0,y) \, d(0,y)^{\frac d k}\;\leq\;\mathcal{X}(\om) \, (t+1)^{\delta}.
 \end{equation*}
Now, let $(z_0, \ldots, z_{R_t})$ denote a nearest-neighbour path connecting $z_0=0$ and $z_{R_t}=X_t$.  Then, by the cocycle property (cf.\ Remark~\ref{rem:cocycle}),
\begin{align*}
 \big|\xi \cdot \chi(\om,X_t) \big|^k \; \leq \; \sum_{i=0}^{R_t-1}  \big|\xi \cdot \chi(\om,z_{i+1})- \xi\cdot \chi(\om,z_{i}) \big|^k \;=\; \sum_{i=0}^{R_t-1} \big|\xi \cdot  \chi(\tau_{z_i} \om, z_{i+1}-z_i) \big|^k.
\end{align*}
Hence, with $H(\om):=\sum_{|e|=1}  \big| \xi \cdot \chi(\om,e) \big|^k$ and $H^*(\omega):=\sup_{R>0}\frac 1 {\# B(R)} \sum_{y \in B(R)} H(\tau_y \om)$ (the associated maximal function), we get
\begin{align*}
   \big|\xi \cdot \chi(\om,X_t) \big| \; \lesssim \;   R_t^{\frac d k}  \bigg( \frac 1 {\# B(R_t)} \sum_{y \in B(R_t)} H(\tau_y \om) \bigg)^{\frac 1 k}\leq R_t^{\frac d k}(H^*)^\frac1k\leq \mathcal X(\omega)(H^*)^\frac1k(t+1)^\delta.
\end{align*}
Note that $\mean[H^{\frac{2n}{k}}] <\infty$ for any $2n>k$ provided $M(p,q)<\infty$ for $p$ and $q$ sufficiently large (depending on $n$ and $k$), see Proposition~\ref{C:corrector} (b). Hence, by the maximal ergodic theorem (cf.\ e.g.\ \cite[ Theorem~6.3 in Chapter~1 and Chapter~6]{Kr85}), we have $\mean\left[(H^*)^{\frac{2n} k}\right]\lesssim\mean\Big[H^{\frac{2n} k} \Big] < \infty$, and thus the claimed statement follows.
 \end{proof}

\subsection{Proof of Theorem \ref{thm:main}}

Let  $\delta \in (0, \frac 1  {10})$ be arbitrary and let $p$ and $q$ be such that the statements in Propositions~\ref{prop:beM} and \ref{prop:conv_corr} hold. Recall that $X_t=M_t-\chi(\om,X_t)$, $t\geq 0$, for $\bbP$-a.e.\ $\om$. Hence, for any $x\in \bbR$,
\begin{align*}
   P_0^\om\big[\xi \cdot X_t \leq  x \sqrt{t}  \big]\; \leq \; P_0^\om\big[\xi \cdot M_t \leq  (x- t^{-\frac 1 5}) \sqrt{t}  \big]+ P_0^\om\big[|\xi \cdot \chi(\om,X_t)| > t^{\frac 3 {10} } \big].
\end{align*}
Further,  Proposition~\ref{prop:conv_corr}  gives that
\begin{align*}
E_0^\om \big[|\xi \cdot \chi(\om,X_t)|\big] \; \lesssim \; \cX(\om) \, (t+1)^\delta,
\end{align*}
and an application of \v{C}eby\v{s}ev's inequality yields
\begin{align*}
P_0^\om\big[|\xi \cdot \chi(\om,X_t)| > t^{\frac 3 {10}} \big] \; \lesssim \;  \cX(\om) \, (t+1)^{-\frac 1 5}.
\end{align*}
Recall that $\Phi$ denotes the distribution function of the standard normal distribution. Since $\Phi'$ is bounded by $1$, we have
\begin{align*}
 \sup_{x \in \bbR} \Big|  \Phi(\tfrac {x+t^{-\frac 1 5}} {\sigma_\xi})  -\Phi(\tfrac x {\sigma_\xi})\Big| \; \leq \; \sigma_\xi^{-1} \, t^{-\frac 1 5 },
\end{align*}
and we get
\begin{align*}
   P_0^\om\big[\xi \cdot X_t \leq  x \sqrt{t}  \big] - \Phi(\tfrac x {\sigma_\xi}) 
 \;  \leq \; &
   \Big | P_0^\om\big[\xi \cdot M_t \leq  (x+ t^{-\frac 1 5}) \sqrt{t}  -  \Phi(\tfrac {x+t^{-\frac 1 5}} {\sigma_\xi})\big] \Big| \\
   & +    \Big|  \Phi(\tfrac {x+t^{-\frac 1 5}} {\sigma_\xi})  -\Phi(\tfrac x {\sigma_\xi})\Big|+
    P_0^\om\big[|\xi \cdot \chi(\om,X_t)| > t^{\frac 3 {10}} \big] \\
    \lesssim \; & \sup_{y\in \bbR} \Big| P_0^\om\big[\xi \cdot M_t \leq  y  \sqrt{t} \big]  -  \Phi(\tfrac {y} {\sigma_\xi}) \Big| + \cX(\om) \, (t+1)^{-\frac 1 5}.
   \end{align*}
On the other hand, since
\begin{align*}
   P_0^\om\big[\xi \cdot M_t  \leq  (x- t^{-\frac 1 5}) \sqrt{t}  \big] \; \leq \; P_0^\om\big[\xi \cdot X_t \leq  x \sqrt{t}  \big]+ P_0^\om\big[|\xi \cdot \chi(\om,X_t)| > t^{\frac 3 {10}} \big],
\end{align*}
we can derive a similar lower bound by using the same arguments. Thus,
\begin{align*}
\Big| P_0^\om\big[\xi \cdot X_t \leq  x \sqrt{t}  \big] - \Phi(\tfrac x {\sigma_\xi}) \Big|
 \; \lesssim \; \sup_{y\in \bbR} \Big | P_0^\om\big[\xi \cdot M_t \leq  y  \sqrt{t} \big]  -  \Phi(\tfrac {y} {\sigma_\xi}) \Big| + \cX(\om)  (t+1)^{-\frac 1 5}
\end{align*}
and the same estimate holds if we replace $P_0^\om$ by $\prob_0$.
Hence, the claim follows from Proposition~\ref{prop:beM}.
\qed

\section{Examples} \label{sec:examples}
In this section we discuss a class of environments satisfying Assumption~\ref{ass:sg} which are related to the Ginzburg Landau $\nabla \phi$-interface model (see \cite{Fu05}). This is a well known model for an interface separating two pure thermodynamical phases. We first explain a slightly more general construction and then revisit that specific class of environments at the end of the discussion. Our starting point is a shift-invariant probability measure $\tilde\mu$ on $\tilde\Omega:=\bbR^{\bbZ^d}$.  We suppose that for any $\tilde u\in C^1(\tilde\Omega,\bbR)$ the measure $\tilde\mu$ satisfies the  Brascamp-Lieb inequality
\begin{equation}\label{example:BL}
  \var_{\tilde\mu}(\tilde u)\leq\frac{1}{\rho}\mean_{\tilde\mu}\Big[\sum_{x\in\bbZ^d}\tilde\partial\tilde u(x) \, \big(G*\tilde\partial\tilde u\big)(x)\Big],
\end{equation}
where $G$ denotes the Green's function associated with the discrete Laplacian on $\bbZ^d$, and $\tilde\partial\tilde u$ denotes the $\ell^2(\bbZ^d)$-gradient of $\tilde u$, which for sufficiently smooth $\tilde u$ is characterised by 
\begin{equation*}
  \tilde\partial\tilde u(\tilde\omega,x)=\lim\limits_{h\to 0}\frac{\tilde u(\tilde\omega+h\delta_x)-\tilde u(\tilde\omega)}{h}\qquad \tilde\omega\in\tilde\Omega,\,  x\in\bbZ^d,
\end{equation*}
with Dirac function $\delta_x:\bbZ^d\to\{0,1\}$. Further, let $\bar\mu$   be the probability measure on $\bar\Omega:=\bbR^{E_d}$  defined as the pushforward of $\tilde\mu$ under the transformation
\begin{equation*}
  T:\tilde\Omega\to\bar\Omega,\qquad \tilde\omega\mapsto\nabla\tilde\omega,
\end{equation*}
(see Section~\ref{sec:notation} for the definition of the discrete gradient $\nabla$). We denote by $\partial\bar u$ the $\ell^2(E_d)$-gradient of $\bar u$, which for sufficiently smooth $\bar u$ is characterised by 
\begin{equation*}
  \partial\bar u(\bar\omega,e)=\partial_e\bar u(\bar\omega)=\lim\limits_{h\to 0}\frac{\bar u(\bar\omega+h\delta_e)-\bar u(\bar\omega)}{h}\qquad \bar\omega\in\bar\Omega, \, e\in E_d,
\end{equation*}
with Dirac function $\delta_e:E_d\to\{0,1\}$. It turns out that for any $\bar u\in C^1(\bar\Omega,\bbR)$ we have the spectral gap estimate
\begin{equation} \label{eq:sgbar}
  \var_{\bar\mu}(\bar u)\;\leq \; \frac{1}{\rho}\mean_{\bar\mu}\Big[\sum_{e\in E_d}|\partial_e\bar u|^2\Big].
\end{equation}
This can be seen as follows. Setting $\tilde u(\tilde\omega):=\bar u(T\bar\omega)$ we get from the product rule the relation
\begin{equation*}
  \tilde\partial\tilde u(x)=\sum_{e\in E_d\atop \oe=x}\partial\bar u(e)-\sum_{e\in E_d\atop \ue=x}\partial\bar u(e)=\nabla^*\partial\bar u(x),
\end{equation*}
where $\nabla^*$ denotes the (negative) discrete divergence operator, see Section~\ref{sec:notation}. Hence, 
\begin{align*}
  \sum_{x\in\bbZ^d}\tilde\partial\tilde u(x) \, (G*\tilde\partial\tilde u)(x)
\; = \;
  \sum_{x,y\in\bbZ^d}\nabla^*\partial\bar u(x) \, 
  G(x-y) \, \nabla^*\partial\bar u(y)
\; = \;
  \sum_{e\in E_d}\partial\bar u(e)\, \nabla v(e),
\end{align*}
where $v:\bbZ^d\to\bbR$ denotes the convolution
\begin{equation*}
  v(x):=\sum_{y\in\bbZ^d} G(x-y)\, \nabla^*\partial\bar u(y).
\end{equation*}
Recall that $G$ denotes the discrete Green function for $\nabla^*\nabla$. In particular, $v$ solves $\nabla^*\nabla v=\nabla^*\partial\bar u$,  and a standard energy estimate yields  $\sum_{e\in E_d}|\nabla v|^2\leq \sum_{e\in E_d}|\partial\bar u(e)|^2$. We conclude that
\begin{eqnarray*}
  \sum_{x\in\bbZ^d}\tilde\partial\tilde u(x)(G*\tilde\partial\tilde u)(x) \;\leq \; 
  \sum_{e\in E_d}|\partial\bar u(e)|^2,
\end{eqnarray*}
which combined with \eqref{example:BL} yields the spectral gap estimate in \eqref{eq:sgbar}. 

Finally, we define the probability measure $\prob$ on $\Omega=(0,\infty)^{E_d}$ as the pushforward of $\bar\mu$ under the (nonlinear) transformation
\begin{align*}
  \Lambda:  \bar\Omega\rightarrow \Omega, \quad \Lambda(\bar\omega)(e):=\lambda(\bar\omega(e)),
\end{align*}
with $\lambda:\bbR\to(0,\infty)$ denoting a Lipschitz function with global Lipschitz constant $c_\lambda>0$. With any $u\in L^2(\Omega,\prob)$ we may associate $\bar u\in L^2(\bar\Omega,\bar\mu)$ via $\bar u:=u\circ\Lambda$. Then the chain rule yields
\begin{equation*}
  \big|\partial_e\bar u\big|\; = \; \big|\partial_e (u\circ\Lambda)\big|\; \leq \;  c_\lambda \big(|\partial_e u|\circ\Lambda\big),
\end{equation*}
and we thus obtain the spectral gap estimate  in Assumption~\ref{ass:sg} in form of
\begin{align*}
  & \mean\big[(u-\mean[u])^2\big]\leq \frac{c^2_\lambda}{\rho}\sum_{e\in E_d} \mean\Big[ |\partial_{e}u|^2  \Big].
\end{align*}

Now we explain the link to the Ginzburg Landau $\nabla\phi$ interface model. In $d\geq 3$ consider an interface  described by a collection of random height variables $\varphi\in \tilde \Omega$ sampled from a Gibbs measure $\mu$ formally given by
\begin{align*}
  \mu(d\varphi)= \frac 1 Z \exp(-H(\varphi)) \, \prod_{x\in \bbZ^d} d \varphi(x),
\end{align*}
with formal Hamiltonian
\begin{align*}
H(\varphi)= \sum_{e\in E_d} V(\nabla \varphi(e)),
\end{align*}
and potential function $V\in  C^2(\bbR; \bbR_+)$, which we suppose to be even and strictly convex with $c_- \leq  V'' \leq c_+$ for some $0 < c_- \leq  c_+ < \infty$. Note that in the special case $V(x)=\frac 1 2 x^2$ the field  $\phi=\{\phi(x); x\in \bbZ^d \}$ becomes a \emph{discrete Gaussian free field}.  For more details on the rigorous definition, which is based on taking the thermodynamical limit of Gibbs measures on finite volume approximations of the infinite lattice $\bbZ^d$, see
\cite[Section~4.5]{Fu05}. 
Then, thanks to the strict convexity we have the Brascamp-Lieb inequality (see \cite{BL76}, cf.\ also \cite{NS97})
\begin{align*}
  \var_\mu(F) & \leq c_-^{-1}  \mean_\mu \Big[ \sum_{x\in\bbZ^d}\tilde\partial F(x) \, \big(G * \tilde\partial F\big)(x)\Big], \qquad F\in C^1(\tilde\Omega).
\end{align*}
In particular, \eqref{example:BL} holds for $\mu$ and the above considerations show that an environment with random conductances of the form $\{\om(e)=\lambda(\nabla \phi(e)), e\in E_d\}$ for any positive, even, globally Lipschitz  function $\lambda\in C^1(\bbR)$ satisfies  Assumption~\ref{ass:sg}. 

As a further consequence from the Brascamp-Lieb inequality it is known that exponential moments for $\nabla \phi(e)$ exist (cf.\ \cite{Fu05, NS97}). Thus, the environment  $\{\om(e)=\lambda(\nabla \phi(e)), e\in E_d\}$ with $\lambda$ as above also satisfies $M(p,q)<\infty$ for all $p,q\in [1,\infty)$.

\subsubsection*{Acknowledgment}
We thank Jean-Dominique Deuschel, Jean-Christophe Mourrat and Martin Slowik for useful discussions and valuable comments. 
SA was partially supported by the German Research Foundation in the Collaborative Research Center 1060 ``The Mathematics of Emergent Effects'', Bonn.
SN was supported by the DFG in the context of TU Dresden’s Institutional Strategy ``The Synergetic University''.

\appendix

\section{The uniformly elliptic case - Proof of Theorem~\ref{thm:mainunif}}
\label{app:mainunif}

In this section we discuss the uniformly elliptic case and establish the improved rates in the Berry-Esseen Theorem as stated in Theorem~\ref{thm:mainunif}. We first recall a deterministic heat kernel estimate, see e.g.\ \cite{GNO15} for a self-contained proof.
\begin{lemma}\label{app:L1}
  Let $d \geq 2$ and assume that uniform ellipticity holds, i.e.~$M(p,q)<\infty$ for $p=q=\infty$. Then, for all $\alpha\geq 0$ there exists a constant $c=C(d,M(\infty,\infty),\alpha)$ such that for all $t\geq 0$ we have
  \begin{align*}
  p(t,y)  & \;\leq \;  c \, (t+1)^{-\frac{d}{2}} \, m(t,y)^{-2\alpha},\\
    \bigg(\sum_{y\in \bbZ^d} m(t,y)^{2\alpha}\, \big|\nabla p(t,y)\big|^2 \bigg)^{\!\frac{1}{2}} & \; \leq \; c\, (t+1)^{-(\frac{d}{4}+\frac{1}{2})}.
  \end{align*}
\end{lemma}
In the uniformly elliptic case and under the assumption of a spectral gap, in \cite{GO11,GO12,GNO15,BMN17} moment bounds on the corrector have been obtained. The following lemma additionally states that in $d=2$ the extended corrector grows logarithmically. For a proof see  \cite[Theorem~4.8]{Nlecture} (see also \cite{GNOreg} where the case of a continuum system is treated).
\begin{lemma}[Bounds on the corrector]\label{app:L2}
  Let $d \geq 2$ and suppose that Assumption~\ref{ass:sg} and uniform ellipticity hold, i.e.\ $M(p,q)<\infty$ for $p=q=\infty$. Then the extended correctors $(\phi_i,\sigma_i)$ exist in the sense of Proposition~\ref{C:corrector} (a) and for all $n\in\mathbb N$ and all $x\in\mathbb Z^d$ we have
  \begin{align*}
    \mean\Big[
    \big|\phi_i(x)\big|^{2n}+\big|\sigma_i(x)\big|^{2n}\Big]^{\frac{1}{2n}}
    & \; \leq \; c
    \begin{cases}
      \log^\frac12(|x|+1)& \text{if $d=2$,}\\
      1& \text{if $d\geq 3$},
    \end{cases}\\
    \mean\Big[
    \big|\nabla\phi_i\big|^{2n}+\big|\nabla\sigma_i\big|^{2n}\Big]^{\frac{1}{2n}}
    \; & \leq \; c
  \end{align*}
  with constant $c=c(d,\rho,M(\infty,\infty),n)$.
\end{lemma}
For any direction $\xi\in\bbR^d$ with $|\xi|=1$, we define the associated corrector $\phi_\xi$, the harmonic coordinate $\psi_\xi$, the op\'erateur carr\'e du champ $\Gamma^\omega$, and $g_\xi$ as in Proposition~\ref{P:decay}, which we can now refine in the uniformly elliptic case.
\begin{lemma}\label{app:L3}
  Let $d \geq 2$. Suppose that Assumption~\ref{ass:sg} and uniform ellipticity hold, i.e.\ $M(p,q)<\infty$ for $p=q=\infty$. 
  Then there exists a constant
  $c=c(d,\rho, M(\infty,\infty))$ such that for all $t\geq 0$,
  \begin{align*}
   \mean\left[\Big(P_t \big(g_\xi-\mean[g_\xi]\big)\Big)^{\!2}\right]^{\frac12} \; \leq \; 
    c\, 
    \begin{cases}
      \Big(\frac{\log(t+1)}{t+1}\Big)^{\!\frac{1}{2}}& \text{if $d=2$},\\
      (t+1)^{-\frac{1}{2}}& \text{if $d\geq 3$}.
    \end{cases}
  \end{align*}
\end{lemma}
\begin{proof}
  As in the proof of Proposition~\ref{P:decay} we deduce that $ g_\xi(\omega)-\mean[g_\xi]=\nabla^*H(\omega,0)$,
  where $H$ is defined in Lemma~\ref{L:QV}. Thus,
  \begin{equation*}
    I=\mean\bigg[\Big(P_t \big(g_\xi-\mean[g_\xi]\big)\Big)^{\!2}\bigg]^{\frac12}=\mean\bigg[\Big(\sum_{y\in\bbZ^d} p(t,0,y) \,\nabla^*H(y)\Big)^{\!2}\bigg]^{\frac12}.
  \end{equation*}
  An integration by parts (which is applicable thanks to Lemma~\ref{app:L2}) and an application of Lemma~\ref{app:L1} yield (for $\alpha=\frac d2+1$)
  \begin{eqnarray*}
    I&=&\mean\bigg[\Big(\sum_{y\in\bbZ^d}\nabla p(t,y)\cdot H(y)\Big)^{\!2}\bigg]^{\frac12}\\
    &\leq&\mean\bigg[\Big(\sum_{y\in\bbZ^d} \big|\nabla p(t,y)\big|^2 \, m(t,y)^{\alpha}\Big) \, \Big(\sum_{y\in\bbZ^d} \big|H(y)\big|^2 \, m(t,y)^{-\alpha}\Big)\bigg]^{\frac12}\\
    &\lesssim &(t+1)^{-(\frac{d}{4}+\frac{1}{2})} \, \mean\bigg[\Big(\sum_{y\in\bbZ^d} \big|H(y)\big|^2 m(t,y)^{-\alpha}\Big)\bigg]^{\frac12}.
  \end{eqnarray*}
  For $d\geq3$ the random field $H$ is stationary and has finite second moments, cf.~Lemma~\ref{app:L2}. We thus get the claimed estimate $I\leq c(t+1)^{-\frac12}$. For $d=2$, we infer from the definition of $H$ and the moment bounds in Lemma~\ref{app:L2} that $\mean\Big[|H(y)|^2\Big]    \leq c\log(|y|+1)$, and thus the claimed  estimate follows as well.
\end{proof}
From Lemma~\ref{app:L3} we obtain the following refinement of Proposition~\ref{prop:beM} (i).
\begin{prop} \label{app:L4}
  Let $d \geq 2$. Suppose that Assumption \ref{ass:sg} and uniform ellipticity hold, i.e.\ $M(p,q)<\infty$ for $p=q=\infty$. Then there exists a constant $c=c(d,\rho, M(\infty,\infty))$ such that for all $t\geq 0$, 
  \begin{align*}
    \sup_{x \in \bbR} \Big|  \prob_0 \big[\xi \cdot M_t \leq \sigma_\xi x \sqrt{t} \ \big] - \Phi(x)\Big|  \; \leq \;
    \begin{cases}
      \; c \, \big(\frac{\log(t+1)}{t+1}\big)^{\frac 1 {5}} & \text{if $d=2$,} \\
      \; c \, (t+1)^{-\frac 1 5} & \text{if $d\geq 3$.}
    \end{cases}
  \end{align*}
\end{prop}
The proof is the same as the one for Proposition~\ref{prop:beM}.
The only difference is that we appeal to Lemma~\ref{app:L3} to improve estimate \eqref{eq:L2bound_G}.
With Lemma~\ref{app:L1} and Lemma~\ref{app:L2} at hand, we also obtain the following refinement of Proposition~\ref{prop:conv_corr}:
\begin{lemma}\label{app:L5}
  Let $d \geq 2$, suppose Assumption~\ref{ass:sg} and uniform ellipticity, i.e.~$M(p,q)<\infty$ for $p=q=\infty$. Then there exists a random variable $\mathcal X$ such that for all $t\geq 0$ and $n\in\mathbb N$,
  \begin{align*}
    \mean\Big[E_0^\om\Big[ \big|\xi \cdot \chi(\om,X_t) \big| \Big]^n\Big]^\frac1n \; \leq \;  c\begin{cases}
      \log^\frac12(t+1)& \text{if $d=2$},\\
      1& \text{if $d\geq 3$},
    \end{cases}
  \end{align*}
  where $c=c(d,\rho, M(\infty,\infty),n)$.
\end{lemma}
\begin{proof}
  We only discuss the case $d=2$ since the argument for $d\geq 3$ is similar but simpler. W.l.o.g.\ let $\xi=e_i$. By Lemma~\ref{app:L1} we have
  \begin{equation*}
    E_0^\om\Big[ \big|\xi \cdot \chi(\om,X_t) \big| \Big]\; \lesssim \;\sum_{x\in\mathbb Z^d} p^\omega(t,0,x) \, \big|\phi_i(\omega,x)\big| \; \lesssim \; (t+1)^{-\frac{d}{2}}\sum_{x\in\mathbb Z^d}m(t,x)^{-2\alpha}\, \big| \phi_i(\omega,x)\big|,
  \end{equation*}
  and thus we get by Lemma~\ref{app:L2} (with $\alpha=\frac d2+1$) for any $n\in\mathbb N$,
  \begin{equation*}
    \mean\Big[E_0^\om\Big[ \big|\xi \cdot \chi(\om,X_t) \big| \Big]^n\Big]^{\frac1n} \; \lesssim \; (t+1)^{-\frac{d}{2}}\sum_{x\in\mathbb Z^d}m(t,x)^{-2\alpha}\log^\frac12(|x|+1) \, \lesssim \, \log^\frac12 (t+1),
  \end{equation*}
  which is the claim.
\end{proof}
With these estimates at hand,  Theorem~\ref{thm:mainunif} follows by the same argument as in the proof of Theorem~\ref{thm:main}.

\bibliographystyle{abbrv}
\bibliography{literature}

\end{document}